\documentclass[11pt]{amsart}
\usepackage[T1]{fontenc}
\usepackage[latin1]{inputenc}
\usepackage{typearea}
\usepackage{geometry}
\usepackage{ulem}
\usepackage{amsmath}
\usepackage{amssymb}
\usepackage{latexsym}
\usepackage{enumerate}
\usepackage{amsthm}
\usepackage[all]{xy}
\usepackage{hhline}
\usepackage{epsf} 
\usepackage{cite}
\usepackage{verbatim}
\usepackage{mathtools}
\usepackage{comment}
\usepackage{dsfont}
\usepackage{mathrsfs}
\usepackage{hyperref}

\usepackage{color}

\newtheorem{theorem}{{\sc Theorem}}[section]
\newtheorem{cor}[theorem]{{\sc Corollary}}

\newtheorem{lemma}[theorem]{{\sc Lemma}}
\newtheorem{prop}[theorem]{{\sc Proposition}}

\theoremstyle{remark}
\newtheorem{remark}[theorem]{{\sc Remark}}

\theoremstyle{definition}

\newcommand{\R}{\mathbb{R} }
\newcommand{\N}{\mathbb{N} }
\newcommand{\X}{\mathbb{X}}

\newcommand{\A}{\mathcal{A}}
\newcommand{\B}{\mathcal{B}}
\newcommand{\F}{\mathcal{F}}
\newcommand{\G}{\mathcal{G}}

\newcommand{\W}{\mathcal{W}}
\newcommand{\K}{\mathcal{K}}
\newcommand{\Z}{\mathbb{Z}}

\newcommand{\im}{\textnormal{im}}

\newcommand{\Prob}{\mathbb{P}}
\newcommand{\E}{\mathbb{E}}
\newcommand{\la}{\lambda}

\newcommand{\Pot}{\mathcal{P}}
\newcommand{\Om}{\Omega}

\providecommand{\abs}[1]{\lvert #1\rvert}
\providecommand{\babs}[1]{\bigl\lvert #1\bigr\rvert}
\providecommand{\Babs}[1]{\Bigl\lvert #1\Bigr\rvert}
\providecommand{\Bbabs}[1]{\biggl\lvert #1\biggr\rvert}
\providecommand{\BBabs}[1]{\Biggl\lvert #1\Biggr\rvert}
\providecommand{\norm}[1]{\lVert #1\rVert}
\providecommand{\fnorm}[1]{\lVert #1\rVert_\infty}

\providecommand{\Enorm}[1]{\lVert #1\rVert_2}
\providecommand{\BEnorm}[1]{\Bigl\lVert #1\Bigr\rVert_2}

\DeclareMathOperator{\Var}{Var}

\DeclareMathOperator{\dom}{dom}
\DeclareMathOperator{\Cov}{Cov}

\DeclareMathOperator{\Lip}{Lip}

\DeclareMathOperator{\Id}{Id}

\DeclareMathOperator{\Inf}{Inf}

\renewcommand{\phi}{\varphi}
\renewcommand{\epsilon}{\varepsilon}
\newcommand{\eps}{\varepsilon}
\renewcommand{\rho}{\varrho}
\renewcommand{\P}{\Prob}
\renewcommand{\1}{\mathds{1}}

\begin{document}
\title[Normal approximation on product spaces]{New bounds for normal approximation on product spaces with applications to monochromatic edges, random sums and an infinite de Jong CLT}
\author{Christian D\"obler}
\thanks{\noindent Mathematisches Institut der Heinrich Heine Universit\"{a}t D\"usseldorf\\
Email: christian.doebler@hhu.de }
\begin{abstract}  
We extend the Malliavin theory for $L^2$-functionals on product probability spaces that has recently been developed by Decreusefond and Halconruy (2019) and by Duerinckx (2021), by  characterizing the domains and investigating the actions of the three Malliavin operators in terms of the infinite Hoeffding decomposition in $L^2$, which we identify as the natural analogue of the famous Wiener-It\^{o} chaos decomposition on Gaussian and Poisson spaces. 
We further combine this theory with Stein's method for normal approximation in order to provide three different types of abstract Berry-Esseen and Wasserstein bounds: a) Malliavin-Stein bounds involving the Malliavin gradient $D$ and the pseudo-inverse of the Ornstein-Uhlenbeck generator $L$, b) bounds featuring the carr\'{e}-du-champ operator $\Gamma$ and c) bounds making use of a Clark-Ocone type integration-by-parts formula. To demonstrate the flexibility of these abstract bounds, we derive quantitative central limit theorems for the number of monochromatic edges in a uniform random coloring of a graph sequence as well as for random sums and prove an infinite version of the quantitative de Jong CLT that has recently been proved by G. Peccati and the author (2017) and by the author (2023). As a further theoretical application, we deduce new abstract Berry-Esseen and Wasserstein bounds for functionals of a general independent Rademacher sequence.

\end{abstract}

\maketitle

\section{Introduction}\label{intro}
\subsection{Overview}
In the recent article \cite{DH} by Decreusefond and Halconruy, a version of Malliavin calculus for functionals on a product of countably many probability spaces has been developed. In particular, suitable versions of the \textit{Malliavin derivative} $D$, its adjoint $\delta$, named \textit{divergence operator}, and the corresponding \textit{Ornstein-Uhlenbeck generator} or \textit{number operator} $L$ satisfying the crucial identity $L=-\delta D$ have been introduced. Important properties of the operators $D$ and $L$ have also been studied independently by Duerinckx \cite{Duer}.

As has been stated in \cite{DH}, nowadays, one main motivation for developing such a Malliavin structure on product spaces is to combine it with Stein's method of distributional approximation to derive suitable so-called \textit{Malliavin-Stein bounds} and, consequently, versions of such bounds in the context of normal and gamma approximation have been given in \cite{DH}. Historically, such Malliavin-Stein bounds have first been proved for functionals of Gaussian processes in the seminal paper \cite{NouPec09} by Nourdin and Peccati and, then, afterwards also for Poisson functionals \cite{PSTU, Schu16} and for functionals of a Rademacher sequence \cite{NPR, KRT1, KRT2}. We refer to the monographs \cite{NouPecbook} and \cite{RePe-book} for comprehensive introductions to the \textit{Malliavin-Stein method} on Gaussian and Poisson spaces, respectively. 

In the years since the appearance of \cite{NouPec09}, the research field initiated by that paper has been extremely active, producing a vast amount of papers dealing with both new theoretical developments and relevant applications of the Malliavin-Stein method. In fact, one reason for its popularity and fame is that it has been successfully applied to important and difficult problems from fields as diverse as random geometry, telecommunication, random graphs, machine learning and mathematical statistics.    
We refer to the constantly updated website \cite{Nweb} for a comprehensive list of research articles related to this line of research.

The purpose of the present paper, on the one hand, is to further extend and complement the Malliavin theory for product probability spaces as developed in \cite{DH, Duer} and to combine this theory with Stein's method in order to prove new abstract error bounds for the normal approximation of real-valued functionals defined on such spaces. On the other hand, we emphasize the power and flexibility of our approach by deriving quantitative centrals limit theorems (CLTs) in three different applications that can be fitted into our framework. \smallskip\\

More precisely, the main contributions of this work include
\begin{enumerate}[1)]
\item new characterizations of the domains of the three Malliavin operators $D,\delta$ and $L$ in terms of the infinite Hoeffding decomposition reviewed in Subsection \ref{Hoeffding} and new formulae for the action of the operators $D$ and $L$ in terms of this decomposition.
\item a new Stroock type formula to recover the infinite Hoeffding decomposition of a functional by computing iterated derivatives.
\item new functional analytic properties of the Ornstein-Uhlenbeck operator $L$ and its pseudo-inverse $L^{-1}$.
\item the introduction of a carr\'{e}-du-champ operator on general product spaces including an alternative representation that allows effectively to assess its non-diffusiveness and the proof of a corresponding integration-by-parts formula.
\item three different types of new abstract bounds for the normal approximation of functionals on product probability spaces assessed in both the Wasserstein and Kolmogorov distances (see below for definitions) that are expressed in terms of the Malliavin operators, the carr\'{e}-du-champ operator and the Clark-Ocone covariance formula, respectively.
\item a new effective chain rule formula for the Malliavin gradient $D$.
\item an infinite version of the quantitative de Jong CLT by the author and G. Peccati \cite{DP17} and by the author \cite{D23}.
\item a quantitative CLT for random sums.
\item a new quantitative CLT of the correct order for the number of monochromatic edges in graph sequences colored uniformly at random, both for the situation of a fixed number of colors and for the number of colors diverging to infinity.
\item new Wasserstein and Berry-Esseen bounds for functionals of a general Rademacher sequence.
\end{enumerate} 

We will express our quantitative CLTs in terms of the following two prominent distributional distances. Recall that, for two real-valued random variables $Y,W\in L^1(\P)$, the \textit{Wasserstein distance} between (the distributions of) $Y$ and $W$ may be defined by 
\begin{align*}
 d_\W(Y,W)&:=\sup_{h\in\Lip(1)}\babs{\E[h(Y)]-\E[h(W)]},
\end{align*}
where $\Lip(1)$ denotes the collection of all Lipschitz-continuous functions on $\R$ with Lipschitz constant $1$. On the other hand, the \textit{Kolmogorov-distance} between (the distributions of) $Y$ and $W$ is defined as the supremum norm distance between their respective distribution functions, i.e. by 
\begin{equation*}
d_\K(Y,W):=\sup_{z\in\R}\babs{\P(Y\leq z)-\P(W\leq z)}.
\end{equation*}
It is well-known that convergence with respect to any of these metrics implies weak convergence of probability measures on $\R$.\smallskip\\

In this work, we will always let $Z\sim N(0,1)$ denote a standard normal random variable. Although one has the general inequality $d_\K(W,Z)\leq\sqrt{d_\W(W,Z)}$ in the context of normal approximation, this estimate is rarely ever sharp. In fact, in most examples the true rate of convergence is the same for both distances. Hence, we provide both, bounds on the Wasserstein distance and also on the Kolmogorov distance that are precise enough to yield estimates on the rate of convergence of the same order in several applications. As a general rule, proving and also applying bounds on the Kolmogorov distance, so-called \textit{Berry-Esseen bounds}, is more involved than bounding the Wasserstein distance.\smallskip\\

We next present concrete new error bounds for three different kinds of applications. The proofs of the respective bounds are provided in Section \ref{proofs}.

\subsection{An infinite quantitative de Jong theorem} \label{dejong}
Suppose that $X_k$, $k\in\N$, are independent random variables defined on a probability space $(\Om,\F,\P)$, where $X_k$ assumes values in an arbitrary measurable space $(E_k,\B_k)$, $k\in\N$. Denote by 
$(E,\B):=(\prod_{k\in\N} E_k,\bigotimes_{k\in\N} \B_k)$ the product of these spaces and suppose that $f:(E,\B)\rightarrow (\R,\B(\R))$ is a measurable function such that $F:=f(\X)=f\circ\X\in L^2(\P)$, where $\X:=(X_k)_{k\in\N}$ and $\B(\R)$ denotes the $\sigma$-field of Borel subsets of $\R$. In Subsection \ref{Hoeffding} below it is explained that every such random variable $F$ has an orthogonal decomposition of the form 
\begin{align}\label{hdec}
F=\sum_{M\subseteq\N: |M|<\infty}F_M=\sum_{p=0}^\infty F^{(p)},
\end{align}  
where $F^{(0)}=F_\emptyset=\E[F]$ $\P$-a.s., both infinite series in \eqref{hdec} converge unconditionally in $L^2(\P)$ (see Section \ref{appendix} for a review of this concept) and, moreover, the respective summands $F_M\in L^2(\P)$, $M\subseteq\N, |M|<\infty$, and $F^{(p)}$, $p\in\N_0$, are uncorrelated. In particular, one has $\E[F_M]=0$ for $M\not=\emptyset$ and $\E[F^{(p)}]=0$ for $p\geq1$ so that the decomposition \eqref{hdec} is orthogonal in $L^2(\P)$.
Thus, this decomposition generalizes the well-known \textit{Hoeffding decomposition} for integrable functionals of finitely many independent random variables to the situation of an infinite independent sequence $\X$. We will further see in Subsection \ref{Hoeffding} that, for any $p\in\N$, one has the orthogonal decomposition 
\[F^{(p)}=\sum_{M\subseteq\N: |M|=p}F_M,\]
where the series again converges unconditionally in $L^2(\P)$. 

For $k\in\N$, the \textit{influence} of the variable $X_k$ on $F$ is customarily defined as $\Inf_k(F):=\E[\Var(F|\G_k)]$ and it will become clear in Subsection \ref{mallop} below that, in terms of the decomposition \eqref{hdec}, one has the alternative formula
\begin{align*}
\Inf_k(F)=\sum_{\substack{M\subseteq\N:\\ |M|<\infty,\, k\in M}} \E[F_M^2]=\sum_{\substack{M\subseteq\N:\\ |M|<\infty,\, k\in M}} \Var(F_M).
\end{align*}
We finally define the \textit{maximal influence} of any single variable $X_k$, $k\in\N$, on $F$ as 
\begin{align*}
\rho^2(F):=\sup_{k\in\N}\Inf_k(F)=\sup_{k\in\N}\sum_{\substack{M\subseteq\N:\\ |M|<\infty,\, k\in M}} \E[F_M^2]
\end{align*}
and let $\rho(F):=\sqrt{\rho^2(F)}$.
In general, one speaks of a \textit{low influence functional}, if $\rho^2$ is small, see e.g. \cite{MOO10}. As has been demonstrated in \cite{MOO10, NPR2}, the quantity $\rho^2(F)$ plays a fundamental role for the universality of multilinear polynomial forms in independent random variables. Moreover, see again \cite{MOO10}, many important current problems in social choice theory and theoretical computer science are only stated for the low influence case in order to exclude pathological and therefore unimportant counterexamples. 

For $p\in\N$, we call $F$ as given in \eqref{hdec} a \textit{not necessarily symmetric, completely degenerate $U$-statistic of order $p$} based on $\X$, if 
\[F=F^{(p)}=\sum_{M\subseteq\N: |M|=p}F_M\quad \P\text{-a.s.,}\]
 that is, if $F_M=0$ $\P$-a.s. for all finite subsets $M$ of $\N$ with $|M|\not=p$. For simplicity, we will henceforth refer to such an $F$ as a \textit{degenerate $U$-statistic of order $p$} based on $\X$.

\begin{theorem}[Infinite quantitative de Jong CLT]\label{infdejong}
With the above notation fixed, suppose that, for some $p\in\N$, $F\in L^4(\P)$ is a degenerate $U$-statistic of order $p$ based on $\X$ such that $\E[F^2]=\Var(F)=1$. Then, for $Z\sim N(0,1)$ we have the bounds 
\begin{align}
d_\W(F,Z)&\leq \Biggl(\sqrt{\frac{2}{\pi}}+\frac{4}{3}\Biggr)\sqrt{\babs{\E[F^4]-3}}+\sqrt{\kappa_p}\Biggl(\sqrt{\frac{2}{\pi}}+ 
\frac{2\sqrt{2}}{\sqrt{3}}\Biggr)\rho(F),\quad\text{and}  \label{djwass}\\
d_\K(F,Z)&\leq  11.9\sqrt{\babs{\E[F^4]-3}} +\bigl(3.5+10.8\sqrt{\kappa_p}\bigr)\rho(F).\label{djkol}
\end{align}
Here, $\kappa_p\in(0,\infty)$ is a combinatorial constant that only depends on $p$.
\end{theorem}

In particular, Theorem \ref{infdejong} implies the following infinite generalization of a classical CLT by P. de Jong \cite[Theorem 1]{deJo90}. Multivariate and functional extensions of the result in \cite{deJo90} have been provided in \cite{DP17} and \cite{D19}, respectively.

\begin{cor}\label{dejocor}
Fix $p\in\N$ and suppose that, for each $n\in\N$, $F_n$ is a normalized, degenerate $U$-statistic of order $p$ based on an independent, finite or countably infinite sequence $\X_n$ that is defined on some probability space $(\Om_n,\F_n,\P_n)$. If $\lim_{n\to\infty}\E[F_n^4]=3$ and $\lim_{n\to\infty}\rho^2(F_n)=0$, then $F_n$ converges in distribution to $Z\sim N(0,1)$ as $n\to\infty$. 
\end{cor}

\begin{remark}\label{dejorem}
\begin{enumerate}[(a)]
\item The bounds in Theorem \ref{infdejong} are direct generalization, to the setting of an infinite underlying sequence $\X$, of previous bounds by G. Peccati and the author \cite[Theorem 1.3]{DP17} for the Wasserstein distance and by the author \cite[Theorem 2.1]{D23} on the Kolmogorov distance (see the bounds \eqref{djw} and \eqref{djk} below).
\item The constant $\kappa_p$ appearing in the above bounds stems from the article \cite{DP17}, where it is shown that one may choose $\kappa_p=2+C_p$ and the finite combinatorial constant $C_p$ is (rather implicitly) defined in display (4.5) of \cite{DP17}.  
\item Theorem \ref{infdejong} is the counterpart to the quantitative fourth moment theorems on Gaussian \cite{NouPecbook} and Poisson spaces \cite{DP18a, DP18c, DVZ18} and for Rademacher chaos \cite{DK19}. In fact, it is a generalization of \cite[Theorem 1.1]{DK19} as is explained in Section \ref{malliavin} below (see in particular Remark \ref{domrem} (e)). Contrary to the Gaussian and Poisson situations, though, it is in general not possible to remove the quantity $\rho^2(F)$ from the bound as has been shown in \cite[Theorem 1.6]{DK19}.  
\end{enumerate}
\end{remark}

\subsection{A quantitative CLT for random sums}
Suppose that $N$ and $X_j$, $j\in\N$, are random variables, defined on the same probability space $(\Om,\F,\P)$, such that $N$ has values in $\N_0$ and the $X_j$ are real-valued. Then, the quantity
\begin{align}\label{rsum}
S&:=\sum_{j=1}^NX_j=\sum_{j=1}^\infty\1_{\{N\geq j\}} X_j
\end{align}
is called a \textit{random sum} with \textit{random index} $N$. Random variables of the form \eqref{rsum} appear frequently in modern probability theory, since many models from disciplines as diverse as branching processes, nuclear physics, finance, reliability, actuarial science and risk theory naturally lead to the consideration of such sums. We refer to the monograph \cite{GneKo} for an introduction and many relevant motivating examples. 

Assume henceforth that the random variables $N,X_1,X_2,\dotsc$ are independent and, for simplicity, we further assume that the summands $X_j$, $j\in\N$, are identically distributed. The asymptotic distribution theory of such random sums is rather well studied and we again refer to \cite{GneKo} for a comprehensive treatment. The study of distributional limits for random sums apparently began with the work \cite{Rob48} of Robbins, who additionally assumes that $N$ and the $X_j$ have a finite second moment. In particular, in \cite{Rob48} it is shown that, under the natural condition $\E[N]\to\infty$, such random sums $S$ (after centering and normalization) satisfy a CLT in any of the following scenarios:
\begin{enumerate}[1)]
\item The index $N$ itself is asymptotically normal.
\item The summands $X_j$ are centered, $\E[X_1^2]>0$ and $\sqrt{\Var(N)}=o\bigl(\E[N]\bigr)$.
\item $\Var(X_1)>0$, $\E[X_1]\not=0$ and $\Var(N)=o\bigl(\E[N]\bigr)$. 
\end{enumerate}  
Once a CLT is known, one may become interested in the corresponding rate of convergence in the distances $d_\W$ and $d_\K$. It turns out that the corresponding literature on quantitative CLTs for random sums seems quite easy to survey. We focus here on the results that do not require $N$ to have a particular distribution.  Under the above assumptions from \cite{Rob48}, the paper \cite{Eng83} gives an upper bound on the Kolmogorov distance between the distribution of the random sum and a suitable normal distribution, which is proved to be sharp in some sense. However, this bound contains the Kolmogorov distance of $N$ to the normal distribution with the same mean and variance as $N$ as one of the terms and is therefore not very explicit and, hence, difficult to assess for certain distributions of $N$. The paper \cite{Kor88} generalizes the results from \cite{Eng83} to the case of not necessarily identically distributed summands and to situations, where the summands might not have finite absolute third moments. However, at least for non-centered summands, the bounds in \cite{Kor88} still lack some explicitness. The most flexible bounds, on both the Kolmogorov and the Wasserstein distance, which nevertheless reduce to completely explicit bounds for various concrete distributions of $N$ and also for general $N$ with an infinitely divisible (or even only finitely divisible) distribution, have been provided in the recent article \cite{D6} by the author. These bounds have been derived by means of a subtle interplay between two prominent coupling constructions from Stein's method that is specifically tailored to the random sums setting. The unpublished article \cite{Doe12rs} further derives Wasserstein and Kolmogorov bounds for normal and non-normal limits under the assumption of centered, independent but not necessarily identically distributed summands.  

We present here a bound on the Wasserstein distance in the CLT for random sums, when the summands are \textit{centered}, i.i.d. and have a finite fourth moment, which is derived by means of the Clark-Ocone bound \eqref{cowass} in Theorem \ref{cobound} below. To this end we write 
\[F:=\frac{S-\E[S]}{\sqrt{\Var(S)}}=\frac{S}{\sqrt{\E[N]\E[X_1^2]}}\]
for the normalized version of $S$.

\begin{theorem}\label{rstheo}
Suppose that $N,X_1,X_2,\dotsc$ are independent, that $\E[N^2]<\infty$ and that the summands $X_j$, $j\in\N$, are identically distributed with $\E[X_1]=0$ and $0<\E[X_1^4]<\infty$. Then, for $Z\sim N(0,1)$ we have the bound
\begin{align}
d_\W(F,Z)&\leq \Biggl[\sqrt{\frac{2}{\pi}}\biggl(\frac{\E|X_1^4|}{\E[X_1^2]^2}-1\biggr)^{1/2} +\frac{\E|X_1|^3}{\E[X_1^2]^{3/2}}\Biggr]\frac{1}{\sqrt{\E[N]}}+\frac{\sqrt{\Var(N)}}{\E[N]}.    \label{rswass}
\end{align}  
\end{theorem}

\begin{remark}\label{rsrem}
\begin{enumerate}[(a)]
\item Note that, under scenario 2) above and when the summands have a finite fourth moment, Theorem \ref{rstheo} provides a rate of convergence in the CLT for random sums. This result compares to \cite[Theorem 2.7]{D6}, where Wasserstein and Berry-Esseen bounds of the same order (but with slightly larger numerical constants) have been derived under the weaker assumption that $\E|X_1|^3<\infty$. The fact that we have to assume finite fourth moments here, is a general artifact of the Malliavin-Stein method that we apply in the proof of \eqref{rswass}. 
\item The order of the bound \eqref{rswass} is optimal in general. If, for instance, $N$ has the Poisson distribution with mean $\alpha\in(0,\infty)$, then the bound is of the optimal order $\alpha^{-1/2}$.   
\item Contrary to the situation of sums of a fixed number of random variables, it is a true restriction to suppose that the summands of a random sum are centered, since this actually corresponds to centering by a random variable instead of by a constant. To prove error bounds in the CLT for random sums also for non-centered summands (where centering is performed by means of a constant), it seems that one needs more specific methods like those developed in \cite{D6}, for example.  
\item Unfortunately, the Berry-Esseen bound resulting from an application of the Kolmogorov bound \eqref{cokol} in Theorem \ref{cobound} below does not converge to zero here as $\E[N]\to\infty$, unless the $X_k$, $k\in\N$, form a symmetric Rademacher sequence.  
\end{enumerate}
\end{remark}

\subsection{A quantitative CLT for the number of monochromatic edges}

Suppose that $G=(V,E)$ is a simple (non-directed) graph on $n\in\N$ vertices, which are without loss of generality given by the numbers $1,2,\dotsc,n$, that is we have $V=[n]=\{1,\dotsc,n\}$. Moreover, we denote by $m:=|E|$ the number of edges in $G$ and fix an integer $c\geq2$. The set $[c]=\{1,\dotsc,c\}$ is the set of possible \textit{colors}. In a \textit{uniform random $c$-coloring} of $G$, every vertex $i\in V$ is assigned a color $s_i\in[c]$ uniformly at random and independently of the other vertices. By choosing independent, and on $[c]$ uniformly distributed random variables $X_1,\dotsc,X_n$ on a suitable probability space $(\Om,\F,\P)$, we may thus assume that the color of $i$ is given by the random variable $X_i$, $i\in V$. Denote by $A(G)=(a_{i,j}(G))_{1\leq i,j\in V}$ the \textit{adjacency matrix} of $G$, that is, for $i,j\in V$, one has $a_{i,j}(G)=1$, if $\{i,j\}\in E$ and $a_{i,j}(G)=0$, otherwise. Then, the random variable 
\begin{align}\label{T2G}
T_2(G):=\sum_{1\leq i<j\leq n}a_{i,j}(G)\1_{\{X_i=X_j\}}
\end{align}
counts the number of \textit{monochromatic edges} in $G$, that is, the number of edges in $G$, both of whose endvertices have been assigned the same random color. The statistic $T_2(G)$ arises in various contexts in probability and non-parametric statistics. Indeed, it appears e.g. as the Hamiltonian of the Ising/Potts model on $G$ \cite{BM} and as a test statistic in non-parametric two-sample tests \cite{FrRa}. If $G$ corresponds to a friendship network and $c=365$, then $T_2(G)$ has the interpretation of the number of pairs of friends sharing the same birthday. In particular, if $G=K_n$ is the complete graph on $[n]$, then $\P(T_2(K_n)\geq1)$ is the probability that is assessed in the classical \textit{birthday problem}. We refer to the two articles \cite{BDM,BFY} and the references therein for more background information and motivation to consider the statistic $T_2(G)$.      

The possible limiting distributions of $T_2(G)$, as $n\to\infty$, have recently been identified by Bhattacharya et al. \cite{BDM} under various asymptotic regimes. We focus here on the results from \cite{BDM} on asymptotic normality. To this end, denote by 
\begin{align*}
F:=F(G):=\frac{T_2(G)-\E[T_2(G)]}{\sqrt{\Var(T_2(G))}}
\end{align*}
the normalization of $T_2(G)$. We further need the following notation: By $N(C_4,G)$ we denote the number of (not necessarily induced) isomorphic copies of a $4$-cycle in $G$, that is 
\begin{align*}
N(C_4,G)=\frac18\sum_{(i,j,k,l)\in[n]^4_{\not=}}a_{i,j}(G)a_{j,k}(G)a_{k,l}(G)a_{k,i}(G),
\end{align*} 
where $[n]^4_{\not=}$ is the set of all $(i,j,k,l)\in[n]^4$ with pairwise distinct entries and we recall that the automorphism group of $C_4$ has cardinality $8$.
\begin{prop}[Theorems 1.2 and 1.3 of \cite{BDM}] \label{bdmprop}
For $n\in\N$ let $G_n=([n],E_n)$ be a simple graph that is colored according to a uniform coloring scheme on $c_n$ colors, where $c_n\geq2$. Let $F_n:=F(G_n)$, $n\in\N$. 
\begin{enumerate}[{\normalfont (i)}]
\item If $c_n\to\infty$ as $n\to\infty$ in such a way that $|E_n|/c_n\to\infty$, then $F_n$ converges in distribution to $Z\sim N(0,1)$, as $n\to\infty$.
\item If $c_n\equiv c$ is constant and $|E_n|\to\infty$, then, as $n\to\infty$, $F_n$ converges in distribution to $Z\sim N(0,1)$, if and only if $\lim_{n\to\infty}N(C_4,G_n)/|E_n|^2=0$. 
\end{enumerate}
\end{prop}

\begin{remark}\label{bdmrem}
\begin{enumerate}[{\normalfont (a)}]
\item More generally, Theorems 1.2 and 1.3 of \cite{BDM} still hold true for random graphs $G_n$ that are independent of the colors $X_1,\dotsc,X_n$, $n\in\N$, when convergence is systematically replaced with convergence in probability.
\item The CLTs reviewed in Proposition \ref{bdmprop} are called \textit{universal} as they only depend on the numbers $c_n,|E_n|$ and $N(C_4,G_n)$ but not on any geometric details of the graphs $G_n$, $n\in\N$.
\item The condition $\lim_{n\to\infty}N(C_4,G_n)/|E_n|^2=0$ has been coined \textit{asymptotic $4$-cycle freeness} or, for short, \textit{ACF4 condition} in \cite{BDM}.
\item Fang \cite{Fang} has proved the following bound on the Wasserstein distance: 
\begin{align}\label{fangb}
d_\W(F,Z)&\leq3\sqrt{\frac{c}{m}}+\frac{10\sqrt{2}}{\sqrt{c}}+\frac{1}{\sqrt{\pi}}\frac{2^{7/4}}{m^{1/4}},
\end{align}
which gives a quantitative extension of part (i) of Proposition \ref{bdmprop}. Note, however that \eqref{fangb} does not converge to zero under the assumptions of (ii) where $c$ is fixed.
\item Recently, Bhattacharya et al \cite{BFY} have applied a classical quantitative CLT for martingales due to Heyde and Brown \cite{HB70} to obtain the Berry-Esseen bound 
\begin{align}\label{bfyb}
d_\K(F,Z)&\leq K\biggl(\frac{c}{m}+\frac{1}{\sqrt{m}}+\frac{N(C_4,G)}{cm^2}\biggr)^{1/5},
\end{align}
where $K>0$ is a finite constant that neither depends on $n$ nor on $c$. Note that, contrary to \eqref{fangb}, the bound \eqref{bfyb} implies the CLT in both situations (i) and (ii) of Proposition \ref{bdmprop}.
\end{enumerate}
\end{remark} 

Although the Berry-Esseen bound for martingales from \cite{HB70} is rate-optimal in general, as has been proved by Haeusler \cite{Haeu88}, it often leads to suboptimal estimates of the rate of convergence in concrete applications. In view of \eqref{fangb}, one is in fact led to expect that also the bound \eqref{bfyb} on the CLT for $T_2(G)$ is suboptimal. In combinatorial situations, like the the number of monochromatic edges, one would generally expect Stein's method to be capable of providing accurate estimates on the rate of convergence for the CLT. However, as has been pointed out in \cite{BDM}, to date no \textit{off-the-shelf version} of Stein's method is available that yields a quantitative version of Proposition \ref{bdmprop} in full generality. It is in fact one of the main contributions of the present work to add such a version to the shelf. Indeed, by applying our new Clark-Ocone bound \eqref{cowass} in Theorem \ref{cobound} below to the situation here, we can prove the following theorem.
   
\begin{theorem}[Quantitative CLT for the number of monochromatic edges]\label{monotheo}
Under the above assumptions and with $Z\sim N(0,1)$ we have the bound 
\begin{align}\label{monobound}
d_\W(F,Z)&\leq \sqrt{\frac{2}{\pi}}\biggl(\frac{3(c-2)}{m }+\frac{10\sqrt{2}}{m^{1/2}  }+\frac{15\sqrt{2}}{m^{1/2} (c-1) }+\frac{30N(C_4,G)}{m^2 (c-1) }\biggr)^{1/2}\notag\\
&\;+\sqrt{2}\biggl(\frac{c-1}{m}+\frac{5\sqrt{2}}{m^{1/2}(c-1)}+\frac{7\sqrt{2}}{m^{1/2}}\biggr)^{1/2}.
\end{align}
\end{theorem} 

\begin{remark}\label{monorem}
\begin{enumerate}[{\normalfont (a)}]
\item Using that $N(C_4,G)\leq A m^2$ (see e.g. \cite{Alon}), where $A\in(0,\infty)$ is an absolute constant, we see that Theorem \ref{monotheo} in particular yields both CLTs given in Proposition \ref{bdmprop}. Moreover, in the situation $c_n\to\infty$, we see that the estimate of the rate of convergence is the same as in the bound \eqref{fangb}.
\item The more general version of Proposition \ref{bdmprop} from \cite{BDM} for random graphs $G_n$ that are independent of the colors $X_1,\dotsc,X_n$ is also a consequence of Theorem \ref{monotheo}. This can be seen by first invoking our bound \eqref{monobound} to estimate the Wasserstein distance between the conditinal law of $F_n$ given $G_n$ and $N(0,1)$ and then, in view of the bound $N(C_4,G)\leq A m^2$, applying (the convergence in probability version of) the dominated convergence theorem to the simple inequality
\[d_\W(F_n,Z)\leq \E\Bigl[d_\W\bigl(\mathcal{L}(F_n|G_n),N(0,1)\bigr)\Bigr].\]
\item It might be possible to also apply the bound \eqref{cokol} in Theorem \ref{cobound} below in order to prove a corresponding Berry-Esseen bound. However, we have not yet managed to successfully deal with the respective second and third terms in the bound and might return to this point in later work. Note, however, that even after applying the inequality $d_\K(F,Z)\leq\sqrt{d_\W(F,Z)}$ to the bound \eqref{monobound}, the resulting Berry-Esseen bound is of smaller order than the bound \eqref{bfyb}. Also, contrary to \eqref{bfyb}, our bound comes with explicit numerical constants.
\item The proof of Theorem \ref{monotheo} heavily relies on our new Clark-Ocone bound \eqref{cowass}. In fact, neither of the Malliavin-Stein bound \eqref{mswass} and the carr\'{e}-du-champ bound \eqref{wasscdc} would be suitable for deriving the bound \eqref{monobound}. Indeed, taking $G$ as the $n$-star graph and letting $c=2$, for instance, one may see that the respective first terms in the bounds \eqref{mswass} and \eqref{wasscdc} do not vanish as $n\to\infty$.
\item Bhattacharya et al \cite{BFY} in fact also prove a similar bound as \eqref{bfyb} for the number $T_3(G)$ of monochromatic triangles in $G$. Moreover, they relate the quantitative CLTs for $T_2(G)$ and $T_3(G)$ to the so-called fourth-moment phenomenon first discovered by Nualart and Peccati \cite{NuaPec}.
\item Even more generally, in the recent work \cite{DHM}, Das et al have proved generalizations of the bound \eqref{bfyb} to the number of monochromatic subgraphs of $G$ that are isomorphic to a given connected graph $H$. Their proof, as the proofs in \cite{BFY}, combine the martingale CLT from \cite{HB70} with the Hoeffding decomposition of the count statistic. We believe that our proof may be adapted to this more general situation in such a way that the rate of convergence in \cite[Theorem 1.4]{DHM} can be improved, too. This will be pursued in future work that will exploit our new abstract normal approximation bounds in Section \ref{normapp} in a systematic way. 
\end{enumerate}
\end{remark}


The remainder of this work is structured as follows. In Section \ref{malliavin} we first review the concept of finite and infinite Hoeffding decompositions and then present and considerably extend the Malliavin theory from \cite{DH, Duer}. In Section \ref{normapp} we prove three different types of abstract normal approximation bounds on general product spaces. We further specialize these bounds to the particular situation of functionals of an independent Rademacher sequence. As a general rule, for the reader's convenience, in Sections \ref{malliavin} and \ref{normapp} we provide proofs immediately following the corresponding statements. In Section \ref{proofs} we give the proofs belonging to our three applications and, finally, in Section \ref{appendix} we review the concept of unconditional convergence of series in Banach spaces and provide a self-contained probabilistic proof of the infinite Hoeffding decomposition for $L^2$-functionals of countably many independent random variables.

\section{Malliavin calculus on product spaces}\label{malliavin}
In this section we review and extend the abstract Malliavin formalism for the product of countably many probability spaces  from \cite{DH}. Contrary to \cite{DH}, for ease of notation, we will assume throughout that the countable index set is given by the natural numbers $\N$. The case of a general countably infinite index set $A$ may be reduced to this situation by choosing a bijection $\phi:A\rightarrow\N$. Moreover, the situation of a product of only finitely many probability spaces is naturally included by adding countably many copies of a (trivial) one-point probability space. In this case, all infinite series appearing in the sequel further reduce to finite sums and the generally subtle problem of identifying the correct domain for the respective operators becomes trivial. 

\subsection{Setup and notation}\label{setup}
We begin by introducing the necessary notation. Thus, let $(E_n,\B_n,\mu_n)$, $n\in\N$, be an arbitrary sequence of probability spaces and denote by 
\[(E,\B,\mu):=\Bigl(\prod_{n\in\N}E_n,\bigotimes_{n\in\N}\B_n,\bigotimes_{n\in\N}\mu_n\Bigr)\]
the corresponding product space. Further, suppose that, on a suitable abstract probability space $(\Om,\A,\P)$, 
\[\X=(X_n)_{n\in\N}:(\Om,\A)\rightarrow(E,\B)\quad\text{and}\quad\X'= (X_n')_{n\in\N}:(\Om,\A)\rightarrow(E,\B)\]
are two independent sequences, each distributed according to $\mu$. In particular, the $X_n$, $n\in\N$, are independent and $X_n$ has distribution $\mu_n$.
 Then, for each $j\in\N$, we construct the sequence $\X^{(j)}:=(X_n^{(j)})_{n\in\N}$ by 
\begin{equation*}
 X_n^{(j)}:=\begin{cases}
             X_j',&n=j\\
             X_n, &n\not=j,
            \end{cases}
\end{equation*}
that is we replace the $j$-th coordinate of $\X$ with that of $\X'$. Then, of course, each sequence $\X^{(j)}$ has again distribution $\mu$. In fact, the sequences $\X$ and $\X^{(j)}$ are even \textit{exchangeable}, that is, $(\X,\X^{(j)})$ has the same distribution as $(\X^{(j)},\X)$ for any $j\in\N$. With this construction at hand, we further let $\F:=\sigma(\X)$ and for $M\subseteq\N$ we let $\F_M:=\sigma(X_j,j\in M)$. For $n\in\N$ and $M=[n]:=\{1,\dotsc,n\}$ we write $\F_n:=\F_{[n]}=\sigma(X_1,\dotsc,X_n)$ and further let $\F_0:=\F_\emptyset=\{\emptyset,\Omega\}$. Thus, $\mathbb{F}:=(\F_n)_{n\in\N_0}$ is the canonical filtration of $\A$ generated by $\X$ and $\F=\sigma(\bigcup_{n\in\N_0}\F_n)$. We also denote by $\P_\F$ the restriction of $\P$ on $\F$. For $k\in\N$ we further let $\G_k:=\F_{\N\setminus\{k\}}=\sigma(X_j,j\not=k)$. 

As usual, by $L^r(\mu)$, $r\in[0,\infty)$, we denote the space of all measurabe functions $f:(E,\B)\rightarrow(\R,\B(\R))$ s.t. $\int_E |f|^rd\mu<\infty$. Furthermore, $L^\infty(\mu)$ denotes the class of all $\mu$-essentially bounded, $\B-\B(\R)$-measurable functions. Here and in what follows, $\B(\R)$ denotes the Borel-$\sigma$-field on $\R$.

As is customary, we do not distinguish between functions and equivalence classes of a.e. identical functions here. Moreover, for $r\in[0,\infty)$ we let 
$L^r_\X:=L^r(\Om,\F,\P)$ be the space of all $\sigma(\X)$-measurable random variables $F:\Om\rightarrow\R$ such that $\E|F|^r<\infty$ and, as usual, for $r\in[1,\infty)$ we let $\norm{F}_r:=(\E|F|^r)^{1/r}$. We further denote $L^\infty_\X:=L^\infty(\Om,\F,\P)$ the space of all $\P$-essentially bounded $F\in L^0(\Om,\F,\P)$ and, for $F\in L^\infty_\X$ we let $\norm{F}_\infty$ denote its essential supremum norm. By the factorization lemma, for any $F\in L^r_\X$, $r\in[0,\infty]$, there exists an $f\in L^r(\mu)$ such that $F=f\circ\X$. Such a function $f$ will be called a \textit{representative} of $F$ in what follows. Note that such an $f$ is $\mu$-a.s. unique.

 A random variable $F\in  L^2_\X$ is called \textit{cylindrical}, if $F$ is $\F_n$-measurable for some $n\in\N$, i.e. if $F$ only depends on finitely many coordinates of $\X$. Following \cite{DH} we denote by $\mathcal{S}$ the linear subspace of $ L^2_\X$ containing all cylindrical random variables.  

We further denote by $\kappa$ the counting measure on $(\N,\Pot(\N))$, i.e. $\kappa(B)=|B|$ for all $B\subseteq\N$ and below we will also make use of the product measure $\kappa\otimes\P_\F$ on the space 
$(\N\times\Omega,\Pot(\N)\otimes\F)$. Measurable functions defined on this space will be called \textit{processes} in what follows. Here, $\Pot(\N)$ denotes the power set of $\N$ and we further write $\Pot_{fin}(\N)$ for the (countable) collection of all finite subsets of $\N$. Note that a process $U$ may be identified with a sequence $(U_k)_{k\in\N}$ of $\F-\B(\R)$- measurable random variables $U_k$, $k\in\N$. The expectation operator with respect to $\P$ or $\P_\F$ will always be denoted by $\E$. For $p\in\N$ we denote by $\N^p_{\not=}$ (respectively, by $[n]^p_{\not=}$) the set of all tuples $(i_1,\dotsc,i_p)\in\N^p$ (all tuples $(i_1,\dotsc,i_p)\in[n]^p$) such that $i_j\not=i_l$ for all $j\not=l$.

\subsection{Infinite Hoeffding decompositions}\label{Hoeffding}
In this subsection we state a fundamental result about infinite Hoeffding decompositions for random variables in $L^2_\X$. 
To this end, we first review the Hoeffding decomposition for functions of finitely many independent random variables and introduce some additional useful notation that will be used throughout this work.
We refer the reader to the monographs \cite{serfling, major, KB-book, Lee} and to the papers \cite{KR, vitale, vanZwet} for basic facts about Hoeffding decompositions for functions of finitely many independent random variables.

To begin with, let us merely assume that $F\in L^1_\X$.
Then, we can define the corresponding \textit{L\'{e}vy martingale} $(F_n)_{n\in\N}$ with respect to the filtration $\mathbb{F}:=(\F_n)_{n\in\N}$ by 
\[F_n:=\E\bigl[F\,|\,\F_n\bigl]\,, \quad n\in\N\,.\]
From martingale theory it is well-known that, as $n\to\infty$, $F_n$ converges to $F$ $\Prob$-a.s. and in $L^1(\Prob)$. Furthermore, if in fact $F\in L^r_\X$ for some $r\in(1,\infty)$, then 
$\sup_{n\in\N}\E\abs{F_n}^r<\infty$ and the convergence also takes place in $L^r(\Prob)$. As $F_n$ is $\F_n$-measurable, by the factorization lemma, there exist measurable functions 
\begin{equation*}
 g_n: \Bigl(\prod_{j=1}^n E_j,\bigotimes_{j=1}^n \B_j\Bigr)\rightarrow\bigl(\R,\B(\R)\bigr)\,,\quad n\in\N\,,
\end{equation*}
such that $F_n=g_n(X_1,\dotsc,X_n)$ for each $n\in\N$. Hence, we have the following \textit{Hoeffding decomposition} of $F_n$:
\begin{equation}\label{hdyn}
 F_n=\sum_{M\subseteq[n]} F_{n,M}\,,
\end{equation}
where we constantly write $[n]:=\{1,\dotsc,n\}$ and where the \textit{Hoeffding components} $F_{n,M}$, $M\subseteq[n]$, are $\Prob$-a.s. uniquely determined by (a) and (b) as follows:
\begin{enumerate}[(a)]
 \item For each $M\subseteq[n]$, the random variable $F_{n,M}$ is measurable with respect to $\F_M=\sigma(X_j,j\in M)$.
 \item For all $M,K\subseteq[n]$ we have $\E[F_{n,M}\,|\,\F_K]=0$ $\Prob$-a.s. unless $M\subseteq K$.
\end{enumerate}
Note that, due to independence, for $M,K\subseteq[n]$ we always have that 
\begin{equation*}
 \E[F_{n,M}\,|\,\F_K]=\E[F_{n,M}\,|\,\F_{K\cap M}]
\end{equation*}
so that we could replace (b) with
\begin{enumerate}[(c)]
 \item For all $M,L\subseteq[n]$ such that $L\subsetneq M$ we have $\E[F_{n,M}\,|\,\F_L]=0$ $\Prob$-a.s.
\end{enumerate}

The following inclusion-exclusion type formula for the Hoeffding components is well known:
\begin{align}\label{hdform}
 F_{n,M}&=\sum_{L\subseteq M}(-1)^{\abs{M}-\abs{L}}\E\bigl[F_n\,\bigl|\,\F_L\bigr]\notag\\
 &=\sum_{L\subseteq M}(-1)^{\abs{M}-\abs{L}}\E\bigl[F\,\bigl|\,\F_L\bigr]\,,\quad M\subseteq[n]\,,
\end{align}
where the second identity follows from 
\begin{equation*}
 \E\bigl[F_n\,\bigl|\,\F_L\bigr]=\E\Bigl[\E\bigl[F\,\bigl|\,\F_n\bigr]\,\Bigl|\,\F_L\Bigr]=\E\bigl[F\,\bigl|\,\F_L\bigr]\,,\quad L\subseteq[n]\,.
\end{equation*}
From now on, we will assume that, in fact, $F\in L^2_\X$. Then, the same holds for $F_n$ for each $n\in\N$ and, by \eqref{hdform}, also for its Hoeffding components $F_{n,M}$, $M\subseteq[n]$, which are in fact known to be pairwise orthogonal or uncorrelated in $L^2(\P)$ in this case. Another important consequence of \eqref{hdform} is that, for any finite subset $M$ of $\N$, the Hoeffding components $F_{n,M}$ are stationary for $n\geq \max(M)$, i.e. we have 
\begin{equation}\label{cons1}
F_{n,M}=F_{l,M}
\end{equation}
whenever $1\leq n<l$ and $M\subseteq[n]$. In particular, for each finite subset $M\subseteq\N$ and all $n\geq\max(M)$ we have 
\begin{equation}\label{cons3}
 F_{n,M}=F_{\max(M),M}\,.
\end{equation}
Henceforth, we may and will thus only write $F_M$ for $F_{\max(M),M}$.

For $p\in\N_0 ,n\in\N$ with $n\geq p$ define 
\begin{equation*}
F_n^{(p)}:=\sum_{\substack{M\subseteq[n]:\\ \abs{M}=p}} F_{M}\,.
\end{equation*}
Then, from \eqref{cons1} we conclude that 
\begin{equation}\label{cons2}
F_{n+1}^{(p)}=F_n^{(p)}+\sum_{\substack{M\subseteq[n+1]:\\ \abs{M}=p,\, n+1\in M}} F_{M}
\end{equation}
for all $p,n\in\N$ with $n\geq p$. Now it follows from (b) above that 
\[\E\bigl[F_{M}\,\bigl|\,\F_n\bigr]=\E\bigl[F_{n+1,M}\,\bigl|\,\F_n\bigr]=0\]
whenever $M\subseteq[n+1],\abs{M}=p$ and $n+1\in M$. Hence, from \eqref{cons2} we infer that, for fixed $p\in\N$, the sequence 
$(F_n^{(p)})_{n\geq p}$ is also a martingale with respect to $(\F_n)_{n\geq p}$. 
If $F\in L^2_\X$, then the martingales $(F_n)_{n\in\N}$ and $(F_n^{(p)})_{n\in\N}$ are in $L^2_\X$ as well and we have the bounds 
\begin{align}
\sup_{n\in\N}\E\bigl[F_n^2\bigr]&=\sup_{n\in\N}\E\Bigl[\Bigl(\E\bigl[F\,\bigl|\,\F_n\bigr]\Bigr)^2\Bigr]\leq \E\bigl[F^2\bigr]\quad\text{and}\label{l2bound1}\\
\sup_{p\in\N_0}\sup_{\substack{n\in\N:\\ n\geq p}}\Var\bigl(F_n^{(p)}\bigr)&\leq\sup_{n\in\N}\Var\bigl(F_n\bigr)=\sup_{n\in\N}\E\bigl[F_n^2\bigr]-\bigl(\E[F]\bigr)^2\leq \Var(F)\label{l2bound2}\,,
\end{align}
where we have used the orthogonality of the Hoeffding decomposition to obtain
\begin{equation*}
 \Var\bigl(F_n^{(p)}\bigr)=\sum_{\substack{M\subseteq[n]:\\ \abs{M}=p}}\Var\bigl(F_{M}\bigr)\leq \sum_{M\subseteq[n]}\Var\bigl(F_{M}\bigr)=\Var(F_n)\,.
\end{equation*}
From \eqref{l2bound2} and martingale theory we conclude that, for fixed $p\in\N_0$, $(F_n^{(p)})_{n\geq p}$ converges to some random variable $F^{(p)}\in L^2_\X$ both $\P$-a.s. and in $L^2(\P)$. 
Moreover, due to the convergence in $L^2(\P)$, for $p\not=q$ we have 
\begin{equation}\label{orthrel}
 \E\bigl[F^{(p)} F^{(q)}\bigr]=\lim_{n\to\infty}\E\bigl[F_n^{(p)} F_n^{(q)}\bigr]=0\,,
\end{equation}
where we have again used the orthogonality of the Hoeffding components for the second equality. This leads to the natural question if the identity 
\begin{equation}\label{Fsum}
 F=\sum_{p=0}^\infty F^{(p)}
\end{equation}
holds in the $L^2(\P)$-sense. In this case \eqref{Fsum} can be considered an \textit{infinite orthogonal Hoeffding decomposition} of $F$ in $L^2(\P)$. 
In fact, one has the following result, whose statement and an outline of a proof have been given in the introduction of \cite{kwapien}. A similar approach to this decomposition can be found in the proof of \cite[Lemma 2.4]{Duer}. Starting from the preparations above, it may be proved from standard facts about unconditional convergence of series with orthogonal summands in Hilbert spaces. Since this result is of crucial importance for the theory in this work, we provide a complete and purely probabilistic proof in Section \ref{appendix}.   

\begin{prop}[Infinite Hoeffding decomposition]\label{infhoeff}
Suppose that $F\in L^2_\X$. Then, using the above notation, the following orthogonal and unconditional series expansions hold in the $L^2(\P)$-sense:
\begin{align}
F^{(p)}&=\sum_{\substack{M\subseteq\N:\\ \abs{M}=p}}F_{M},\quad 
p\in\N_0\label{genhd1}\,,\\
F&=\sum_{M\in\Pot_{fin}(\N)}F_{M} 
=\sum_{p=0}^\infty F^{(p)}\label{genhd2}\,.
\end{align}
The representations of $F$ in \eqref{genhd2} are $\P$-a.s. unique and both are called the \textit{infinite Hoeffding decomposition} of $F$ in $L^2_\X$. 
\end{prop}

For $p\in\N_0$ the space $\mathcal{H}_p=\mathcal{H}_{p,\X}$ consisting of all random variables $F\in L^2_\X$ whose infinite Hoeffding decomposition is of the form
\[F=F^{(p)}=\sum_{M\subseteq\N: |M|=p}F_M\]
is called the \textit{$p$-th Hoeffding space} associated with $\X$. It is easy to see that each $\mathcal{H}_p$, $p\in\N_0$, is a closed linear subspace of $L^2_\X$.
Let us denote by $J_p$ the orthogonal projection on $\mathcal{H}_p$, $p\in\N_0$. Thus, if $F\in L^2_\X$ has the infinite Hoeffding decomposition \eqref{genhd2}, then 
\begin{equation}\label{intformHD}
 F=\sum_{p=0}^\infty J_p(F)=\E[F]+\sum_{p=1}^\infty J_p(F),
\end{equation}
where $J_p(F)=F^{(p)}$, $p\in\N_0$.
We will see in the next subsection that the infinite Hoeffding decomposition plays the same role in our setting as the Wiener-It\^{o} chaos decomposition does for the Malliavin calculus on Gaussian \cite{Nualart} and Poisson spaces \cite{Lastsa} or for Rademacher sequences \cite{Priv08}. Therefore, for $p\in\N_0$, we also call an $F\in\mathcal{H}_p$ a \textit{$p$-th chaos}. 
If $p\geq1$ and for a finite vector $\X=(X_1,\dotsc,X_n)$ of independent random variables, in statistical theory such a random variable is also called a (not necessarily symmetric) \textit{completely degenerate $U$-statistic of order $p$}. We continue to use this notion in our setting of an infinite sequence $\X$.

\subsection{Malliavin operators and chaotic decomposition}\label{mallop}

In this subsection we review the definitions and domains of the three Malliavin operators $D,\delta$ and $L$ given in \cite{DH}. Moreover, we extend the theory therein by showing that the infinite Hoeffding decomposition from Proposition \ref{infhoeff} plays the role of a chaotic decomposition for random variables $F\in L^2_\X$. In particular, we derive new and explicit characterizations of the domains for $D,\delta$ and $L$ in terms and by means of this decomposition. On the one hand these criteria facilitate the verification, whether the random variable under consideration belongs to the respective domain. On the other hand, several novel important structural properties of the operators $D$ and $L$ are proved via the infinite Hoeffding decomposition. We begin by reviewing the relevant definitions and theory from \cite{DH} before stating and proving our various extensions.\smallskip\\

For $k\in\N$ and $F\in L^1_\X$ with representative $f\in L^1(\mu)$ we define 
\begin{equation*}
D_kF:=F-\E[F\,|\,\G_k]=F-\E\bigl[f\bigl(\X^{(k)}\bigr)\,\bigl|\,\X\bigr]=\E\bigl[f(\X)-f(\X^{(k)})\,\bigl|\,\X\bigr]. 
\end{equation*}
It is easy to see that each $D_k$ defines a self-adjoint operator with norm one on $L^2_\X$. Indeed, $D_k$ is the orthogonal projection on the space of all random variables in $L^2_\X$ that really depend on the random variable $X_k$. It thus follows that $D_kD_k=D_k$ and $D_kD_l=D_lD_k$ for all $k,l\in\N$. Furthermore, $F\in L^r_\X$ implies $D_kF\in L^r_\X$ and $\norm{D_kF}_r\leq 2\norm{F}_r$ for all $r\in[1,\infty)$. Also note that the symmetry relation $\E[FD_kG]=\E[GD_kF]$ continues to hold for $G\in L^r_\X$ and $F\in L^s_\X$, if $r,s\in[1,\infty]$ are such that $r^{-1}+s^{-1}=1$ (with the usual convention that $1/\infty:=0$).

Moreover, we let $\tilde{D}F:=(D_kF)_{k\in\N}$ denote the corresponding process. This definition makes sense for any $F\in L^2_\X$ (in fact even for any $F\in L^1_\X$) although $\tilde{D}F$ need not necessarily be an element of $L^2(\kappa\otimes\P_\F)$. Thus, in \cite{DH} the \textit{Malliavin derivative} $D$ is first defined on the $L^2_\X$-dense subspace $\mathcal{S}$ via $DF=\tilde{D}F$, which is always in $L^2(\kappa\otimes\P_\F)$ as $D_kF=0$ if $F$ is $\F_n$-measurable and $k>n$ (see \cite[Lemma 3.1]{DH}). Then, it is proved in \cite[Corollary 2.5]{DH} that this operator on $\mathcal{S}$ is closable and its closure is again denoted by $D$. We write $\dom(D)$ for the domain of this closure, which therefore defines a densely defined, closed operator $D:\dom(D)\subseteq L^2_\X\rightarrow L^2(\kappa\otimes\P_\F)$. In \cite[Lemma 2.6]{DH}, a sufficient condition for $F\in L^2_\X$ to be in $\dom(D)$ is given but no explicit characterization of $\dom(D)$ is provided. Moreover, it is not clear from the outset that, for $F\in\dom(D)\setminus\mathcal{S}$, this operator is still given by $DF=(D_kF)_{k\in\N}$, although we will see that this holds in Lemma \ref{domDle} below. Since $\dom(D)$ is in general a strict subset of $L^2_\X$, it is thus desirable to have verifiable equivalent conditions for $F$ to be in $\dom(D)$. These are provided by Proposition \ref{domD} below. The following slight extension of \cite[Lemma 3.2]{DH} will be used in what follows without further mention: For $F\in L^2_\X$, $k\in\N$ and $M\subseteq \N$ one has
\begin{align*}
 D_k\E\bigl[F\,\big|\,\F_M\bigr]=\E\bigl[F\,\big|\,\F_M\bigr]-\E\bigl[F\,\big|\,\F_M\cap \G_k\bigr]=\E\bigl[D_kF\,\big|\,\F_M\bigr].
\end{align*}

\medskip

As usual, the \textit{divergence operator} $\delta$ has been defined in \cite{DH} as the adjoint operator of $D$. In particular, by definition, its domain $\dom(\delta)$ is the collection of all processes $U=(U_k)_{k\in\N}\in L^2(\kappa\otimes\P_\F)$ such that there is a constant $C\in[0,\infty)$ with the property that 
\[\Babs{\langle DF,U\rangle_{L^2(\kappa\otimes\P_\F)}}=\Bbabs{\sum_{k=1}^\infty \E\bigl[D_kF U_k\bigr]}\leq C\norm{F}_2\]
holds for any $F\in\dom(D)$. For $U=(U_k)_{k\in\N}\in\dom(\delta)$ and $F\in\dom(D)$ one thus has the \textit{Malliavin integration by parts formula}
\begin{equation}\label{intpartsM}
\langle DF,U\rangle_{L^2(\kappa\otimes\P_\F)}=\E\bigl[F\delta U\bigr]
\end{equation} 
which, as $\langle DF,U\rangle_{L^2(\kappa\otimes\P_\F)}=\sum_{k=1}^\infty \E[D_kF U_k]=\sum_{k=1}^\infty \E[F D_kU_k]$, entails that 
\begin{equation}\label{formdelta}
\delta U=\sum_{k=1}^\infty D_kU_k.
\end{equation}
Note that for the latter argument to be sound, one actually has to make sure that interchanging the infinite sum and the expectation is justified. We refer to Proposition \ref{domdelta} below for a rigorous proof of the formula \eqref{formdelta} as well as for new explicit characterizations of $\dom(\delta)$.
\smallskip\\

The third Malliavin operator defined in \cite{DH} is the \textit{Ornstein-Uhlenbeck generator} or \textit{number operator} $L$ corresponding to $\X$. By definition, its domain $\dom(L)$ is the collection of all $F\in\dom(D)$ such that $DF\in\dom(\delta)$ and then one defines 
\[LF=-\delta DF=-\sum_{k=1}^\infty D_kF=\sum_{k=1}^\infty \E\bigl[f(\X^{(k)})-f(\X)  \,\bigl|\,\X\bigr],\]
where the explicit formula \eqref{formdelta} for $\delta U$ as well as $D_kD_k=D_k$ have been used for the second equality. In fact, in \cite[Theorem 2.11]{DH} it is shown that $L$ is the infinitesimal generator of a strongly Feller semigroup $(P_t)_{t\geq0}$ on $L^\infty(\mu)$, the \textit{Ornstein-Uhlenbeck semigroup} associated to $\X$ and \cite[Theorem 2.12]{DH} provides the important \textit{Mehler formula} for $(P_t)_{t\geq0}$. \smallskip\\

The remainder of this section is devoted to our extensions to the theory developed in \cite{DH} and \cite{Duer}. Most of the results to follow are of independent interest but are also important for the proofs of and/or for applying the normal approximation bounds in Section \ref{normapp}.

\begin{lemma}\label{domDle}
Suppose that $F_n,F\in L^2_\X$, $n\in\N$, are such that $(F_n)_{n\in\N}$ converges in $L^2_\X$ to $F$. Then, for any $k\in\N$, the sequence $(D_kF_n)_{n\in\N}$ converges to $D_kF$ in $L^2_\X$. Moreover, if $F\in\dom(D)$, then $DF=\tilde{D}F=(D_kF)_{k\in\N}$.
\end{lemma}

\begin{proof}
The first claim follows from $D_kF_n=F_n-\E[F_n|\G_k]$ and $D_kF=F-\E[F|\G_k]$ via well-known properties of conditional expectations. For the second claim let $DF=(V_k)_{k\in\N}\in L^2(\kappa\otimes\P_\F)$. If $F\in\dom(D)$, then there is a sequence $F_n\in\mathcal{S}$, $n\in\N$, such that $F_n\rightarrow F$ in $L^2(\P)$ and $DF_n\rightarrow DF$ in $L^2(\kappa\otimes\P_\F)$ as $n\to\infty$. Thus, for each fixed $k\in\N$, 
\[0\leq \E\Bigl[\bigl(D_kF_n-V_k\bigr)^2\Bigr]\leq\sum_{l=1}^\infty\E\Bigl[\bigl(D_lF_n-V_l\bigr)^2\Bigr]=\int_{\N\times\Omega} (DF_n-DF)^2d(\kappa\otimes\P_\F)\stackrel{n\to\infty}{\longrightarrow}0.\]
Since $(D_kF_n)_{n\in\N}$ converges also to $D_kF$ in $L^2_\X$ by the first claim, it follows that $V_k=D_kF$ $\P$-a.s. for each $k\in\N$. Since $\N$ is countable, this implies that 
$(V_k)_{k\in\N}=(D_kF)_{k\in\N}$ $\P$-a.s.
\end{proof}

\begin{prop}\label{HoeffD}
Let $F\in L^2_\X$ have the infinite Hoeffding decomposition \eqref{genhd2}. Then, for any $k\in\N$, the infinite Hoeffding decomposition of $D_kF$ is given by\\ 
$D_kF=\sum_{M\in\Pot_{fin}(\N): k\in M} F_M$.
\end{prop}

\begin{proof}
Note that, by the defining properties of the Hoeffding components $F_M$, we have $\E[F_M|\G_k]=\1_{\{k\notin M\}} F_M$ $\P$-a.s. for any $k\in\N$ and any $M\in\Pot_{fin}(\N)$. Hence, using that limits in $L^2(\P)$ and conditional expectation may be interchanged, we obtain
\begin{align*}
D_kF&=F-\E[F_M|\G_k]=\sum_{M\in\Pot_{fin}(\N)}F_M-\sum_{M\in\Pot_{fin}(\N)}\E\bigl[F_M\,\bigl|\,\G_k\bigr]\\
&=\sum_{M\in\Pot_{fin}(\N)}F_M-\sum_{M\in\Pot_{fin}(\N)}\1_{\{k\notin M\}}F_M
=\sum_{M\in\Pot_{fin}(\N): k\in M} F_M.
\end{align*}
\end{proof}

\begin{cor}[Stroock type formula]\label{stroock}
 Let $F\in L^2_\X$ have the infinite Hoeffding decomposition \eqref{genhd2}. Then, for any 
 $M=\{i_1,\dotsc,i_p\}\in\Pot_{fin}(\N)$ it holds that 
 $F_M=\E\bigl[D_{i_1}D_{i_2}\ldots D_{i_p}F\,\bigl|\, \F_M\bigr]$. Thus, one has the explicit formulae
\begin{align*}
F&=\E[F]+\sum_{p=1}^\infty\sum_{1\leq i_1<i_2<\ldots<i_p<\infty}\E\bigl[D_{i_1}D_{i_2}\ldots D_{i_p}F\,\bigl|\, \F_{\{i_1,\dotsc,i_p\}}\bigr]\\
&=\E[F]+\sum_{p=1}^\infty\frac{1}{p!}\sum_{(i_1,\dotsc,i_p)\in\N^p_{\not=}}\E\bigl[D_{i_1}D_{i_2}\ldots D_{i_p}F\,\bigl|\, \F_{\{i_1,\dotsc,i_p\}}\bigr].
\end{align*}
 \end{cor}

 \begin{proof}
  By iteration, from Proposition \ref{HoeffD} we obtain that 
  \begin{align*}
  D_{i_1}D_{i_2}\ldots D_{i_p}F&= \sum_{N\in\Pot_{fin}(\N): M\subseteq N} F_N.
  \end{align*}
  Since the series here converges (unconditionally) in $L^2(\P)$, we can interchange summation and the conditional expectation with respect to $\F_M$, and then property (b) of the Hoeffding components yields the first claim. The formulae for $F$ now follow from the first claim and \eqref{genhd2} by observing that $D_kD_l=D_lD_k$ for all $k,l\in\N$. 
 \end{proof}

The next result gives novel verifiable criteria for $F\in L^2_\X$ to be in $\dom(D)$.

\begin{prop}\label{domD}
For a random variable $F\in L^2_\X$ the following conditions are equivalent:
\begin{enumerate}[{\normalfont (i)}]
\item $F\in \dom(D)$.
\item $(D_kF)_{k\in\N}\in L^2(\kappa\otimes\P_\F)$, i.e. $\sum_{k=1}^\infty \E[(D_kF)^2]<\infty$.
\item For one (and then every) representative $f\in L^2(\mu)$ of $F$ one has \\$\displaystyle\sum_{k=1}^\infty\E\Bigl[\Bigl(f\bigl(\X^{(k)}\bigr)-f\bigl(\X\bigr)\Bigr)^2 \Bigr]<\infty$.
\item With the infinite Hoeffding decomposition \eqref{genhd2} of $F$
one has that \\
$\displaystyle\sum_{p=1}^\infty p\,\E\bigl[(F^{(p)})^2\bigr]=\sum_{p=1}^\infty p\,\E\bigl[J_p(F)^2\bigr]=  \sum_{M\in \Pot_{fin}(\N)}|M|\,\E[F_{M}^2]<\infty$.
\end{enumerate}
\end{prop}

\begin{proof}
(i)$\Rightarrow$(ii): Suppose that (i) holds but (ii) does not. Then, by (i) there is a sequence $F_n\in\mathcal{S}$, $n\in\N$, such that $F_n\to F$ in $L^2(\P)$ and $DF_n=\tilde{D}F_n\to DF$ in $L^2(\kappa\otimes\P_\F)$ as $n\to\infty$. In particular, there is a $C\in(0,\infty)$ such that $\norm{DF_n}_{L^2(\kappa\otimes\P_\F)}^2= \sum_{k=1}^\infty \E[(D_kF_n)^2]\leq C$ for each $n\in\N$. On the other hand, since (ii) does not hold, there is an $m\in\N$ such that $\sum_{k=1}^m \E[(D_kF)^2]>2C$. By Lemma \ref{domDle}, from $F_n\to F$ in $L^2(\P)$ it follows that $D_kF_n\to D_kF$ in $L^2(\P)$ as $n\to\infty$ for each $k\in\N$. Therefore, 
\[\lim_{n\to\infty}\sum_{k=1}^m \E[(D_kF_n)^2]=\sum_{k=1}^m \E[(D_kF)^2]>2C\]
contradicting 
\[\sum_{k=1}^m \E[(D_kF_n)^2]\leq \sum_{k=1}^\infty \E[(D_kF_n)^2]\leq C\]
for each $n\in\N$. Hence, (i) implies (ii).\smallskip\\
(ii)$\Leftrightarrow$(iv): From Proposition \ref{HoeffD} we know that $D_kF=\sum_{M\in\Pot_{fin}(\N):k\in M}F_M$ for each $k\in\N$. Thus, from the orthogonality of the Hoeffding components we have
\begin{align*}
\sum_{k=1}^\infty \E[(D_kF)^2]&=\sum_{k=1}^\infty \Bigl(\sum_{M\in\Pot_{fin}(\N):k\in M}\E\bigl[F_M^2\bigr]\Bigr)=\sum_{M\in\Pot_{fin}(\N)}\E\bigl[F_M^2\bigr]\sum_{k\in M}1\\
&=\sum_{M\in\Pot_{fin}(\N)}|M|\,\E\bigl[F_M^2\bigr]=\sum_{p=1}^\infty p\,\E\bigl[(F^{(p)})^2\bigr],
\end{align*}
where the last equality follows from \eqref{genhd1} and the (general) rearrangement theorem for real series with nonnegative summands.\smallskip\\
(ii)$\Leftrightarrow$(iii): On the one hand, for $k\in\N$, we have 
\begin{align*}
\E\bigl[(D_kF)^2\bigr]&=\E\Bigl[F^2+\bigl(\E\bigl[F\,\big|\,\G_k\bigr]\bigr)^2-2F\bigl(\E\bigl[F\,\big|\,\G_k\bigr]\bigr)\Bigr]
=\E[F^2]-\E\Bigl[\bigl(\E\bigl[F\,\big|\,\G_k\bigr]\bigr)^2\Bigr]
\end{align*}
and, on the other hand, recalling that $\X$ and $\X^{(k)}$ have the same distribution we have 
\begin{align}\label{eqdomd}
&\E\Bigl[\Bigl(f\bigl(\X^{(k)}\bigr)-f\bigl(\X\bigr)\Bigr)^2 \Bigr]=\E\Bigl[f\bigl(\X^{(k)}\bigr)^2+f\bigl(\X\bigr)^2-2f\bigl(\X^{(k)}\bigr)f\bigl(\X\bigr)\Bigr]\notag\\
&=2\E[F^2]-2\E\Bigl[f\bigl(\X\bigr)\E\bigl[f\bigl(\X^{(k)}\bigr)\,\big|\,\X\bigr]\Bigr]=2\E[F^2]-2\E\Bigl[F\,\E\bigl[F\,\big|\,\G_k\bigr]\Bigr]\notag\\
&=2\E[F^2]-2\E\Bigl[\bigl(\E\bigl[F\,\big|\,\G_k\bigr]\bigr)^2\Bigr]=2\E\bigl[(D_kF)^2\bigr],
\end{align}
proving the claim.\smallskip\\
(iv)$\Rightarrow$(i): For $n\in\N$ let $F_n:=\E[F|\F_n]\in\mathcal{S}$. Then, we know that $F_n\to F$ in $L^2(\P)$ as $n\to\infty$ and by the definition of the closure of an operator it suffices to show that $(DF_n)_{n\in\N}$ converges in $L^2(\kappa\otimes\P_\F)$, as then $F\in \dom(D)$ and $DF=\lim_{n\to\infty} DF_n$ in $L^2(\kappa\otimes\P_\F)$. But, for integers $n>m\geq1$ we have 
\begin{align*}
&\sum_{k=1}^\infty\E\Bigl[\bigl(D_kF_n-D_kF_m\bigr)^2\Bigr]=\sum_{k=1}^\infty\sum_{\substack{M\subseteq[n]:\\k\in M, M\cap[m+1,n]\not=\emptyset}}\E\bigl[F_M^2\bigr]
=\sum_{\substack{M\subseteq[n]:\\k\in M, M\cap[m+1,n]\not=\emptyset}}|M|\E\bigl[F_M^2\bigr]\notag\\
&=\sum_{\substack{M\subseteq[n]:\\k\in M}}|M|\E\bigl[F_M^2\bigr]-\sum_{\substack{M\subseteq[m]:\\k\in M}}|M|\E\bigl[F_M^2\bigr]\longrightarrow0,
\end{align*}
as $n,m\to\infty$ by (iii). Hence, $(DF_n)_{n\in\N}$ is a Cauchy sequence in the Hilbert space $L^2(\kappa\otimes\P_\F)$ and therefore convergent.
\end{proof}

The following proposition similarly characterizes the domain of the adjoint $\delta$ of $D$.

\begin{prop}\label{domdelta}
For a process $U=(U_k)_{k\in\N}\in L^2(\kappa\otimes\P_\F)$ the following statements are equivalent: 
\begin{enumerate}[{\normalfont (i)}]
  \item $U\in\dom(\delta)$.
	\item If $\sum_{M\in\Pot_{fin}(\N)}U_{k,M}$ is the infinite Hoeffding decomposition of $U_k\in L^2_\X$, $k\in\N$, then $\displaystyle\sum_{M\in\Pot_{fin}(\N)}\E\Bigl[\Bigl(\sum_{k\in M}U_{k,M}\Bigr)^2\Bigr]<\infty$.
  \item The series $\sum_{k=1}^\infty D_kU_k$ converges in $L^2_\X$.
	\item $\displaystyle\sup_{m\in\N}\E\biggl[\Bigl(\sum_{k=1}^m D_kU_k\Bigr)^2\biggr]<\infty$.
  \end{enumerate}
	In this case, one has $\delta U=\sum_{k=1}^\infty D_kU_k$ in $L^2(\P)$.
\end{prop}

\begin{proof}
(i)$\Rightarrow$(ii): Suppose that (ii) does not hold. Then, for any $n\in\N$, there are finitely many pairwise distinct sets $M_1,\dotsc,M_{m_n}\in\Pot_{fin}(\N)$ such that 
\[\sum_{j=1}^{m_n}\E\Bigl[\Bigl(\sum_{k\in M_j}U_{k,M_j}\Bigr)^2\Bigr]>n,\quad n\in\N.\]
For $n\in\N$, define $F_n:=\sum_{j=1}^{m_n}\sum_{l\in M_j} U_{l,M_j}$ which is certainly in $\dom(D)$ by Proposition \ref{domD}. Moreover, by the definition of $\delta$, orthogonality, the symmetry of the $D_k$, $k\in\N$, and by Proposition \ref{HoeffD}, for any $n\in\N$,
\begin{align*}
\Babs{\E\bigl[F_n \delta U\bigr]}&=\Babs{\sum_{k=1}^\infty \E\bigl[D_kF_n U_k\bigr]}=\Babs{\sum_{k=1}^\infty \E\bigl[F_n D_k U_k\bigr]}\\
&=\Babs{\sum_{k=1}^\infty \sum_{j=1}^{m_n}\sum_{l\in M_j}\1_{\{k\in M_j\}}\E\bigl[U_{k,M_j}U_{l,M_j}\bigr]}=\sum_{j=1}^{m_n}\E\Bigl[\Bigl(\sum_{k\in M_j}U_{k,M_j}\Bigr)^2\Bigr]\\
&=\norm{F_n}_2^2>\sqrt{n}\norm{F_n}.
\end{align*}
Hence, $U$ cannot be in $\dom(\delta)$ and thus (i) does not hold, either.\smallskip\\
(ii)$\Rightarrow$(iii): Since $L^2_\X$ is a Hilbert space, it suffices to show that the series $\sum_{k=1}^\infty D_kU_k$ is Cauchy in $L^2_\X$. Using Proposition \ref{HoeffD} and orthogonality, for integers $n>m\geq1$ we have 
\begin{align}\label{domdel1}
\E\biggl[\Bigl(\sum_{k=m+1}^n D_kU_k\Bigr)^2\biggr]&=\sum_{k,l=m+1}^n \E\bigl[D_kU_k D_lU_l\bigr]=\sum_{k,l=m+1}^n\sum_{M\in\Pot_{fin}(\N):k,l\in M}\E\bigl[U_{k,M}U_{l,M}\bigr]\notag\\
&=\sum_{M\in\Pot_{fin}(\N)}\E\Bigl[\Bigl(\sum_{k\in M\cap[m+1,n]}U_{k,M}\Bigr)^2\Bigr]\stackrel{m,n\to\infty}{\longrightarrow}0
\end{align} 
by (ii) and the dominated convergence theorem.\smallskip\\
(iii)$\Rightarrow$(i): If (iii) holds, then $C:=\norm{\sum_{k=1}^\infty D_kU_k}_2$ is a well-defined and finite constant and, thanks to the symmetry of $D_k$, $k\in\N$, for any $F\in\dom(D)$ one further has that 
\begin{align*}
\Babs{\langle DF,U\rangle_{L^2(\kappa\otimes\P_\F)}}&=\Bbabs{\sum_{k=1}^\infty \E\bigl[D_kF U_k\bigr]}=\Bbabs{\sum_{k=1}^\infty \E\bigl[F D_kU_k\bigr]}
=\Bbabs{\E\biggl[F \sum_{k=1}^\infty D_kU_k\biggr]}\\
&\leq \norm{F}_2 \BEnorm{\sum_{k=1}^\infty D_kU_k}=C\norm{F}_2,
\end{align*}
implying $U\in\dom(\delta)$. Note that the $L^2(\P)$-convergence of the series has been used to obtain the third equality.\smallskip\\
(iii)$\Rightarrow$(iv): This is clear.\smallskip\\
(iv)$\Rightarrow$(ii): Let $S<\infty$ be the supremum in (iv). As in \eqref{domdel1}, for $m\in\N$, we have  
\begin{align}\label{domdel2}
+\infty&>S\geq\E\biggl[\Bigl(\sum_{k=1}^m D_kU_k\Bigr)^2\biggr]=\sum_{M\in\Pot_{fin}(\N)}\E\Bigl[\Bigl(\sum_{k\in M\cap[m]}U_{k,M}\Bigr)^2\Bigr].
\end{align} 
If the sum in (ii) were not finite, then there would exist finitely many $M_1,\dotsc,M_r\in\Pot_{fin}(\N)$ such that 
\[\sum_{j=1}^r\E\Bigl[\Bigl(\sum_{k\in M_j}U_{k,M_j}\Bigr)^2\Bigr]\geq S+1.\]
But then, choosing $m\geq\max(M_1\cup\ldots\cup M_r)$ in \eqref{domdel2} would lead to the contradiction $S\geq S+1$, which is impossible for finite $S$. Hence (iv) implies (ii).
\smallskip\\
To prove the claimed explicit formula for $\delta U$ recall that, by \eqref{intpartsM}, $\delta U$ is characterized by
\begin{equation*}
\E\bigl[F\delta U\bigr]=\langle DF,U\rangle_{L^2(\kappa\otimes\P_\F)}=\sum_{k=1}^\infty \E\bigl[D_kF U_k\bigr]=\lim_{n\to\infty} \E\bigl[F \sum_{k=1}^nD_kU_k\bigr],
\end{equation*}
for all $F\in\dom(D)$, where the last equality holds again due to the symmetry of $D_k$, $k\in\N$. Now, since $F\in L^2_\X$, by (iii) we can interchange the limit and the expectation here to obtain
\[\E\bigl[F\delta U\bigr]=\lim_{n\to\infty} \E\bigl[F \sum_{k=1}^nD_kU_k\bigr]=\E\bigl[F \sum_{k=1}^\infty D_kU_k\bigr],\]
implying that $\delta U=\sum_{k=1}^\infty D_kU_k$ since $\dom(D)$ is dense in $L^2_\X$.
\end{proof}

 As for the Malliavin derivative $D$ and the divergence $\delta$, we can give more explicit characterizations of $\dom(L)$.

\begin{prop}\label{domL}
 For $F\in L^2_\X$ with the infinite Hoeffding decomposition \eqref{genhd2} the following conditions are equivalent:
 \begin{enumerate}[{\normalfont (i)}]
  \item $F\in\dom(L)$.
  \item The series $\sum_{k=1}^\infty D_kF$ converges in $L^2_\X$. 
  \item $\displaystyle \sup_{m\in\N}\E\Bigl[\Bigl(\sum_{k=1}^m D_kF\Bigr)^2\Bigr]<\infty$.
	\item $\displaystyle\sum_{p=1}^\infty p^2\,\E\bigl[(F^{(p)})^2\bigr]=\sum_{p=1}^\infty p^2\,\E\bigl[J_p(F)^2\bigr] = \sum_{M\in \Pot_{fin}(\N)}|M|^2\,\E[F_{M}^2]<\infty$.
 \end{enumerate}
 In this case, $LF=-\sum_{k=1}^\infty D_kF$ has the infinite Hoeffding decomposition \\
 $\displaystyle LF=-\sum_{M\in\Pot_{fin}(\N)}|M|\, F_M =-\sum_{p=1}^\infty p\,F^{(p)}=-\sum_{p=1}^\infty p\,J_p(F)$.
\end{prop}

\begin{proof}
(i)$\Rightarrow$(ii): If $F\in\dom(L)$, then, by definition, $F\in\dom(D)$ and $DF=(D_kF)_{k\in\N}\in\dom(\delta)$. Since $D_kD_k=D_k$, by Proposition \ref{domdelta} this in particular implies that 
$\sum_{k=1}^\infty D_kF$ converges in $L^2_\X$.\smallskip\\
(ii)$\Rightarrow$(iii): This is clear. \smallskip\\
(iii)$\Rightarrow$(iv): Using orthogonality, for each fixed $m\in\N$ we have 
\begin{align*}
 & \E\Bigl[\Bigl(\sum_{k=1}^m D_kF\Bigr)^2\Bigr]
=\E\biggl[\Bigl(\sum_{k=1}^m \sum_{M\in\Pot_{fin}(\N):k\in M}F_M\Bigr)^2\biggr]\\
&=\sum_{k,l=1}^m\E\biggl[ \sum_{\substack{M,N\in\Pot_{fin}(\N):\\
 k\in M, l\in N}}F_M F_N\biggr] =\sum_{k,l=1}^m \sum_{\substack{M,N\in\Pot_{fin}(\N):\\
 k\in M, l\in N}}\E\bigl[F_M F_N\bigr]\\
&=\sum_{k,l=1}^m \sum_{\substack{M\in\Pot_{fin}(\N):\\ k,l\in M}}\E\bigl[F_M^2\bigr]
 =\sum_{M\in\Pot_{fin}(\N)}|M\cap[m]|^2\, \E\bigl[F_M^2\bigr].
\end{align*}
By (iii) the left hand side here is bounded in $m\in\N$. On the other hand, by the monotone convergence theorem, as $m\to\infty$, the right hand side converges to
\[\sum_{M\in\Pot_{fin}(\N)}|M|^2\, \E\bigl[F_M^2\bigr]=\sum_{p=1}^\infty p^2\,\E\bigl[(F^{(p)})^2\bigr],\]
which then must be finite as well. This proves (iv).\smallskip\\
(iv)$\Rightarrow$(i): By Proposition \ref{domD}, (iv) implies that $F\in\dom(D)$ and we thus need to make sure that $DF\in\dom(\delta)$, i.e. that there is a constant $C\in[0,\infty)$ with the property that 
\begin{align*}
\Babs{\langle DF,DG\rangle_{L^2(\kappa\otimes\P_\F)}}=\Bbabs{\sum_{k=1}^\infty \E\bigl[D_kF D_kG\bigr]}\leq C\norm{G}_2
\end{align*}
holds for any $G\in\dom(D)$. Let 
\[C:=\Biggl(\sum_{M\in \Pot_{fin}(\N)}|M|^2\,\E[F_{M}^2]\Biggr)^{1/2}\]
which is finite by (iv). 
Now, let $G\in \dom(D)$ be given with infinite Hoeffding decomposition
\[G=\sum_{M\in\Pot_{fin}(\N)}G_M.\]
Then, since  $F\in\dom(D)$ by (iv) and Proposition \ref{domD}, using orthogonality, we obtain
\begin{align*}
 &\Bbabs{\sum_{k=1}^\infty \E\bigl[D_kF D_kG\bigr]}=\Bbabs{\sum_{k=1}^\infty\sum_{\substack{M,N\in\Pot_{fin}(\N):\\
 k\in M\cap N}}\E\bigl[F_M G_N\bigr]}\leq\sum_{k=1}^\infty\sum_{\substack{M\in\Pot_{fin}(\N):\\ k\in M}}\Babs{\E\bigl[F_MG_M\bigr]}\\
 &=\sum_{M\in\Pot_{fin}(\N)}|M|\babs{\E\bigl[F_MG_M\bigr]}
 \leq\sum_{M\in\Pot_{fin}(\N)} |M| \Bigl(\E\bigl[F_M^2\bigr]\Bigr)^{1/2}\Bigl(\E\bigl[G_M^2\bigr]\Bigr)^{1/2} \\
 &\leq \biggl(\sum_{M\in\Pot_{fin}(\N)} |M|^2 \E\bigl[F_M^2\bigr]\biggr)^{1/2}
 \biggl(\sum_{M\in\Pot_{fin}(\N)}  \E\bigl[G_M^2\bigr]\biggr)^{1/2}=C\norm{G}_2,
\end{align*}
as desired.\smallskip\\
Altough redundant, we also give an argument for the proof of (i)$\Rightarrow$(iv): Suppose that (iv) is false. Then, for $n\in\N$ there are distinct sets
$M_1,\dotsc,M_n\in\Pot_{fin}(\N)$ such that 
\begin{equation*}
 a_n:=\sum_{i=1}^n|M_i|^2\,\E\bigl[F_{M_i}^2\bigr]\geq n.
\end{equation*}
Let $G_n:=\sum_{i=1}^n|M_i|F_{M_i}$ which is in $\mathcal{S}\subseteq\dom(D)$. Then, since $G_n$ has the Hoeffding components 
$G_{n,M_i}:=|M_i|F_{M_i}$, $i=1,\dotsc,n$, by a simple computation, 
\begin{align*}
 &\Bbabs{\sum_{k=1}^\infty \E\bigl[D_kF D_kG\bigr]} =\sum_{i=1}^n|M_i|^2\,\E\bigl[F_{M_i}^2\bigr]=a_n=\E[G_n^2]
\end{align*}
and thus 
\[\frac{\Babs{\sum_{k=1}^\infty \E\bigl[D_kF D_kG\bigr]}}{\norm{G_n}_2}=\frac{a_n}{\sqrt{a_n}}
 =\sqrt{a_n}\geq \sqrt{n}.\]
Thus, $F$ cannot be in $\dom(L)$.\smallskip\\
To prove the additional statement about the representation of $L$ in terms of the infinite Hoeffding decomposition, observe that
\[H:=-\sum_{M\in\Pot_{fin}(\N)}|M|\, F_M =-\sum_{p=1}^\infty p\,F^{(p)}\]
converges (unconditionally) in $L^2(\P)$ by (iv). On the other hand, we know that $-\sum_{k=1}^m D_kF$ converges in $L^2(\P)$ to $LF$ as $m\to\infty$. Moreover, for fixed $m\in\N$, Proposition \ref{HoeffD} and a short computation yield that 
\begin{align*}
&\E\biggl[\Bigl(-\sum_{k=1}^m D_kF -H\Bigr)^2\biggr]=\E\Biggl[\biggl(\sum_{k=1}^m \sum_{\substack{M\in\Pot_{fin}(\N):\\ k\in M}}F_M -\sum_{M\in\Pot_{fin}(\N)}|M|\, F_M\biggr)^2\Biggr]\\
&=\sum_{M\in\Pot_{fin}(\N)}\babs{M\setminus[m]}^2\,\E\bigl[F_M^2\bigr]\,,
\end{align*}
and the right hand side converges to $0$ as $m\to\infty$ by (iv) and the dominated convergence theorem. Hence, $LF=H$.
\end{proof}

\begin{remark}\label{domrem}
\begin{enumerate}[(a)]
\item Propositions \ref{domD}, \ref{domdelta} and \ref{domL} characterize explicitly the domains of the three Malliavin operators $D,\delta$ and $L$ and, hence, significantly extend the theory of \cite[Section 2]{DH}. We remark further that the description of $\dom(L)$ given in \cite[Definition 2.10]{DH} in general is incorrect as, by Proposition \ref{domD}, it actually describes $\dom(D)$.
\item From Propositions \ref{domD} and \ref{domL} it immediately follows that $\dom(L)\subseteq\dom(D)$. Moreover, it is clear from Proposition \ref{domL} that $F\in\dom(L)$, whenever $F\in\bigoplus_{p=0}^n\mathcal{H}_p$ for some $n\in\N$. In particular, if there is an $n\in\N$ such that $X_j$ is $\P$-a.s. constant for each $j\geq n+1$, then we have $\dom(D)=\dom(L)=L^2_\X$ and $\dom(\delta)=L^2(\kappa\otimes\P_\F)$. Recall that this is the case in the situation of functionals of only finitely many independent random variables.
\item Note that the characterizations in parts (iv) of Propositions \ref{domD} and \ref{domL} differ from those in terms of the chaotic decomposition in the Gaussian \cite{Nualart}, Poisson \cite{Lastsa} or Rademacher setting \cite{Priv08} only in that a factor of $p!$ is missing from the sums. The reason for this is that, whereas the Hoeffding decomposition is given by a sum over finite subsets of $\N$, the multiple Wiener-It\^{o} integrals in these three cases involve permutations of the distinct arguments of the kernels. Apart from this conventional difference, the criteria are completely analogous to those in these three standard cases.  
\item Recall the definition of the influence functions $\Inf_k(F)$, $k\in\N$, from Subsection \ref{dejong}. Then, a straight-forward computation yields that $\Inf_k(F)=\E[(D_kF)^2]$ so that Proposition \ref{domD} states that $F\in\dom(D)$ if and only if $\sum_{k=1}^\infty\Inf_k(F)<\infty$. 
\item A (possibly non-symmetric and non-homogeneous) Rademacher sequence is a sequence $\X=(X_n)_{n\in\N}$ of independent and $\{-1,1\}$-valued random variables such that $p_k:=\P(X_k=1)=1-q_k:=1-\P(X=-1)\in(0,1)$ for each $k\in\N$. For Rademacher sequences, the Malliavin operators $\hat{D},\hat{\delta}$ and $\hat{L}$ and the corresponding calculus have existed for several years (see e.g. \cite{Priv08} for their definition and for a comprehensive treatment of the corresponding theory). As has been observed in the introduction of \cite{DH} (for the case of symmetric Rademacher sequences, i.e. $p_k=q_k=1/2$ for all $k\in\N$), the Malliavin derivative $D$ defined above does not reduce to the usual Malliavin $\hat{D}$ for a Rademacher sequence $\X$. We remark here that, actually, the derivatives $D_k$ and $\hat{D}_k$ are yet very similar. Indeed, a simple computation along the formula for $\hat{D}_k$ given in \cite[Proposition 7.3]{Priv08} shows that, in fact, $D_k F=Y_k\hat{D}_kF$ for any square integrable functional $F$ of $\X$. Here, $Y_k=\frac{X_k+q_k-p_k}{2\sqrt{p_kq_k}}$ denotes the normalization of $X_k$, $k\in\N$. In view of Proposition \ref{domD} and \cite[Lemma 2.3]{KRT1}, and since $\hat{D}_kF$ and $Y_k$ are independent and $\E[Y_k^2]=1$, this in particular implies that $\dom(D)=\dom(\hat{D})$. Moreover, the respective Ornstein-Uhlenbeck generators $L$ and $\hat{L}$ are actually exactly the same (compare \cite[Proposition 10.1]{Priv08} and Proposition \ref{domL} and note the different sign convention that $\hat{L}:=\hat{\delta}\hat{D}$ in \cite{Priv08}). 
The similarity between $D_k$ and $\hat{D}_k$ is also reflected in the respective Stroock type formulas (see e.g. \cite[Remark 2.1]{KRT1} for the Rademacher case). Indeed, whereas one has 
$\frac{1}{p!}\E[\hat{D}_{i_1}\ldots \hat{D}_{i_p}F]=f_p(i_1,\dotsc,i_p)$ with $f_p$ the $p$-th symmetric kernel in the (Rademacher) chaotic decomposition $F=\E[F]+\sum_{p=1}^\infty \hat{J}_p(f_p)$, the formula in Corollary \ref{stroock} gives $\frac{1}{p!}\E[D_{i_1}\ldots D_{i_p}F\,|\,X_{i_1},\dotsc,X_{i_p}]=f_p(i_1,\dotsc,i_p)\prod_{l=1}^pY_{i_l}$. These strong similarities with the well-established Rademacher setting sustain our viewpoint that the derivative $D$ introduced by \cite{DH} is the right starting point for a Malliavin structure on general product spaces. 
\item Similarly, the formulae given in Propositions \ref{HoeffD} and \ref{domD} for the action of the $D_k$ and of $L$ in terms of the infinite Hoeffding decomposition as well as the Stroock formula in Corollary \ref{stroock} and the characterizations of $\dom(D)$ and $\dom(L)$ emphasize our point of view that the infinite Hoeffding decomposition from Proposition \ref{infhoeff} is the natural counterpart to the Wiener-It\^{o} chaos expansion for Gaussian, Poisson and Rademacher functionals.  
\end{enumerate}
\end{remark}

From the formula $LF=-\sum_{k=1}^\infty D_kF$ and the symmetry of the $D_k$ one immediately obtains that $L$ is \textit{symmetric}, i.e. that
\[\E\bigl[FLG\bigr]=\E\bigl[GLF\bigr],\quad F,G\in\dom(L).\]
We observe next that $L$ is in fact \textit{self-adjoint}, i.e. that $L=L^*$ with $L^*$ the Hilbert space adjoint of the generally unbounded operator $L$ and completely describe its spectrum. Recall that the \textit{spectrum} $\sigma(L)$ of $L$ is the set of all $\lambda\in\R$ such that the operator $L-\lambda \Id:\dom(L)\rightarrow L^2_\X$ is not bijective and that its \textit{point spectrum} $\sigma_p(L)$ consists of those $\lambda\in\sigma(L)$ such that $L-\lambda \Id:\dom(L)\rightarrow L^2_\X$ is not injective. The elements of $\sigma_p(L)$ are called \textit{eigenvalues} of $L$ and the linear subspace $\ker(L-\lambda\Id)\not=\{0\}$ is the 
\textit{eigenspace} corresponding to $\lambda\in\sigma_p(L)$. 

\begin{prop}\label{saL}
The operator $L:\dom(L)\rightarrow L^2_\X$ is self-adjoint, $-L$ is positive and one has $\sigma(L)=\sigma_p(L)\subseteq-\N_0=\{p\in\Z\,:\,p\leq 0\}$. Moreover, the eigenspace of $L$ corresponding to an eigenvalue $-p\in\sigma_p(L)$ is precisely given by the $p$-th Hoeffding space $\mathcal{H}_p$. 
\end{prop}

\begin{proof}
 Since $\mathcal{S}\subseteq\dom(L)$, $L$ is densely defined and, hence, the symmetry of $L$ and general functional analytic facts imply that $L\subseteq L^*$, that is $L^*$ is an extension of $L$. Regarding the first claim, it thus suffices to show that $\dom(L^*)\subseteq\dom(L)$. 
 Suppose on the contrary that $G\in L^2_\X\setminus\dom(L)$ and denote by $G=\sum_{p=0}^\infty G^{(p)}$ its infinite Hoeffding decomposition. Since $G\notin\dom(L)$, by Proposition \ref{domL} it holds that $\sum_{p=1}^\infty p^2\Var\bigl(G^{(p)}\bigr)=+\infty$. In particular, for any $n\in\N$, we may find $k_n\in\N$ such that 
 \[\sum_{p=1}^{k_n} p^2\Var\bigl(G^{(p)}\bigr)>n.\]
 Define
 \[F_n:=\sum_{p=1}^{k_n} p\,G^{(p)},\quad n\in\N.\]
Since $\sum_{p=1}^{k_n} p^4\Var\bigl(G^{(p)}\bigr)<\infty$,  by Proposition \ref{domL} we have $F_n\in\dom(L)$ for each $n\in\N$. However, from $LF_n=-\sum_{p=1}^{k_n} p^2\,G^{(p)}$ we obtain that
 \begin{align*}
  \babs{\langle LF_n,G\rangle}&=\sum_{p=1}^{k_n}p^2\,\Var\bigl(G^{(p)}\bigr)
  =\norm{F_n}_2^2>\sqrt{n}\norm{F_n}_2,\quad n\in\N.
 \end{align*}
Therefore, it follows from the general definition of $\dom(L^*)$ that $G\notin \dom(L^*)$. Thus, $\dom(L^*)\subseteq\dom(L)$ and, hence, $L=L^*$ as claimed. The positivity of $-L$ follows from the fact that for $F\in\dom(L)$ with infinite Hoeffding decomposition \eqref{genhd2} one has 
\begin{align*}
 \E\bigl[F(-LF)\bigr]=\sum_{p=0}^\infty p\, \E\Bigl[\bigl(F^{(p)}\bigr)^2\Bigr]\geq0,
\end{align*}
where the orthogonality of the $F^{(p)}$ and $L^2(\P)$-convergence have been used.
\smallskip\\
We next show that $\sigma_p(L)\subseteq-\N_0$. By Proposition \ref{domL} we have that $LF=-pF$ holds for any $F\in\mathcal{H}_p$, $p\in\N_0$. Conversely, suppose that $G\in \dom(L)$ satisfies $LG=\lambda G$ for some $\lambda\in\R$ and let $G= \sum_{p=0}^\infty G^{(p)}$ denote its infinite Hoeffding decomposition. Then, from Proposition \ref{domL} we have 
\[\sum_{p=0}^\infty \lambda\,G^{(p)}=\lambda G=LG=-\sum_{p=0}^\infty p\, G^{(p)}\]
and the uniqueness of the Hoeffding decomposition for $LG$ implies that 
$\lambda\,G^{(p)}=-p\,G^{(p)}$ holds $\P$-a.s. for any $p\in\N_0$. This implies that $\Var(G^{(p)})>0$ for at most one $p\in\N_0$ and if this $p$ exists, then $\lambda=-p$. Hence, we have shown that $\sigma_p(L)\subseteq-\N_0$ and that for $-p\in\sigma_p(L)$ the corresponding eigenspace is given by $\mathcal{H}_p$. \smallskip\\
We finally show that $\sigma(L)=\sigma_p(L)$. To this end, it suffices to prove that, for any $\lambda\in\R\setminus\sigma_p(L)$ the operator 
$T_\lambda:=L-\lambda\Id:\dom(L)\rightarrow L^2_\X$ is bijective and that there is a $c_\lambda>0$ such that $\norm{T_\lambda F}_2\geq c_\lambda\norm{F}_2$ for any $F\in\dom(L)$. We have already proved that $T_\lambda$ is injective, so let $G\in L^2_\X$ with the infinite Hoeffding decomposition $G= \sum_{p=0}^\infty G^{(p)}=\sum_{p\in \sigma_p(L)} G^{(p)}$ be given. Note that the second representation for $G$ holds, since $G^{(p)}=0$ $\P$-a.s. for any $p\in-\N_0\setminus\sigma_p(L)$ as $\mathcal{H}_p=\{0\}$ in this case. 
Then, 
\[F:=-\sum_{p\in \sigma_p(L)} \frac{1}{\lambda+p} G^{(p)}\]
is well-defined and contained in $\dom(L)$ by Proposition \ref{domL} and one has $T_\lambda F=G$. Moreover, for any $F\in\dom(L)$  with infinite Hoeffding decomposition $F= \sum_{p\in \sigma_p(L)}  F^{(p)}$ one has 
\begin{align*}
 \norm{T_\lambda F}_2^2=\sum_{p\in \sigma_p(L)} (\lambda+p)^2 \norm{F^{(p)}}_2^2
 \geq d\bigl(\lambda,\sigma_p(L)\bigr)^2\norm{F}_2^2,
\end{align*}
where $d(\lambda,\sigma_p(L)):=\min\{|\lambda-q|\,:\,q\in\sigma_p(L)\}>0$ since $\lambda\notin\sigma_p(L)$ and $\sigma_p(L)\subseteq-\N_0$ . Thus, we can take $c_\lambda=d(\lambda,\sigma_p(L))$.
\end{proof}

\begin{remark}\label{saLrem}
\begin{enumerate}[(a)]
\item In the situation of Proposition \ref{saL} it is not always true that every $-p\in-\N_0$ is an eigenvalue of $L$. For instance, if for some $n\in\N$, the random variables $X_{n+1},X_{n+2},\dotsc$ are $\P$-a.s.constant, then one necessarily has $\mathcal{H}_p=\{0\}$ for each $p\geq n+1$ due to the classical Hoeffding decomposition for $L^2(\P)$-functionals of $X_1,\dotsc,X_n$.
\item Since $\E[LF]=0$ for any $F\in\dom(L)$, no random variable with non-zero mean is contained in the image $\im(L)$ of $L$. Suppose, on the contrary, that $G\in L^2_\X$ satisfies $\E[G]=0$. Then, $G$ has infinite Hoeffding decomposition of the form $G=\sum_{p=1}^\infty G^{(p)}=\sum_{ M\in\Pot_{fin}(\N)}G_M $ and it follows from Proposition \ref{domL} that 
\[F:=-\sum_{p=1}^\infty \frac{1}{p}G^{(p)}=-\sum_{\emptyset\not= M\in\Pot_{fin}(\N)}\frac{1}{|M|}G_M \in\dom(L)\]
 and $LF=G$. Moreover since, by Proposition \ref{saL}, $\ker(L)$ only consists of the constants, $F$ is the only centered random variable in $\dom(L)$ with this property. As usual, we write $F=L^{-1}G$ and call $L^{-1}$ the \textit{pseudo-inverse} of the Ornstein-Uhlenbeck generator $L$. Note that, with this definition, one has 
\begin{itemize}
\item $L L^{-1}G=G$ for each centered $G\in L^2_\X$.
\item $L^{-1} LF=F-\E[F]$ for any $F\in\dom(L)$.
\end{itemize}
\item The Ornstein-Uhlenbeck generator $L$ and its pseudo-inverse $L^{-1}$ have also been analyzed in \cite[Lemma 2.4]{Duer}. There, it has been provided that $L$ is \textit{essentially} self-adjoint, positive, has pure point spectrum $-\N$ (which, in view of part (a) of this remark, is not completely correct) and that its image is dense in the subspace $\{1\}^\perp$ of $L^2_\X$ of centered random variables. Note that our Proposition \ref{saL} strengthens some of these findings, by establishing that $L$ is in fact self-adjoint and we have moreover just seen that the image of $L$ actually coincides with $\{1\}^\perp$.
\end{enumerate}
\end{remark}

\subsection{Carr\'{e}-du-champ operator}\label{cdc}
We finally introduce the \textit{carr\'{e}-du-champ operator} $\Gamma$ associated to the Markov generator $L$. We refer to the monograph \cite{BGL14} for a comprehensive study of (diffusive) Markovian generators via their associated carr\'{e}-du-champ operators. 

For $F,G\in\dom(L)$ such that also $FG\in\dom(L)$ this bilinear operator is defined via 
\begin{equation*}
 \Gamma(F,G):=\frac12\Bigl(L(FG)-GLF-FLG\Bigr)\in L^1(\P)
\end{equation*}
so that the symmetry of $L$ and the fact that $L(FG)$ is centered imply the \textit{carr\'{e}-du-champ integration by parts formula}
\begin{equation}\label{cdcintparts}
 \E\bigl[\Gamma(F,G)\bigr]=-\E\bigl[FLG\bigr]=-\E\bigl[GLF\bigr].
\end{equation}
We remark that, contrary to the situation considered in \cite{BGL14}, the operator $L$ associated to $\X$ is \textit{non-diffusive} in the sense that, for a $C^1$-function $\psi$ on $\R$,
\begin{equation*}
 R_\psi(F,G):=\Gamma(\psi(F),G)-\psi'(F)\Gamma(F,G)\not=0
\end{equation*}
in general.
For the purpose of normal approximation it will be important to control the size 
of the error term $R_\psi(F,G)$. To this end, as in \cite{DP18a} for the Poisson and in \cite{DK19} for the Rademacher case,
we need an alternative representation of $\Gamma(F,G)$ as well as an alternative integration by parts formula. 
For $F,G\in \dom(D)$ with respective representatives $f,g\in L^2(\mu)$ we therefore define
\begin{equation}\label{Gamma0}
\Gamma_0(F,G):=\frac{1}{2}\sum_{k=1}^\infty\E\biggl[\Bigl(f\bigl(\X^{(k)}\bigr)-f\bigl(\X\bigr)\Bigr) \Bigl(g\bigl(\X^{(k)}\bigr)-g\bigl(\X\bigr)\Bigr)\,\biggl|\,\X\biggr].
\end{equation}
Note that it follows from the estimate
\begin{align*}
&\E \Biggl[\sum_{k=1}^\infty\Bbabs{ \E\biggl[\Bigl(f\bigl(\X^{(k)}\bigr)-f\bigl(\X\bigr)\Bigr) \Bigl(g\bigl(\X^{(k)}\bigr)-g\bigl(\X\bigr)\Bigr)\,\biggl|\,\X\biggr]   }\Biggr]\\
&\leq\sum_{k=1}^\infty\E\biggl[\Babs{f\bigl(\X^{(k)}\bigr)-f\bigl(\X\bigr)}\Babs{g\bigl(\X^{(k)}\bigr)-g\bigl(\X\bigr)}\biggr]\\
&\leq\biggl(\sum_{k=1}^\infty\E\Bigl[\Bigl(f\bigl(\X^{(k)}\bigr)-f\bigl(\X\bigr)\Bigr)^2 \Bigr]\biggr)^{1/2}
\biggl(\sum_{k=1}^\infty\E\Bigl[\Bigl(g\bigl(\X^{(k)}\bigr)-g\bigl(\X\bigr)\Bigr)^2 \Bigr]\biggr)^{1/2}
\end{align*}
and from Proposition \ref{domD} that $\Gamma_0(F,G)$ is a well-defined element of $L^1(\P)$ for all $F,G\in\dom(D)$ and that also 
\begin{equation*}
\Gamma_0(F,G)=\frac{1}{2}\E\biggl[\sum_{k=1}^\infty\Bigl(f\bigl(\X^{(k)}\bigr)-f\bigl(\X\bigr)\Bigr) \Bigl(g\bigl(\X^{(k)}\bigr)-g\bigl(\X\bigr)\Bigr)\,\biggl|\,\X\biggr]
\end{equation*}
in this case. The next result shows that $\Gamma_0$ coincides with $\Gamma$, whenever the latter is defined.

\begin{prop}\label{formgamma}
Let $F,G\in\dom(L)$ be such that $FG\in\dom(L)$ as well. 
Then, $\Gamma_0(F,G)\in L^1(\P)$ is well-defined and, in fact, $\Gamma(F,G)=\Gamma_0(F,G)$. 
\end{prop}

\begin{proof}
Since $\dom(L)\subseteq\dom(D)$ by Propositions \ref{domD} and \ref{domL} it is clear that $\Gamma_0(F,G)\in L^1(\P)$ is well-defined. As before, let $f,g\in L^2(\mu)$ be representatives of $F$ and $G$, respectively. By definition, we have (with convergence in $L^1_\X$)
\begin{align*}
2\Gamma(F,G)&=L(FG)-GLF-FLG=\sum_{k=1}^\infty\Bigl(-D_k(FG)+GD_kF+FD_kG\Bigr)\\
&=\sum_{k=1}^\infty\biggl(-\E\bigl[(fg)(\X)-(fg)(\X^{(k)})\,\bigl|\,\X\bigr]+g(\X)\E\bigl[f(\X)-f(\X^{(k)})\,\bigl|\,\X\bigr]\\
&\hspace{3cm}+f(\X)\E\bigl[g(\X)-g(\X^{(k)})\,\bigl|\,\X\bigr]\biggr)\\
&=\sum_{k=1}^\infty\E\biggl[\Bigl(f\bigl(\X^{(k)}\bigr)-f\bigl(\X\bigr)\Bigr) \Bigl(g\bigl(\X^{(k)}\bigr)-g\bigl(\X\bigr)\Bigr)\,\biggl|\,\X\biggr]=2\Gamma_0(F,G),
\end{align*}
where we have used the fact that $fg$ is a representative of $FG$. 
\end{proof}

\begin{prop}[Carr\'{e}-du-champ integration-by-parts]\label{intpartsgamma0}
 Suppose that $F\in\dom(D)$ and $G\in\dom(L)$. Then, 
 \[\E\bigl[FLG\bigr] = -\E\bigl[\Gamma_0(F,G)\bigr].\]
\end{prop}

\begin{proof}
Let $f,g\in L^2(\mu)$ be representatives of $F$ and $G$, respectively. Since $G\in\dom(L)$, the series $-\sum_{k=1}^\infty D_kG$ converges in $L^2(\P)$ to $LG$ and we have
 \begin{align*}
  &\E\bigl[FLG\bigr]=-\sum_{k=1}^\infty\E\bigl[FD_kG\bigr]
  =-\sum_{k=1}^\infty\E\Bigl[f(\X)\E\bigl[g(\X)-g(\X^{(k)})\,\bigl|\,\X\bigr]\Bigr]\\
  &=-\sum_{k=1}^\infty\E\Bigl[\E\bigl[f(\X)\bigl( g(\X)-g(\X^{(k)})\bigr)\,\bigl|\,\X\bigr]\Bigr]
  =-\sum_{k=1}^\infty\E\Bigl[f(\X)\bigl( g(\X)-g(\X^{(k)})\bigr)\Bigr].
 \end{align*}
Now, for each $k\in\N$, $(\X,\X^{(k)})$ has the same distribution as $(\X^{(k)},\X)$ so that we also have 
\begin{align*}
  \E\bigl[FLG\bigr]&=\sum_{k=1}^\infty\E\Bigl[f(\X^{(k)})\bigl( g(\X)-g(\X^{(k)})\bigr)\Bigr]
\end{align*}
implying
\begin{align*}
 &\E\bigl[FLG\bigr]=-\frac12\sum_{k=1}^\infty\E\Bigl[\bigl(f(\X)-f(\X^{(k)})\bigr)\bigl( g(\X)-g(\X^{(k)})\bigr)\Bigr]  \\
&= -\frac12\sum_{k=1}^\infty\E\Bigl[\E\bigl[\bigl(f(\X)-f(\X^{(k)})\bigr)\bigl( g(\X)-g(\X^{(k)})\bigr)\,\bigl|\,\X\bigr]\Bigr]  
 =-\E\bigl[\Gamma_0(F,G)\bigr].
\end{align*}
Note that for the last identity we have used the fact that $F,G\in\dom(D)$ in order to interchange the expectation and the infinite sum. 
\end{proof}

The carr\'{e}-du-champ operator $\Gamma$ in general is closely related to the associated \textit{Dirichlet form} $\mathcal{E}$ (see e.g. \cite[Section 1.7]{BGL14}). We further refer to \cite{BH} for a comprehensive treatment of (abstract and concrete) Dirichlet forms. For $F,G\in\dom(L)$ this symmetric bilinear form is usually defined by 
\begin{equation*}
\mathcal{E}(F,G):=-\E\bigl[FLG\bigr]=-\E\bigl[GLF\bigr],
\end{equation*}
where we have used the symmetry of $L$. This formula is also given in \cite[Section 4]{DH}. If, additionally, $FG\in\dom(L)$, then by \eqref{cdcintparts} it also holds that 
\begin{equation*}
\mathcal{E}(F,G)=\E\bigl[\Gamma(F,G)\bigr].
\end{equation*} 
Now, since $\Gamma_0$ extends $\Gamma$, and thanks to Proposition \ref{intpartsgamma0} we can extend the definition of $\mathcal{E}$ by letting 
\begin{equation*}
\mathcal{E}(F,G):=\E\bigl[\Gamma_0(F,G)\bigr]=\frac{1}{2}\sum_{k=1}^\infty\E\biggl[\Bigl(f\bigl(\X^{(k)}\bigr)-f\bigl(\X\bigr)\Bigr) \Bigl(g\bigl(\X^{(k)}\bigr)-g\bigl(\X\bigr)\Bigr)\biggr]
\end{equation*} 
for all $F,G\in\dom(D)\supseteq\dom(L)$. Note that, in \cite[Definition 4.1]{DH}, the alternative definition 
\begin{equation*}
\mathcal{E}(F,G)=\langle DF,DG\rangle_{L^2(\kappa\otimes\P_\F)}=\sum_{k=1}^\infty \E\bigl[D_kF D_kG\bigr], \quad F,G\in\dom(D),
\end{equation*}
has been given. However, since a straightforward generalization of the computation leading to \eqref{eqdomd} shows that
\begin{align}\label{ES}
\E\bigl[D_kF D_kG\bigr]&= \frac12\E\biggl[\Bigl(f\bigl(\X^{(k)}\bigr)-f\bigl(\X\bigr)\Bigr) \Bigl(g\bigl(\X^{(k)}\bigr)-g\bigl(\X\bigr)\Bigr)\biggr],\quad k\in\N,
\end{align}
 these two definitions actually coincide. Moreover, in \cite[Corollary 3.5]{DH} the important \textit{Poincar\'{e} inequality} 
\begin{equation}\label{poincare}
\Var(F)\leq \mathcal{E}(F,F)=\sum_{k=1}^\infty\E\bigl[(D_kF)^2\bigr], \quad F\in\dom(D),
\end{equation}
has been provided, which, thanks to \eqref{ES}, may thus now also be written as 
\begin{equation*}
\Var(F)\leq \frac{1}{2}\sum_{k=1}^\infty\E\biggl[\Bigl(f\bigl(\X^{(k)}\bigr)-f\bigl(\X\bigr)\Bigr)^2 \biggr], \quad F\in\dom(D).
\end{equation*}
As has been remarked in \cite{DH}, this inequality is an infinite version of the celebrated \textit{Efron-Stein inequality} (see e.g. \cite{EfSt,Steele,BLM}). It further trivially continues to hold for all $F\in L^2_\X\setminus \dom(D)$ since, in this case, the right hand side equals $+\infty$ by Proposition \ref{domD}. The inequality \eqref{poincare} has actually implicitly been established in the proof of Proposition \ref{domD}, since therein we have seen that, for $F\in\dom(D)$ with the infinite Hoeffding decomposition \eqref{genhd2}, we have 
\begin{align}\label{devpoincare}
\sum_{k=1}^\infty\E\bigl[(D_kF)^2\bigr]=\sum_{M\in\Pot_{fin}(\N)}|M|\,\E\bigl[F_M^2\bigr] 
\end{align}
which is of course not smaller than 
\[\sum_{\emptyset\not=M\in\Pot_{fin}(\N)}\,\E\bigl[F_M^2\bigr]=\Var(F).\]
Note that this argument also indicates how far the inequality is from being an equality. Note that, for $F\in\bigoplus_{p=0}^m\mathcal{H}_p$, from \eqref{poincare} and \eqref{devpoincare} we obtain the chain of inequalities 
\begin{align}\label{poincare2}
\Var(F)=\sum_{k=1}^\infty\E\bigl[(D_kF)^2\bigr]\leq m\Var(F).
\end{align}

\subsection{Covariance formulae}\label{covid}
The adaptation of Stein's method for normal approximation to our framework relies on the fact that several useful covariance identities can be established, which in turn lead to fruitful new \textit{Stein identities} (see e.g. \cite{CGS}). In this subsection we provide three such formulae.

The first one, the Malliavin type covariance formula, has implicitly been used in \cite{DH} and a very similar formula has also been given in \cite[Lemma 2.5]{Duer}. We state it here for the sake of later reference and also provide its short proof.
\begin{prop}\label{mallcovid}
For all $F\in L^2_\X$ and $G\in\dom(D)$ one has
\begin{equation*}
\Cov(F,G)=-\E\Bigl[\bigl\langle D L^{-1}\bigl(F-\E[F]\bigr),DG\bigr\rangle_{L^2(\kappa\otimes\P_\F)}\Bigr].
\end{equation*}
\end{prop}

\begin{proof}
Since $L^{-1}(F-\E[F])\in\dom(L)\subseteq\dom(D)$, using $\delta D=-L$ and \eqref{intpartsM} we have
\begin{align*}
\Cov(F,G)&=\E\bigl[(F-\E[F]) G\bigr]=\E\bigl[LL^{-1}(F-\E[F]) G\bigr]\\
&=-\E\bigl[\delta DL^{-1}((F-\E[F]) G\bigr]
=-\E\Bigl[\bigl\langle D L^{-1}\bigl(F-\E[F]\bigr),DG\bigr\rangle_{L^2(\kappa\otimes\P_\F)}\Bigr].
\end{align*}
\end{proof}

The first equality in the next covariance formula, which is of the Clark-Ocone type, is Theorem 3.6 in \cite{DH}.
\begin{prop}\label{clarkcovid}
For all $F,G\in \dom(D)$ one has
\begin{equation*}
\Cov(F,G)=\E\Biggl[\sum_{k=1}^\infty D_k\E[F|\F_k]\, D_kG  \Biggr]=\sum_{k=1}^\infty \E\Bigl[ D_k\E[F|\F_k]\, D_kG  \Bigr] .
\end{equation*}
\end{prop}

Note that the second equality in Proposition \ref{clarkcovid} holds true as $F,G\in\dom(D)$ and, thus,  
\begin{align}\label{clarkex}
\sum_{k=1}^\infty \E\Bigl[\bigl( D_k\E[F|\F_k]\bigr)^2\Bigr]&=\sum_{k=1}^\infty \E\Bigl[\bigl(\E[ D_kF|\F_k]\bigr)^2\Bigr]
\leq\sum_{k=1}^\infty \E\Bigl[\E\bigl[( D_kF)^2|\F_k\bigr]\Bigr]\notag\\
&=\sum_{k=1}^\infty \E\bigl[( D_kF)^2\bigr]<\infty,
\end{align}
where the final inequality is by Proposition \ref{domD}. Hence, by the Cauchy-Schwarz inequality, it is justified to interchange the expectation and the infinite sum.

Finally, we present a novel covariance formula that makes use of the carr\'{e}-du-champ integration by parts formula in Proposition \ref{intpartsgamma0}. It can be considered an infinite analogue of the covariance formula from \cite[Lemma 2.5]{D23} in the context of non-linear exchangeable pairs.
\begin{prop}\label{cdccovid}
For all $F\in L^2_\X$ and $G\in \dom(D)$ one has
\begin{equation*}
\Cov(F,G)=\E\biggl[\Gamma_0\Bigl(-L^{-1}\bigl(F-\E[F]\bigr),G\Bigr)\biggr].
\end{equation*}
\end{prop}

\begin{proof}
Since $G\in \dom(D)$ and $L^{-1}(F-\E[F])\in \dom(L)$, by Proposition \ref{intpartsgamma0} we have
\begin{align*}
\Cov(F,G)&=\E\bigl[(F-\E[F]) G\bigr]=\E\bigl[LL^{-1}(F-\E[F]) G\bigr]\\
&=-\E\Bigl[\Gamma_0\bigl(L^{-1}\bigl(F-\E[F]\bigr),G\bigr)\Bigr]=\E\Bigl[\Gamma_0\bigl(-L^{-1}\bigl(F-\E[F]\bigr),G\bigr)\Bigr].
\end{align*}
\end{proof}

\begin{remark}
At the end of this section on Malliavin calculus and infinite Hoeffding decompositions we make a final comment regarding the generality of the theory initiated in \cite{DH} and extended in the present work. Suppose that $I$ is in fact an uncountable index set and that, for each $i\in I$, $(E_i,\B_i,\mu_i)$ is a probability space. Denote by 
$(E,\B,\mu):=(\prod_{i\in I} E_i,\bigotimes_{i\in I}\B_i,\bigotimes_{i\in I}\mu_i)$ the corresponding product probability space and by $Y_L:(E,\B)\rightarrow (\prod_{i\in L}E_i,\bigotimes_{i\in L}\B_i)$ the canonical projections, $L\subseteq I$. From the well-known fact that any $B\in\B$ is already contained in 
$\sigma(Y_{J_B})$ for some countable $J_B\subseteq I$ and since the Borel-$\sigma$-field $\B(\R)$ on $\R$ is countably generated, 
one can infer that any measurable $f:(E,\B)\rightarrow(\R,\B(\R))$ is already measurable with respect to $\sigma(Y_{K})$ for some countable $K\subseteq I$. Hence, as long as only countably many such functionals $f$ are considered, the theory presented here in principle covers the most general case of functionals on products of arbitrarily many probability spaces.
\end{remark}

\section{Normal approximation via Malliavin and carr\'{e}-du-champ operators}\label{normapp} 

\subsection{Preparations}\label{preps}

Recall that a function $\psi:\R\rightarrow\R$ is Lipschitz-continuous, if and only if 
\[\sup_{x\not=y}\frac{\abs{\psi(y)-\psi(x)}}{\abs{y-x}}<\infty\]
and in this case the quantity on the left hand side equals both the smallest possible Lipschitz-constant for $\psi$ and also the essential supremum norm $\norm{\psi'}_\infty$ of the $\la$-a.e. existing derivative $\psi'$ of $\psi$, where $\lambda$ denotes the Lebesgue measure on $\R$.\smallskip\\

The proofs of our error bounds make use of \textit{Stein's method for normal approximation}. We briefly review some facts that are relevant for what follows and otherwise refer to \cite{CGS} for a comprehensive introduction to this topic. 

For a Borel-measurable function $h$ on $\R$ such that $\E|h(Z)|<\infty$ we denote by $\psi_h$ the particular solution to the \textit{Stein equation} 
 \begin{equation}\label{steineq}
  \psi'(x)-x\psi(x)=h(x)-\E[h(Z)]
 \end{equation}
that is given by 
\begin{equation}\label{steinsol}
 \psi_h(x)=e^{x^2/2}\int_{-\infty}^x\bigl(h(t)-\E[h(Z)]\bigr)e^{-t^2/2}dt\,.
\end{equation} 
It follows from standard facts of integration theory that $\psi_h$ is indeed $\la$-a.e. differentiable and solves \eqref{steineq} at its points of differentiability. Now, at the remaining points of $\R$, and contrary to the usual convention, we define $\psi_h'$ in such a way as to satisfy \eqref{steineq} pointwise on $\R$, i.e. we let $\psi_h'(x):= x\psi(x)+h(x)-\E[h(Z)]$, whenever $\psi_h$ is not differentiable at $x$. If $h=h_z=\1_{(-\infty,z]}$ for some $z\in\R$, then we also write $\psi_z$ for $\psi_{h_z}$ for short. The following result gathers important properties of the solutions $\psi_h$.
For proofs of these facts we refer the reader to Lemmas 2.3 and 2.4 in \cite{CGS}.

\begin{lemma}\label{solprops}
For a Borel-measurable function $h$ on $\R$ such that $\E|h(Z)|<\infty$ let $\psi_h$ and $\psi_h'$ be given as above. 
\begin{enumerate}[{\normalfont (i)}]
\item If $h\in \Lip(1)$, then $\psi_h\in C^1(\R)$ and both $\psi_h$ and $\psi_h'$ are Lipschitz-continuous with respective Lipschitz constants 
\begin{equation}\label{sfac}
 \norm{\psi_h'}_\infty\leq \sqrt{\frac{2}{\pi}}\quad\text{and}\quad \norm{\psi_h''}_\infty\leq 2\,.
\end{equation}
\item If $h=h_z=\1_{(-\infty,z]}$ for some $z\in\R$, then $\psi_z$ is Lipschitz-continuous on $\R$ and infinitely differentiable on $\R\setminus\{z\}$. Moreover, $\psi_z$ and $\psi_z'$ have the following further properties.
\begin{enumerate}[{\normalfont (a)}] 
\item The function $x\mapsto x\psi_z(x)$ is increasing on $\R$.
\item $|x\psi_z(x)|\leq 1$ and $|x\psi_z(x)-y\psi_z(y)|\leq1$ for all $x,y\in\R$.
\item $\norm{\psi_z'}_\infty\leq1$ and $\norm{\psi_z}_\infty\leq \sqrt{2\pi}/4$.
\end{enumerate}
\end{enumerate}
\end{lemma} 

We next provide a formula that is a substitute for the exact chain rule formula in Gaussian Malliavin calculus \cite{Nualart} which cannot hold here as $D_k$ is not a derivation. This approximate chain rule formula is key to proving our normal approximation bounds in Theorems \ref{msbound} and \ref{cobound}.  

\begin{prop}[Approximate chain rule]\label{chain}
Suppose that $F\in L^2_\X$ and that $h:\R\rightarrow\R$ is either in $\Lip(1)$ or $h=h_z=\1_{(-\infty,z]}$ for some $z\in\R$. Moreover, let $\psi_h$ and $\psi_h'$ be given as above. Then, for $k\in\N$, 
 \begin{equation*}
  D_k\psi_h(F)=\psi_h'(F) D_kF -S_k-\E[R_k\,|\,\G_k],
 \end{equation*}
where the random variables $S_k$ and $R_k$ are defined by the equations
\begin{align}
 \psi_h(F)-\psi_h\bigl(\E[F\,|\,\G_k]\bigr)&=\psi_h'\bigl(\E[F\,|\,\G_k]\bigr) D_kF +R_k\label{cr1}\quad\text{and}\\
 \psi_h\bigl(\E[F\,|\,\G_k]\bigr)-\psi_h(F)&=\psi_h'(F)(-D_kF) +S_k\label{cr2},
\end{align}
respectively, and have the following properties.
\begin{enumerate}[{\normalfont (i)}]
\item If $h\in\Lip(1)$, then
\begin{align*}
 S_k&=\bigl(D_kF\bigr)^2\int_0^1(1-u)\psi_h''\bigl(u\E[F\,|\,\G_k]+(1-u)F\bigr)du,\\
 R_k&=\bigl(D_kF\bigr)^2\int_0^1(1-u)\psi_h''\bigl(uF+(1-u)\E[F\,|\,\G_k]\bigr)du.
\end{align*}
Moreover, one then has the bounds 
\begin{equation}\label{RkSk}
 \max\bigl(|R_k|,|S_k|\bigr)\leq\min\bigl((D_kF)^2,2|D_kF|\bigr)\,,\quad k\in\N\,.
\end{equation}
\item If $h=h_z=\1_{(-\infty,z]}$ for some $z\in\R$, then 
\begin{align}
 \max\bigl(|R_k|,|S_k|\bigr)&\leq 2|D_kF|\quad\text{and} \label{Rk}\\
S_k&=\int_0^{D_kF}\Bigl(\psi_z'(F)-\psi_z'\bigl(\E[F\,|\,\G_k]+t\bigr)\Bigr)dt,\quad k\in\N. \label{Sk}
\end{align}

\end{enumerate}
\end{prop}

\begin{proof}
In this proof we write $F_k:=\E[F\,|\,\G_k]$ in such a way that $F=F_k+D_kF$, $k\in\N$, as well as $\psi:=\psi_h$.
Hence, using \eqref{cr1} and \eqref{cr2}, the boundedness of $\psi'$ (by Lemma \ref{solprops}), that $\E\bigl[D_kF \,\bigl|\,\G_k\bigr]=0$ and that $F_k$ is $\G_k$-measurable we have 
\begin{align*}
 D_k\psi(F)&=\psi(F)-\E\bigl[\psi(F)\,\bigl|\,\G_k\bigr]
 =\psi(F)- \E\bigl[\psi(F_k) +  \psi'(F_k) D_kF +R_k \,\bigl|\,\G_k\bigr]\\
 &=\psi(F)-\psi(F_k)-\psi'(F_k)\E\bigl[D_kF \,\bigl|\,\G_k\bigr]-\E\bigl[R_k \,\bigl|\,\G_k\bigr]\\
 &=\psi(F)-\psi(F_k)-\E\bigl[R_k \,\bigl|\,\G_k\bigr]\\
 &=\psi'(F) D_kF -S_k-\E\bigl[R_k \,\bigl|\,\G_k\bigr]
 \end{align*}
and it remains to prove the respective bounds and formulae for $R_k$ and $S_k$. In the situation of (i) we make use of (a suitable version of) Taylor's formula
\begin{align*}
 \psi(x+h)&=\psi(x)+\psi'(x)h + h^2\int_0^1(1-u)\psi''(x+uh)du
\end{align*}
which, in view of Lemma \ref{solprops} (i), immediately yields the claimed formulae for $R_k$ and $S_k$ as well as the bound $\max\bigl(|R_k|,|S_k|\bigr)\leq(D_kF)^2$. Next observe that, since by Lemma \ref{solprops}, in both situations (i) and (ii), $\psi_h$ is Lipschitz-continuous with Lipschitz-constant one, from \eqref{cr1} and \eqref{cr2} we have 
\begin{align*}
|R_k|&=\babs{\psi_h(F)-\psi_h(F_k)-\psi_h'(F_k)D_kF}\leq 2\norm{\psi_h'}_\infty |D_kF|\leq 2|D_kF|\quad \text{and}\\
|S_k|&=\babs{\psi_h(F)-\psi_h(F_k)-\psi_h'(F)D_kF}\leq 2\norm{\psi_h'}_\infty |D_kF|\leq 2|D_kF|.
\end{align*} 
Hence, we have proved the bounds \eqref{RkSk} and \eqref{Rk}. Finally, in the situation of (ii), the formula \eqref{Sk} for $S_k$ follows from \eqref{cr2}, the fundamental theorem of calculus and the Lipschitzness of $\psi_z$. 
\end{proof}

\begin{remark}\label{chainrem}
Note that the proof of Proposition \ref{chain} significantly differs from the proofs of corresponding chain rule formuals in the Poisson \cite{PSTU} and Rademacher \cite{Zheng} settings. The reason is that, contrary to these situations, the derivative $D_k$ here is not defined by alternating the argument of the functional but via a conditional expectation. This is why a two-step Taylor approximation is necessary in its proof. Interestingly, thanks to its $\G_k$-measurability, the final term $\E[R_k|\G_k]$ consistently does not affect the bounds on normal approximation in the next subsection.    
\end{remark}

In order to apply the covariance formulas from Subsection \ref{covid}, we will also need the following result, whose proof is now simple thanks to Proposition \ref{domD}.

\begin{prop}\label{lipprop}
Let $\psi:\R\rightarrow\R$ be Lipschitz-continuous and $F\in\dom(D)$. Then, $\psi(F)\in\dom(D)$.
\end{prop}

\begin{proof}
Let $K\in (0,\infty)$ be a Lipschitz constant for $\psi$.
The claim  follows immediately from Proposition \ref{domD}, since $\psi\circ f$ is a representative of $\psi(F)$, whenever  $f$ is a representative of $F$, and 
\[\babs{\psi\bigl(f(\X^{(k)})\bigr)-\psi\bigl(f(\X)\bigr)}^2\leq K^2 \babs{f(\X^{(k)})-f(\X)}^2\]
for any $k\in\N$.
\end{proof}

\subsection{Abstract bounds on normal approximation}\label{bounds}
In this subsection we provide three types of abstract Wasserstein and Berry-Esseen bounds on the normal approximation of certain $F\in L^2_\X$ that involve the operators $D,L^{-1}$ and $\Gamma_0$. These bounds are perfect counterparts to similar bounds in Gaussian, Poisson and Rademacher situations that we will refer to below. Our proofs combine the theory from Section \ref{malliavin} and Subsection \ref{preps} with the above facts from Stein's method of normal approximation. In particular, the approximate chain rule from Proposition \ref{chain} will play a key role in the proofs of two of our theorems. 

We begin with the Malliavin-Stein bounds which make use of the covariance formula in Proposition \ref{mallcovid}. The Wasserstein bound \eqref{mswass} looks very similar to the bound in \cite[Theorem 3.1]{PSTU} on the Poisson space and to the bound in \cite[Theorem 3.1]{Zheng} for Rademacher functionals. As we explain below, it is an improvement of the bound provided in Theorem 5.9 of \cite{DH} and of the very similar bound proved on page 629 of \cite{Duer}.

\begin{theorem}[Malliavin-Stein bounds]\label{msbound}
Suppose that $F\in\dom(D)$ is such that $\E[F]=0$. Then, we have the following bounds:
\begin{align}
 d_\W(F,Z)&\leq \sqrt{\frac{2}{\pi}}\E\Bbabs{1-\sum_{k=1}^\infty D_k(-L^{-1}F)D_kF}\notag\\
&\;+\sum_{k=1}^\infty\E\Bigl[\babs{D_k(-L^{-1}F)}\bigl(D_kF\bigr)^2\Bigr]\quad\text{and}\label{mswass}\\
d_\K(F,Z)&\leq \E\Bbabs{1-\sum_{k=1}^\infty D_k(-L^{-1}F)D_kF}+
2\sum_{k=1}^\infty\E\Babs{ D_k\babs{D_kL^{-1}F}  D_kF}\notag\\
&\;+2\E\Bbabs{\sum_{k=1}^\infty\E\bigl[\abs{D_kL^{-1}F}\,\big|\,\G_k\bigr]D_kF}.\label{mskol}
\end{align}
\end{theorem}

\begin{remark}\label{msboundrem}
\begin{enumerate}[(a)]
\item  If, additionally, $\Var(F)=\E[F^2]=1$, then by Proposition \ref{mallcovid} the respective first terms in the above bounds may be estimated by means of the inequality 
\[\E\Bbabs{1-\sum_{k=1}^\infty D_k(-L^{-1}F)D_kF}\leq \biggl(\Var\Bigl(\sum_{k=1}^\infty D_k(-L^{-1}F)D_kF\Bigr)\biggr)^{1/2}.\]
\item The proof will show that the second term in the bound may be replaced with the (smaller) term
\[\sup_{z\in\R}\Bbabs{\sum_{k=1}^\infty\E\biggl[ D_k\babs{D_kL^{-1}F}  D_kF \Bigl(\chi_z(F)-\chi_z\bigl(\E[F\,|\,\G_k]\bigr)  \Bigr)\biggr]},\]
where $\chi_z(x):=x\psi_z(x)+\1_{(z,\infty)}(x)$ is an increasing function of $x\in\R$ by Lemma \ref{solprops}. 
\end{enumerate}
\end{remark}

\begin{proof}[Proof of Theorem \ref{msbound}]
Let $h$ be either in $ \Lip(1)$ or $h=h_z=\1_{(-\infty,z]}$ for some $z\in\R$ and denote by $\psi=\psi_h$ the solution to the Stein equation \eqref{steineq} given by \eqref{steinsol}.
Since $\psi_h$ is Lipschitz by Lemma \ref{solprops}, it follows from Proposition \ref{lipprop} that $\psi_h(F)\in\dom(D)$. Thus, since $F$ is centered, by Proposition \ref{mallcovid} we have 
\begin{align}\label{msb0}
 \E\bigl[F\psi_h(F)\bigr]&=\Cov\bigl(F,\psi_h(F)\bigr)=\E\Bigl[\sum_{k=1}^\infty D_k(-L^{-1}F)D_k\psi_h(F)\Bigr]\notag\\
& =\sum_{k=1}^\infty\E\Bigl[ D_k(-L^{-1}F)D_k\psi_h(F)\Bigr].
\end{align}
Now, for $k\in\N$, by Proposition \ref{chain} we can write (with $R_k$ and $S_k$ defined therein)
\begin{align*}
 D_k\psi_h(F)=\psi_h'(F)D_kF-S_k-\E[R_k|\G_k], \quad k\in\N,
\end{align*}
implying 
\begin{align*}
 \E\bigl[F\psi_h(F)\bigr]&=\sum_{k=1}^\infty \E\Bigl[\psi_h'(F)D_k(-L^{-1}F)D_kF\Bigr]
 -\sum_{k=1}^\infty\E\Bigl[ D_k(-L^{-1}F)S_k\Bigr]\\
 &\;-\sum_{k=1}^\infty \E\Bigl[D_k(-L^{-1}F)\E[R_k|\G_k]\Bigr]=:T_1+T_2+T_3.
 \end{align*}
We still need to show that splitting up the infinite sum on the right hand side of \eqref{msb0} into the three terms $T_1, T_2$ and $T_3$ is justified. Since $\psi_h'$ is bounded by Lemma \ref{solprops}, $F\in\dom(D)$ and $L^{-1}F\in\dom(L)\subseteq\dom(D)$, it follows by an application of the Cauchy-Schwarz inequality that $|T_1|\leq\sum_{k=1}^\infty \E\babs{\psi_h'(F)D_k(-L^{-1}F)D_kF} <\infty$. 
That $|T_2|\leq\sum_{k=1}^\infty\E\babs{ D_k(-L^{-1}F)S_k}<\infty$ follows similarly from the bounds \eqref{RkSk} and \eqref{Rk} and again from $F\in\dom(D)$, $L^{-1}F\in\dom(D)$. 
Finally, again from $F\in\dom(D)$, $L^{-1}F\in\dom(D)$ and from the bounds \eqref{RkSk} and \eqref{Rk}, it follows by an application of the Cauchy-Schwarz and the conditional Jensen inequalities that 
\begin{align}\label{Rkfin}
|T_3|&\leq\sum_{k=1}^\infty \E\babs{D_k(-L^{-1}F)\E[R_k|\G_k]}\leq \biggl(\sum_{k=1}^\infty \E\babs{D_k(-L^{-1}F)}^2\biggr)^{1/2}  \biggl(\sum_{k=1}^\infty \babs{\E[R_k|\G_k]}^2\biggr)^{1/2}\notag\\
&\leq2 \biggl(\sum_{k=1}^\infty \E\babs{D_k(-L^{-1}F)}^2\biggr)^{1/2} \biggl(\sum_{k=1}^\infty \E\babs{D_kF}^2\biggr)^{1/2} <\infty.
\end{align}   
 Next, we observe that, in fact,  
 \begin{align*}
  T_3&=-\sum_{k=1}^\infty \E\Bigl[D_k(-L^{-1}F)\E[R_k|\G_k]\Bigr]
  =\sum_{k=1}^\infty \E\Bigl[\E\bigl[D_kL^{-1}F\,\big|\,\G_k\bigr]\E\bigl[R_k\big|\G_k\bigr]\Bigr]=0
 \end{align*}
 as $\E\bigl[D_kL^{-1}F\,\big|\,\G_k\bigr]=0$. Note that, here, we have used the fact that\\ $\E|D_k(L^{-1}F) \E[R_k|\G_k]|<\infty$ for each $k\in\N$, which immediately follows from \eqref{Rkfin}.
\smallskip\\

Now, we first continue with the proof of the Wasserstein bound \eqref{mswass}. If $h\in \Lip(1)$, then from \eqref{sfac} and \eqref{RkSk} we have that 
 \begin{align*}
  |T_2|&\leq \frac{\norm{\psi_h''}_\infty}{2}\sum_{k=1}^\infty \E\Bigl[\bigl(D_kF\bigr)^2\babs{D_k(-L^{-1}F)}\Bigr]\leq \sum_{k=1}^\infty \E\Bigl[\bigl(D_kF\bigr)^2\babs{D_k(-L^{-1}F)}\Bigr].
 \end{align*}
Now, noting that, by Fubini's theorem, since $F, L^{-1}F\in\dom(D)$ and since $\norm{\psi_h'}_\infty<\infty$,
\[T_1= \E\Bigl[\psi_h'(F)\sum_{k=1}^\infty D_k(-L^{-1}F)D_kF\Bigr]\]
 we obtain that 
\begin{align*}
 &\babs{\E[h(Z)]-\E[h(F)]}=\babs{\E\bigl[\psi_h'(F)-F\psi_h(F)\bigr]}\\
& \leq \E\Bbabs{\psi'(F)\Bigl(1-\sum_{k=1}^\infty D_k(-L^{-1}F)D_kF\Bigr)}+|T_2|\\
&\leq \sqrt{\frac{2}{\pi}}\E\Babs{1-\sum_{k=1}^\infty D_k(-L^{-1}F)D_kF}
+\sum_{k=1}^\infty\E\Bigl[\babs{D_k(-L^{-1}F)}\bigl(D_kF\bigr)^2\Bigr],
\end{align*}
where we have used \eqref{sfac} again for the final inequality. Since the right hand side of this bound does not depend on $h$, the bound \eqref{mswass} follows by taking the supremum over all $h\in\Lip(1)$.\smallskip\\

To continue the proof of the Berry-Esseen bound \eqref{mskol} first note that from the above reasoning with $\psi_z=\psi_h$ we already have established that
\begin{align}\label{msb1}
&\babs{\P(Z\leq z)-\P(F\leq z)}=\babs{\E\bigl[\psi_z'(F)-F\psi_z(F)\bigr]}\notag\\
& \leq \E\Babs{1-\sum_{k=1}^\infty D_k(-L^{-1}F)D_kF}+|T_2|,
\end{align} 
where we have applied the bound on $\norm{\psi_z'}_\infty$ in Lemma \ref{solprops} (ii), (c) this time. It thus remains to bound $|T_2|$. Using \eqref{Sk}, the fact that $\psi_z$ solves \eqref{steineq} pointwise and writing $F_k:=\E[F\,|\,\G_k]$, $k\in\N$, we have 
\begin{align*}
|T_2|&=\Babs{\sum_{k=1}^\infty\E\Bigl[ D_kL^{-1}F S_k\Bigr]}=\BBabs{\sum_{k=1}^\infty\E\Biggl[ D_kL^{-1}F \int_0^{D_kF}\Bigl(\psi_z'(F)-\psi_z'\bigl(F_k+t\bigr)\Bigr)dt \Biggr] }\notag\\
&\leq \sum_{k=1}^\infty\E\Biggl[ \Babs{D_kL^{-1}F}\Bbabs{ \int_0^{D_kF}\Bigl(F\psi_z(F)-(F_k+t)\psi_z\bigl(F_k+t\bigr)\Bigr)dt} \Biggr]\notag\\
&\;+ \sum_{k=1}^\infty\E\Biggl[\Babs{ D_kL^{-1}F} \Bbabs{\int_0^{D_kF}\Bigl(\1_{\{F\leq z\}}-\1_{\{F_k+t\leq z\}}\Bigr)dt } \Biggr]=:T_{2,1}+T_{2,2}.
\end{align*}
Since $x\mapsto x\psi_z(x)$ is increasing on $\R$ by Lemma \ref{solprops} (ii), (a), by distinguishing the cases $D_kF\geq0$ and $D_kF<0$, for $k\in\N$ we have 
\begin{align*}
0\leq I_{k,1}:=\int_0^{D_kF}\Bigl(F\psi_z(F)-(F_k+t)\psi_z\bigl(F_k+t\bigr)\Bigr)dt\leq D_kF \Bigl(F\psi_z(F)-F_k\psi_z\bigl(F_k\bigr)\Bigr).
\end{align*}
Similarly, since $x\mapsto \1_{\{x\leq z\}}$ is non-increasing on $\R$ we have 
\begin{align*}
0&\geq I_{k,2}:=\int_0^{D_kF}\Bigl(\1_{\{F\leq z\}}-\1_{\{F_k+t\leq z\}}\Bigr)dt \geq D_kF \Bigl(\1_{\{F\leq z\}}-\1_{\{F_k\leq z\}}\Bigr)\notag\\
&=-D_kF \Bigl(\1_{\{F>z\}}-\1_{\{F_k>z\}}\Bigr)
\end{align*}
and, thus, 
\begin{align*}
|I_{k,2}|&=-I_{k,2}\leq D_kF \Bigl(\1_{\{F>z\}}-\1_{\{F_k>z\}}\Bigr),\quad k\in\N.
\end{align*}
Hence, we obtain
\begin{align}\label{msb2}
|T_2|&\leq T_{2,1}+T_{2,2}\leq \sum_{k=1}^\infty\E\biggl[ \babs{D_kL^{-1}F} \Bigl(|I_{k,1}|+|I_{k,2}|\Bigr)\biggr]\notag\\
&\leq \sum_{k=1}^\infty\E\biggl[ \babs{D_kL^{-1}F}  D_kF \Bigl(F\psi_z(F)+\1_{\{F>z\}} -F_k\psi_z\bigl(F_k\bigr)- \1_{\{F_k>z\}}  \Bigr)\biggr]\notag\\
&=\sum_{k=1}^\infty\E\biggl[ D_k\babs{D_kL^{-1}F}  D_kF \Bigl(F\psi_z(F)+\1_{\{F>z\}} -F_k\psi_z\bigl(F_k\bigr)- \1_{\{F_k>z\}}  \Bigr)\biggr]\notag\\
&\;+\sum_{k=1}^\infty\E\biggl[\E\Bigl[\babs{D_kL^{-1}F}\,\Big|\,\G_k\Bigr]  (D_kF) D_k\Bigl(F\psi_z(F)+\1_{\{F>z\}}\Bigr)\biggr]\notag\\
&\leq 2\sum_{k=1}^\infty\E\Babs{ D_k\babs{D_kL^{-1}F}  D_kF}\notag\\
&\;+\BBabs{\sum_{k=1}^\infty\E\biggl[\E\Bigl[\babs{D_kL^{-1}F}\,\Big|\,\G_k\Bigr]  (D_kF) D_k\Bigl(F\psi_z(F)+\1_{\{F>z\}}\Bigr)\biggr] }
\end{align}
Note that for the equality we have used the fact that, for $k\in\N$, 
\[H_k:=\E\Bigl[\babs{D_kL^{-1}F}\,\Big|\,\G_k\Bigr] \Bigl(F_k\psi_z\bigl(F_k\bigr)+ \1_{\{F_k>z\}}  \Bigr)\]
is $\G_k$-measurable, that $\E[D_kF|\G_k]=0$ and that $H_k D_kF\in L^1(\P)$, which holds since $x\mapsto x\psi_z(x)$ is bounded by Lemma \ref{solprops} (ii). Moreover, for the final inequality we have applied the second bound in Lemma \ref{solprops} (ii), (b). Next, we write
\[U_k:=\E\Bigl[\babs{D_kL^{-1}F}\,\Big|\,\G_k\Bigr]  D_kF,\quad k\in\N,\]
and let $U:=(U_k)_{k\in\N}$. Then note that, for $k\in\N$, by the symmetry of $D_k$ and since $D_kU_k=U_k$
 we have 
\begin{align*}
\E\biggl[\E\Bigl[\babs{D_kL^{-1}F}\,\Big|\,\G_k\Bigr]  (D_kF) D_k\Bigl(F\psi_z(F)+\1_{\{F>z\}}\Bigr)\biggr]
=\E\Bigl[U_k\bigl(F\psi_z(F)+\1_{\{F>z\}}\bigr)\Bigr]
\end{align*} 
and, therefore, \eqref{msb2} implies that 
\begin{align*}
|T_2|&\leq2\sum_{k=1}^\infty\E\Babs{ D_k\babs{D_kL^{-1}F}  D_kF}+\BBabs{\sum_{k=1}^\infty\E\Bigl[U_k\bigl(F\psi_z(F)+\1_{\{F>z\}}\bigr)\Bigr] }.
\end{align*}
Now, since $F,L^{-1}F\in\dom(D)$, the Cauchy-Schwarz and the conditional Jensen inequalities imply that $U\in L^1(\kappa\otimes\P_\F)$ so that, in view of Lemma \ref{solprops} (ii), (b) and the dominated convergence theorem, we can interchange the infinite sum and the expectation in the second term on the right hand side to obtain that 
\begin{align}\label{msb3}
|T_2|&\leq 2\sum_{k=1}^\infty\E\Babs{ D_k\babs{D_kL^{-1}F}  D_kF}+\BBabs{\E\biggl[\bigl(F\psi_z(F)+\1_{\{F>z\}}\bigr) \sum_{k=1}^\infty U_k \biggr] }\notag\\
&\leq2\sum_{k=1}^\infty\E\Babs{ D_k\babs{D_kL^{-1}F}  D_kF}+2\E\Bbabs{\sum_{k=1}^\infty U_k}
 \end{align}
where we have applied the bound in Lemma \ref{solprops} (ii), (b) dor the final inequality. The bound \eqref{mskol} now follows from \eqref{msb1}-\eqref{msb3} by taking the supremum over $z\in\R$.
 \end{proof}

The next result is a bound on the normal approximation that makes use of the Clark-Ocone covariance identity in Proposition \ref{clarkcovid}. Such bounds are less common in the Malliavin-Stein theory, with the works \cite{PrTo} and \cite{PrTo2} being two notable exceptions. 

\begin{theorem}[Clark-Ocone bounds]\label{cobound}
 Suppose that $F\in\dom(D)$ is such that $\E[F]=0$. Then, we have the following bounds:
 \begin{align}
  d_\W(F,Z)&\leq\sqrt{\frac{2}{\pi}}\E\Bbabs{1- \sum_{k=1}^\infty \bigl(D_k\E[F\,|\,\F_k]\bigr)\, D_kF} \notag\\
&\;  + \sum_{k=1}^\infty\E\Bigl[\babs{D_k\E[F\,|\,\F_k]}\bigl(D_kF\bigr)^2\Bigr]\quad\text{and}\label{cowass}\\
   d_\K(F,Z)&\leq\E\Bbabs{1- \sum_{k=1}^\infty \bigl(D_k\E[F\,|\,\F_k]\bigr)\, D_kF} +
2\sum_{k=1}^\infty\E\Babs{ D_k\babs{ D_k\E[F\,|\,\F_k]}  D_kF}\notag\\
&\;+2\E\Bbabs{\sum_{k=1}^\infty \E\bigl[\abs{D_k\E[F\,|\,\F_k]}\,\big|\,\G_k\bigr]D_kF }. \label{cokol}
\end{align}
\end{theorem}

\begin{remark}\label{remcobound}
\begin{enumerate}[(a)]
\item  If, additionally, $\Var(F)=\E[F^2]=1$, then by Proposition \ref{clarkcovid} the respective first term in the bounds may be estimated by means of the inequality 
\[\E\Babs{1- \sum_{k=1}^\infty \bigl(D_k\E[F\,|\,\F_k]\bigr)\, D_kF}\leq\Bigl(\Var\bigl(\sum_{k=1}^\infty \bigl(D_k\E[F\,|\,\F_k]\bigr)\, D_kF\bigr)\Bigr)^{1/2}.\] 
\item The second term in the bound \eqref{cokol} may be replaced with the (smaller) term 
\[\sup_{z\in\R}\Bbabs{\sum_{k=1}^\infty\E\biggl[ D_k\babs{ D_k\E[F\,|\,\F_k]} D_kF \Bigl(\chi_z(F)-\chi_z\bigl(\E[F\,|\,\G_k]\bigr)  \Bigr)\biggr]},\]
where, again, $\chi_z(x):=x\psi_z(x)+\1_{(z,\infty)}(x)$. 
\item Note that, by an application of the H\"older and the conditional Jensen inequalities, for the second term appearing in the Wasserstein bound \eqref{cowass} one has that 
\begin{align*}
 \sum_{k=1}^\infty \E\Bigl[ \babs{\E[D_kF\,|\,\F_k]}\, (D_kF)^2\Bigr]&\leq \sum_{k=1}^\infty \E\Bigl[ \babs{\E[D_kF\,|\,\F_k]}^3\Bigr]^{1/3} 
 \E\Bigl[|D_kF|^3\Bigr]^{2/3} \leq \sum_{k=1}^\infty \E\babs{D_kF}^3.
\end{align*}
 Alternatively, this term may also be estimated as follows: Using the Cauchy-Schwarz inequality we have that 
\begin{align*}
 &\sum_{k=1}^\infty\E\Bigl[\babs{D_k\E[F\,|\,\F_k]}\bigl(D_kF\bigr)^2\Bigr]
 \leq \sum_{k=1}^\infty \biggl(\E\Bigl[\E\bigl[D_kF|\F_k\bigr]^2\Bigr]\biggr)^{1/2}
 \biggl(\E\Bigl[\bigl(D_kF\bigr)^4\Bigr]\biggr)^{1/2}\\
 &\leq \biggl( \sum_{k=1}^\infty \E\Bigl[\E\bigl[D_kF|\F_k\bigr]^2\Bigr]\biggr)^{1/2}
 \biggl(\sum_{k=1}^\infty\E\Bigl[\bigl(D_kF\bigr)^4\Bigr]\biggr)^{1/2}= \biggl(\sum_{k=1}^\infty\E\Bigl[\bigl(D_kF\bigr)^4\Bigr]\biggr)^{1/2}\,,
\end{align*}
where we have applied Proposition \ref{clarkcovid} for the final equality. Finally, we also present the following third bound for this term:
\begin{align}\label{cobrem}
&\sum_{k=1}^\infty\E\Bigl[\babs{D_k\E[F\,|\,\F_k]}\bigl(D_kF\bigr)^2\Bigr]\leq \biggl( \sum_{k=1}^\infty \E\Bigl[\E\bigl[D_kF\big|\F_k\bigr]^2\bigl(D_kF\bigr)^2 \Bigr]\biggr)^{1/2}
 \biggl(\sum_{k=1}^\infty\E\Bigl[\bigl(D_kF\bigr)^2\Bigr]\biggr)^{1/2}\notag\\
&= \biggl( \sum_{k=1}^\infty \Var\Bigl(\E\bigl[D_kF\big|\F_k\bigr]D_kF \Bigr)+ \sum_{k=1}^\infty \Var\Bigl(\E\bigl[D_kF\big|\F_k\bigr]\Bigr)^2\biggr)^{1/2}
 \biggl(\sum_{k=1}^\infty\E\Bigl[\bigl(D_kF\bigr)^2\Bigr]\biggr)^{1/2},
\end{align}
where we have again used the Cauchy-Schwarz inequality and the identity
\[ \E\Bigl[\E\bigl[D_kF\big|\F_k\bigr]D_kF \Bigr]=\E\Bigl[\E\bigl[D_kF\big|\F_k\bigr]^2 \Bigr] =\Var\Bigl(\E\bigl[D_kF\big|\F_k\bigr]\Bigr), \]
as $\E\bigl[\E[D_kF|\F_k] \bigr]=\E[D_kF]=0$.
\end{enumerate}
\end{remark}

\begin{proof}[Proof of Theorem \ref{cobound}]
Again, let $h$ be either in $ \Lip(1)$ or $h=h_z=\1_{(-\infty,z]}$ for some $z\in\R$ and denote by $\psi=\psi_h$ the solution to the Stein equation \eqref{steineq} given by \eqref{steinsol}. 
Since $\psi_h$ is Lipschitz-continuos by Lemma \ref{solprops}, by Proposition \ref{lipprop} we have $\psi_h(F)\in\dom(D)$ and since $F$ is centered, it follows from Propositions \ref{clarkcovid} and \ref{chain}  that 
\begin{align}
&\E[F\psi_h(F)]=\Cov(F,\psi_h(F))=  \sum_{k=1}^\infty \E\bigl[D_k\E[F\,|\,\F_k]\, D_k\psi_h(F)\bigr]\label{co1}\\
&=  \sum_{k=1}^\infty \E\bigl[D_k\E[F\,|\,\F_k]\, \psi_h'(F) D_kF\bigr]-  \sum_{k=1}^\infty \E\bigl[  \bigl(D_k\E[F\,|\,\F_k]\bigr)\, S_k\bigr]\notag\\
&\;- \sum_{k=1}^\infty  \E\bigl[  D_k\E[F\,|\,\F_k]\, \E[R_k\,|\,\G_k]\bigr]\notag\\
&= \sum_{k=1}^\infty \E\bigl[\psi_h'(F)D_k\E[F\,|\,\F_k]\,  D_kF\bigr]-  \sum_{k=1}^\infty \E\bigl[  \bigl(D_k\E[F\,|\,\F_k]\bigr)\, S_k\bigr]\,.\notag
\end{align}
For the last identity we have used that 
\begin{align*}
 \sum_{k=1}^\infty \E\Bigl[  D_k\E[F\,|\,\F_k]\, \E[R_k\,|\,\G_k]\Bigr]=\sum_{k=1}^\infty \E\Bigl[ \E\bigl[ D_k\E[F\,|\,\F_k]\,\big|\,\G_k\bigr]\, \E[R_k\,|\,\G_k]\Bigr]=  0\,,
\end{align*}
by the definition of $D_k$ and since $\E[R_k\,|\,\G_k]$ is $\G_k$-measurable. Justifying that splitting up the infinite sum in \eqref{co1} into the three individual sums is legitimate as well as proving that $D_k\E[F\,|\,\F_k]\, \E[R_k\,|\,\G_k]\in L^1(\P)$ for each $k\in\N$ can be done in a similar way as in the proof of Theorem \ref{msbound} and is therefore omitted.\smallskip\\

Let us now first assume that $h\in\Lip(1)$. Then, from \eqref{co1} and the bound \eqref{RkSk} in Proposition \ref{chain} on the $|S_k|$, $k\in\N$,  we have that 
\begin{align}
 &\babs{\E[h(F)]-\E[h(Z)]}=\babs{\E[\psi_h'(F)]-\E[F\psi_h(F)]}\notag\\
& \leq \Babs{\E\Bigl[\psi_h'(F)\Bigl(1-\sum_{k=1}^\infty \bigl(D_k\E[F\,|\,\F_k]\bigr)\, D_kF\Bigr)\Bigr]}
+\Bbabs{\sum_{k=1}^\infty\E\bigl[ D_k\E[F\,|\,\F_k] S_k\bigr]}\label{co3}\\
&\leq \sqrt{\frac{2}{\pi}}\E\Babs{1- \sum_{k=1}^\infty \bigl(D_k\E[F\,|\,\F_k]\bigr)\, D_kF} 
  + \sum_{k=1}^\infty\E\Bigl[\babs{D_k\E[F\,|\,\F_k]}\bigl(D_kF\bigr)^2\Bigr]\,.\notag
\end{align}
Note that we have again made use of the fact that, by \eqref{clarkex} and \eqref{sfac}, we have
\[\sum_{k=1}^\infty \E\bigl[\psi_h'(F)D_k\E[F\,|\,\F_k]\,  D_kF\bigr]=\E\biggl[ \psi_h'(F) \sum_{k=1}^\infty D_k\E[F\,|\,\F_k]\, D_kF\biggr].\]
Again, taking the supremum over all $h\in\Lip(1)$ finishes the proof of \eqref{cowass}.\smallskip\\

To continue the proof of the bound \eqref{cokol} assume from now on that $h=h_z=\1_{(-\infty,z]}$ for some $z\in\R$. In this case, \eqref{co3} and Lemma \ref{solprops} (ii), (c) yield 
\begin{align}\label{co5}
&\babs{\P(F\leq z)-\P(Z\leq z)}\leq \E\Bbabs{1- \sum_{k=1}^\infty \bigl(D_k\E[F\,|\,\F_k]\bigr)\, D_kF}+\Bbabs{\sum_{k=1}^\infty\E\bigl[ D_k\E[F\,|\,\F_k] S_k\bigr]}.
\end{align}
Now, similarly as for the term $|T_2|$ in the proof of Theorem \ref{msbound}, we can show that
\begin{align}\label{co4}
&\Bbabs{\sum_{k=1}^\infty\E\bigl[ D_k\E[F\,|\,\F_k] S_k\bigr]}\leq 2\sum_{k=1}^\infty\E\Babs{ D_k\babs{ D_k\E[F\,|\,\F_k]}  D_kF}\notag\\
&\;+\BBabs{\sum_{k=1}^\infty\E\biggl[\E\Bigl[\babs{ D_k\E[F\,|\,\F_k]}\,\Big|\,\G_k\Bigr]  (D_kF) D_k\Bigl(F\psi_z(F)+\1_{\{F>z\}}\Bigr)\biggr] }\notag\\
&= 2\sum_{k=1}^\infty\E\Babs{ D_k\babs{ D_k\E[F\,|\,\F_k]}  D_kF}\notag\\
&\;+\BBabs{\E\biggl[\Bigl(F\psi_z(F)+\1_{\{F>z\}}\Bigr)\sum_{k=1}^\infty \E\Bigl[\babs{ D_k\E[F\,|\,\F_k]}\,\Big|\,\G_k\Bigr]  (D_kF) \biggr] }\notag\\
&\leq2\sum_{k=1}^\infty\E\Babs{ D_k\babs{ D_k\E[F\,|\,\F_k]}  D_kF}+2\E\Bbabs{\sum_{k=1}^\infty \E\bigl[\abs{D_k\E[F\,|\,\F_k]}\,\big|\,\G_k\bigr]D_kF },
\end{align}
where we have used the symmetry of the $D_k$ and the assumption $F\in\dom(D)$ in order to interchange the expectation and the infinite for the equality. Then, the bound \eqref{cokol} follows from \eqref{co5} and \eqref{co4} by taking the supremum over $z\in\R$.
\end{proof}

We finally provide bounds that make use of the carr\'{e}-du-champ operator $\Gamma_0$ and do not involve the Malliavin operators $D$ and $\delta$. 

\begin{theorem}[Carr\'{e}-du-champ bounds]\label{cdcbound}
 Suppose that $F\in\dom(D)$ is such that $\E[F]=0$. Moreover, let $f,g\in L^2(\mu)$ be representatives of $F$ and $L^{-1}F$, respectively. Then, we have the following bounds:
\begin{align}
  d_\W(F,Z)&\leq \sqrt{\frac{2}{\pi}}\E\Babs{1-\Gamma_0(F,-L^{-1}F)}\notag\\
  &\hspace{2cm}  +\frac12\sum_{k=1}^\infty\E\Bigl[\babs{g(\X^{(k)})-g(\X)}\bigl(f(\X^{(k)})-f(\X)\bigr)^2\Bigr]\quad\text{and}\label{wasscdc}\\
	d_\K(F,Z)&\leq\E\Babs{1-\Gamma_0(F,-L^{-1}F)}\notag\\
	&\hspace{2cm}+\E\Bbabs{\E\biggl[\sum_{k=1}^\infty\Babs{g\bigl(\X^{(k)}\bigr)- g\bigl(\X\bigr) }\Bigl(f\bigl(\X^{(k)}\bigr)- f\bigl(\X\bigr)  \Bigr)\,\bigg|\,F\biggr]}\label{kolcdc}
 \end{align}

\end{theorem}

\begin{remark}\label{cdcboundrem}
\begin{enumerate}[(a)]
\item  If, additionally, $\Var(F)=\E[F^2]=1$, then by Proposition \ref{cdccovid} the respective first term appearing in the bounds may be estimated by means of the inequality 
\[\E\Babs{1-\Gamma_0(F,-L^{-1}F)}\leq\Bigl(\Var\bigl(\Gamma_0(F,-L^{-1}F)\bigr)\Bigr)^{1/2}.\]
\item In the second term of the bound \eqref{kolcdc}, conditioning with respect to $F$ may be replaced by conditioning with respect to $\X$, which can only make this quantity larger. Moreover, note that 
\[\E\biggl[\sum_{k=1}^\infty\Babs{g\bigl(\X^{(k)}\bigr)- g\bigl(\X\bigr) }\Bigl(f\bigl(\X^{(k)}\bigr)- f\bigl(\X\bigr)  \Bigr)\bigg|F\biggr]
=\sum_{k=1}^\infty \E\biggl[\Babs{g\bigl(\X^{(k)}\bigr)- g\bigl(\X\bigr) }\Bigl(f\bigl(\X^{(k)}\bigr)- f\bigl(\X\bigr)  \Bigr)\bigg|F\biggr]
\]
is a well-defined random variable in $L^1(\P)$ since $F,L^{-1}F\in\dom(D)$. This follows in the same way as the well-definedness of $\Gamma_0(F,G)$ for $F,G\in\dom(D)$.  
\item By an application of the Cauchy-Schwarz inequality, for the second term in \eqref{wasscdc} one has 
\begin{align*}
&\sum_{k=1}^\infty\E\Bigl[\babs{g(\X^{(k)})-g(\X)}\bigl(f(\X^{(k)})-f(\X)\bigr)^2\Bigr]\\
&\leq \biggl(\sum_{k=1}^\infty\E\Bigl[\bigl(g(\X^{(k)})-g(\X)\bigr)^2\Bigr]\biggr)^{1/2} \biggl(\sum_{k=1}^\infty\E\Bigl[\bigl(f(\X^{(k)})-f(\X)\bigr)^4\Bigr]\biggr)^{1/2}\\
&=\biggl(\E\Bigl[\Gamma_0\bigl(L^{-1}F,L^{-1}F)\bigr)\Bigr]\biggr)^{1/2} \biggl(\sum_{k=1}^\infty\E\Bigl[\bigl(f(\X^{(k)})-f(\X)\bigr)^4\Bigr]\biggr)^{1/2}\\
&=\Bigl(-\E\bigl[FL^{-1}F\bigr]\Bigr)^{1/2}\biggl(\sum_{k=1}^\infty\E\Bigl[\bigl(f(\X^{(k)})-f(\X)\bigr)^4\Bigr]\biggr)^{1/2},
\end{align*}
where we have used that $g(\X)=L^{-1}F$. 
\item If $F\in L^4_\X$ is in fact $\F_n$-measurable for some $n\in\N$, then it can be shown as in \cite[Lemma 2.6]{D23} that 
\begin{align}\label{fundid}
\sum_{k=1}^\infty\E\Bigl[\bigl(f(\X^{(k)})-f(\X)\bigr)^4\Bigr]&=12\E\bigl[F^2\Gamma_0(F,F)\bigr]+4\E\bigl[F^3 LF\bigr]
\end{align}
so that for the second term in the bound \eqref{wasscdc} we further obtain
\begin{align*}
&\frac12\sum_{k=1}^\infty\E\Bigl[\babs{g(\X^{(k)})-g(\X)}\bigl(f(\X^{(k)})-f(\X)\bigr)^2\Bigr]\\
&\leq\Bigl(-\E\bigl[FL^{-1}F\bigr]\Bigr)^{1/2}\Bigl( 3\E\bigl[F^2\Gamma_0(F,F)\bigr]+\E\bigl[F^3 LF\bigr] \Bigr)^{1/2}.
\end{align*}
\end{enumerate}
\end{remark}

\begin{proof}[Proof of Theorem \ref{cdcbound}]
Again, let $h$ be either in $ \Lip(1)$ or $h=h_z=\1_{(-\infty,z]}$ for some $z\in\R$ and denote by $\psi=\psi_h$ the solution to the Stein equation \eqref{steineq} given by \eqref{steinsol}.
Since $\psi_h$ is Lipschitz-continuos by Lemma \ref{solprops}, by Proposition \ref{lipprop} we have $\psi_h(F)\in\dom(D)$ and since $F$ is centered, it follows from Proposition \ref{cdccovid} that 
\begin{align}\label{cdcb1}
 \E\bigl[F\psi_h(F)\bigr]&=\Cov\bigl(F,\psi_h(F)\bigr)
 =\E\Bigl[\Gamma_0\bigl(-L^{-1}F,\psi_h(F)\bigr)\Bigr]\notag\\
& =\E\Bigl[\psi_h'(F)\Gamma_0\bigl(-L^{-1}F,F\bigr)\Bigr]+\E\bigl[R_{\psi_h}(F,-L^{-1}F)\bigr],
\end{align}
where 
\begin{align}\label{Rpsi}
&R_{\psi_h}(F,-L^{-1}F):=\Gamma_0\bigl(-L^{-1}F,\psi_h(F)\bigr)-\psi_h'(F)\Gamma_0\bigl(-L^{-1}F,F\bigr)\notag\\
&=\frac12\sum_{k=1}^\infty\E\Bigl[\Bigl(\psi_h\bigl(f(\X^{(k)})\bigr)-\psi_h\bigl(f(\X)\bigr)-\psi_h'(F)\bigl(f(\X^{(k)})-f(\X)\bigr)\Bigr)\notag\\
&\hspace{6cm}\cdot\Bigl(g(\X)-g(\X^{(k)})\Bigr)\,\Bigl|\,\X\Bigr].
\end{align}
From \eqref{cdcb1} we thus have
\begin{align}\label{cdcb2}
 &\E[h(Z)]-\E[h(F)]=\E\bigl[\psi_h'(F)-F\psi_h(F)\bigr]\notag\\
&= \E\Bigl[\psi_h'(F)\Bigl(1-\Gamma_0\bigl(-L^{-1}F,F\bigr)\Bigr)\Bigr]+\E\bigl[R_{\psi_h}(F,-L^{-1}F)\bigr]=:T_1+T_2.
\end{align}
We have 
\begin{align}\label{cdcb3}
|T_1|&\leq\norm{\psi_h'}_\infty\,\E\Babs{1-\Gamma_0(F,-L^{-1}F)}
\end{align}
yielding the respective first term in the bounds \eqref{wasscdc} and \eqref{kolcdc} by virtue of Lemma \ref{solprops}. To bound $|T_2|$ we argue separately for the Wasserstein and Kolmogorov bounds. 

Suppose first that $h\in \Lip(1)$. Then, denoting by 
\begin{align*}
r_{\psi_h}(x,y):=\psi_h(x+y)-\psi_h(x)-\psi_h'(x)y=\int_0^y(y-s)\psi_h''(x+s)ds
\end{align*}
the error term in the Taylor expansion of $\psi_h$, by \eqref{sfac} we have 
\[|r_{\psi_h}(x,y)|\leq \frac{\fnorm{\psi_h''}}{2}y^2\leq y^2\,,\quad x,y\in\R,\]
and, hence, from \eqref{Rpsi} we obtain
\begin{align}\label{cdcb4}
\babs{R_{\psi_h}(F,-L^{-1}F) }&\leq\frac12\sum_{k=1}^\infty\E\Bigl[\babs{g(\X^{(k)})-g(\X)}\bigl(f(\X^{(k)})-f(\X)\bigr)^2\,\Bigl|\,\X\Bigr].
\end{align}
Therefore, \eqref{cdcb2}, \eqref{cdcb3}, \eqref{cdcb4} and \eqref{sfac} imply that
\begin{align*}
 &\babs{\E[h(Z)]-\E[h(F)]}=\babs{\E\bigl[\psi_h'(F)-F\psi_h(F)\bigr]}\leq |T_1|+|T_2|\\
&\leq\norm{\psi_h'}_\infty\,\E\Babs{1-\Gamma_0(F,-L^{-1}F)}+ \E\babs{R_{\psi_h}(F,-L^{-1}F)}  \\
&\leq\sqrt{\frac{2}{\pi}}\E\Babs{1-\Gamma_0\bigl(-L^{-1}F,F\bigr)} +\frac12\sum_{k=1}^\infty\E\Bigl[\babs{g(\X^{(k)})-g(\X)}\bigl(f(\X^{(k)})-f(\X)\bigr)^2\Bigr]
\end{align*}
and taking the supremum over all $h\in\Lip(1)$ finishes the proof of \eqref{wasscdc}. \smallskip\\

Suppose now that $h=h_z=\1_{(-\infty,z]}$ for some $z\in\R$. Then, writing 
\[G:=g(\X)=L^{-1}F,\quad F_k:=f(\X^{(k)})\quad\text{and}\quad G_k:=g(\X^{(k)}),\quad k\in\N,\]
since $\psi_z=\psi_h$ is Lipschitz-continuous by Lemma \ref{solprops} (ii), from \eqref{Rpsi} we have  
\begin{align}\label{cdcb5}
T_2&=\frac12\sum_{k=1}^\infty\E\Bigl[\bigl(\psi_z(F_k)-\psi_z(F)-\psi_z'(F)(F_k-F)\bigr)\bigl(G_k-G\bigr)\Bigr]\notag \\
&=\frac12 \sum_{k=1}^\infty\E\biggl[\bigl(G_k-G\bigr)\int_{0}^{F_k-F}\bigl(\psi_z'(F+t)-\psi_z'(F)\bigr)dt\biggr].
\end{align}
Now, for each $k\in\N$, since $\psi_z$ solves \eqref{steineq} pointwise, we have 
\begin{align*}
T_{2,k}&:=\E\biggl[\bigl(G_k-G\bigr)\int_{0}^{F_k-F}\bigl(\psi_z'(F_k+t)-\psi_z'(F)\bigr)dt\biggr]\notag \\
&=\E\biggl[\bigl(G_k-G\bigr)\int_{0}^{F_k-F}\bigl((F+t)\psi_z(F+t)-F\psi_z(F)\bigr)dt\biggr]\notag\\
&\;+\E\biggl[\bigl(G_k-G\bigr)\int_{0}^{F_k-F}\bigl(\1_{\{F+t\leq z\}}-\1_{\{F\leq z\}}\bigr)dt\biggr]=:T_{3,k}+T_{4,k}.
\end{align*} 
Since $x\mapsto x\psi_z(x)$ is increasing by Lemma \ref{solprops} (ii), (a) andsince  $x\mapsto \1_{\{x\leq z\}}$ is non-increasing, we have 
\begin{align}
0&\leq I_{k,1}:=\int_{0}^{F_k-F}\bigl((F+t)\psi_z(F+t)-F\psi_z(F)\bigr)dt\notag\\
&\leq (F_k-F)\bigl(F_k\psi_z(F_k)-F\psi_z(F)\bigr)\quad\text{and} \label{Ik1}\\
0&\geq I_{k,2}:=\int_{0}^{F_k-F}\bigl(\1_{\{F+t\leq z\}}-\1_{\{F\leq z\}}\bigr)dt\geq  (F_k-F)\bigl(\1_{\{F_k\leq z\}}-\1_{\{F\leq z\}}\bigr)\notag\\
&=-(F-F_k)\bigl(\1_{\{F> z\}}-\1_{\{F_k> z\}}\bigr).\label{Ik2}
\end{align}
Next, writing $\Delta_k:=G_k-G=g(\X^{(k)})-g(\X)$ and using that $(\X,\X^{(k)})$ has the same distribution as $(\X^{(k)},\X )$, $k\in\N$, from \eqref{Ik1} we obtain
\begin{align*}
T_{3,k}&=\E\bigl[\Delta_k I_{k,1}\bigr]\leq \E\bigl[\Delta_k\1_{\{\Delta_k>0\}} I_{k,1}\bigr]+\frac12\E\bigl[\Delta_k\1_{\{\Delta_k=0\}} I_{k,1}\bigr]\\
&=\E\bigl[|\Delta_k|\1_{\{\Delta_k>0\}} I_{k,1}\bigr]+\frac12\E\bigl[|\Delta_k|\1_{\{\Delta_k=0\}} I_{k,1}\bigr]\\
&\leq\E\Bigl[|\Delta_k|\1_{\{\Delta_k>0\}} (F_k-F)\bigl(F_k\psi_z(F_k)-F\psi_z(F)\bigr)\Bigr]\\
&\;+\frac12\E\Bigl[|\Delta_k|\1_{\{\Delta_k=0\}}(F_k-F)\bigl(F_k\psi_z(F_k)-F\psi_z(F)\bigr)\Bigr]\\
&=-\E\Bigl[|\Delta_k|\bigl(\1_{\{\Delta_k>0\}}+\1_{\{\Delta_k<0\}}+\1_{\{\Delta_k=0\}}\bigr) (F_k-F)F\psi_z(F)\Bigr]\\
&=-\E\Bigl[F\psi_z(F)\E\bigl[|G_k-G|(F_k-F)\,\big|\, F\bigr]\Bigr].
\end{align*}
Since we can obtain analogously that 
\[T_{3,k}\geq\E\Bigl[F\psi_z(F)\E\bigl[|G_k-G|(F_k-F)\,\big|\, F\bigr]\Bigr],\quad k\in\N,\]
using the first bound in Lemma \ref{solprops} (ii), (b), we conclude that 
\begin{align*}
\Bbabs{\sum_{k=1}^\infty T_{3,k}}&\leq \Biggl|\E\Biggl[F\psi_z(F)\E\biggl[\sum_{k=1}^\infty|G_k-G|(F_k-F)\,\bigg|\, F\biggr]\Biggr]\Biggr|\\
&\leq \E\Bbabs{\sum_{k=1}^\infty\E\Bigl[|G_k-G|(F_k-F)\,\Big|\, F\Bigr] }.
\end{align*}
Similarly, this time using \eqref{Ik2} in place of \eqref{Ik1}, we can prove that 
 \begin{align*}
\Bbabs{\sum_{k=1}^\infty T_{4,k}}&\leq \E\Bbabs{\sum_{k=1}^\infty\E\Bigl[|G_k-G|(F_k-F)\,\Big|\, F\Bigr] }
\end{align*}
so that we obtain 
\begin{align}\label{cdcb6}
|T_2|&=\frac12\Bbabs{\sum_{k=1}^\infty T_{2,k}}=\frac12\Bbabs{\sum_{k=1}^\infty(T_{3,k}+T_{4,k})}\leq\E\Bbabs{\sum_{k=1}^\infty\E\Bigl[|G_k-G|(F_k-F)\,\Big|\, F\Bigr] }.
\end{align}
Therefore, from \eqref{cdcb2}, \eqref{cdcb3} and \eqref{cdcb6} we can conclude that 
\begin{align*}
&\babs{\P(Z\leq z)-\P(F\leq z)}=\babs{\E\bigl[\psi_z'(F)-F\psi_z(F)\bigr]}\leq|T_1|+|T_2|\notag\\
&\leq \E\Babs{1-\Gamma_0\bigl(-L^{-1}F,F\bigr)}+ \E\Bbabs{\sum_{k=1}^\infty\E\Bigl[|G_k-G|(F_k-F)\,\Big|\, F\Bigr] }
\end{align*}
and the Berry-Esseen bound \eqref{kolcdc} follows from taking the supremum over all $z\in\R$.
\end{proof}

\begin{remark}\label{boundrem}
\begin{enumerate}[(a)]
\item The Wasserstein bounds in Theorems \ref{msbound} and \ref{cdcbound} are similar to corresponding bounds for Poisson functionals \cite{PSTU, DP18a} and functionals of Rademacher sequences \cite{Zheng, DK19}. Moreover,  in the situation of functionals of finitely many independent random variables $X_1,\dotsc,X_n$, the bound \eqref{wasscdc} in Theorem \ref{cdcbound} is a special case of the bound in \cite[Theorem 2.8]{D23}. 
\item As indicated above, the Wasserstein bound \eqref{mswass} in Theorem \ref{msbound} slightly improves on the bound in \cite[Theorem 5.9]{DH} since, thanks to our new chain rule formula in Proposition \ref{chain}, the second term of the bound in Theorem \ref{msbound} is generally smaller than the corresponding term in that paper. Moreover, our bound is very similar to (but still smaller than) the Wasserstein bound stated at the bottom of page 629 in \cite{Duer}. 
 We further remark that in the proof given in \cite{DH} no argument for $\psi_h(F)\in\dom(D)$ (as in Proposition \ref{lipprop} (i) above) has been provided.
\item The Clark-Ocone bound \eqref{wasscdc} on the Wasserstein distance in Theorem \ref{cobound} compares to the bound for Rademacher functionals given in \cite[Theorem 3.2]{PrTo} (see Remark \ref{raderem} (c) below). 
\item The new Berry-Esseen bound \eqref{kolcdc} compares to similar carr\'{e}-du-champ bounds in the Poisson \cite{DP18a} and Rademacher \cite{DK19,ERTZ} settings, although its terms are simpler and, apart from $F\in\dom(D)$, in contrast to these references, we do not have to make any additional technical assumptions. It formally resembles the recent Berry-Esseen bound by Shao and Zhang \cite{ShaZh} for exchangeable pairs satisfying (at least approximately) Stein's linear regression property \cite{St86}. Note that, in the recent work \cite{D23} we have pointed out that, for functionals $F$ of independent random variables, this linear regression property is equivalent to the condition that $F\in\mathcal{H}_p$ for some $p\in\N$. In this sense our Berry-Esseen bound \eqref{kolcdc} is in fact more general than the bound in \cite{ShaZh}, since no such homogeneity assumption has been made here. Moreover, the bound in \cite{ShaZh} does not seem applicable to products of infinitely many probability spaces.  The bound \eqref{wasscdc} on the Wasserstein distance may be considered a more general, infinite version of bounds provided in the recent articles \cite{DP17,DP19, D23}, which have been derived by the technique of exchangeable pairs and which have been successfully applied to several relevant examples that had not been amenable by previous methods. 
\item The Berry-Esseen bound \eqref{mskol} may be compared to corresponding bounds in the Poisson \cite{Sch16,ET14,LRPY} and Rademacher \cite{KRT1, KRT2,ERTZ} situations (see again Remark \ref{raderem} (e) below for details on the latter). It looks simpler and makes milder assumptions than these bounds, since it neither (explicitly) features the divergence operator $\delta$ nor the parameter $z\in\R$ stemming from the indicator function $\1_{\{x\leq z\}}$. On the other hand, the respective second term in the bounds \eqref{mskol} and\eqref{cokol} might be difficult to control in practice.   
\item The proofs of the Berry-Esseen bounds in Theorems \ref{msbound}, \ref{cobound} and \ref{cdcbound} make use of a recent trick due to Shao and Zhang \cite{ShaZh}, who have applied it in the context of exchangeable pairs. This trick essentially consists in exploiting the monotonicity of the functions $x\mapsto x\psi_z(x)$ and $x\mapsto\1_{\{x\leq z\}}$ instead of relying on a first order Taylor estimate of the function $\psi_z$ as e.g. given in \cite[Proposition 3.1 (f)]{LRP3}. This allows to dispense with certain redundant terms in the resulting bound, leading to shorter proofs in concrete applications. We remark that this trick has also been applied in the Poisson \cite{LRPY} and Rademacher frameworks \cite{ERTZ}. However, somewhat surprisingly, contrary to these two references, we do not need to make any additional assumptions here like $F\psi_z(F)\in\dom(D)$ for all $z\in\R$ or that some process $U=(U_k)_{k\in\N}$ be in $\dom(\delta)$. In fact, a variant of the Malliavin integration by parts formula \eqref{intpartsM}, where the conditions $U\in\dom(\delta)$ and $F\in\dom(D)$ are replaced with $U\in L^1(\kappa\otimes\P_\F)$ and $F\in L^\infty_\X$, has been (implicitly) applied at the end of the proofs of Theorems \ref{msbound} and \ref{cobound}.    
\item For many relevant applications it would be interesting to also have comparable bounds on the multivariate normal approximation of random vectors $(F_1,\dotsc,F_d)^T\in (L_\X^2)^d$. It is indeed possible to combine the above techniques with suitable versions of Stein's method for multivariate normal approximation to obtain such bounds. Such results and relevant applications will be dealt with in a follow-up project.
\item We briefly discuss here some alternative approaches to the normal approximation of functionals on product spaces. Contrary to the approach considered here and in the works \cite{DH,Duer}, nearly all these methods only deal with functionals of finitely many, say $n$, independent random variables $X_1,\dotsc,X_n$. To the best of our knowledge, the only other exception is the approach by Privault and Serafin \cite{PS1,PS2} that exploits the well-established Malliavin calculus for sequences of independent and uniformly distributed random variables in the interval $[-1,1]$ to prove Wasserstein and Berry-Esseen bounds. By means of the quantile transformation, the bounds in \cite{PS1,PS2} can in principle be applied to functionals of general sequences of real-valued random variables, which is illustrated therein by means of several examples. 

The classical works by van Zwet \cite{vanZwet} and Friedrich \cite{Fr} derive Berry-Esseen bounds for symmetric functionals of i.i.d. and for general functionals of independent random variables, respectively, by combining Esseen's smoothing inequality with suitable projection methods. As a consequence, their theorems appear to be applicable mainly to functionals whose first Hoeffding projection is asymptotically dominant. Chen and Shao \cite{ChSh07} develop a version of the concentration inequality approach to Stein's method and apply it to prove a general Berry-Esseen bound that has a similar realm of applications but seems easier to apply than the bound in \cite{Fr}. 

Chatterjee \cite{Cha08} develops a new version of Stein's method that is based on an elegant covariance formula, which itself is expressed in terms of difference operators making use of an independent copy $X'$ of the vector $X=(X_1,\dotsc,X_n)$. He obtains a general bound on the Wasserstein distance to normality and demonstrates the flexibility of his approach by means of several impressive examples. This approach has been considerably extended by Lachi\`{e}ze-Rey and Peccati \cite{LRP3} who provide corresponding Berry-Esseen bounds and investigate further classes of relevant examples. We remark that the difference operator $\Delta_k$ of \cite{Cha08} is in our notation given by $\Delta_k F=f(\X)-f(\X^{(k)})$, where $F=f(\X)$, and has thus been used implicitly in \cite{DH} and in the present paper as well. In particular, one has $D_kF=\E[\Delta_k F|\X]$ for the Malliavin derivative considered here. However, the covariance formula in \cite[Lemma 2.3]{Cha08} is rather different from those considered in Subsection \ref{covid} and, therefore, the bounds on normal approximation in \cite{Cha08,LRP3} are not straightforward to compare to those in the present section. 
\end{enumerate}
\end{remark}

\subsection{Normal approximation bounds for functionals of Rademacher sequences}\label{rade}
In this subsection we present consequences of the Theorems \ref{msbound} and \ref{cobound} in the special case of a (non-homogeneous and non-symmetric) Rademacher sequence $\X=(X_k)_{k\in\N}$. That is, here we assume that the $X_k$, $k\in\N$, are independent and that for certain $p_k=1-q_k\in(0,1)$ one has $\P(X_k=1)=p_k=1-\P(X_k=-1)=1-q_k$. We refer to Remark \ref{domrem} (e) for the notation used here. In particular, as $\hat{L}=L$ we continue to write $L$ in place of $\hat{L}$. 

\begin{cor}\label{radecor}
Let $F\in\dom(\hat{D})=\dom(D)$ be centered. Then, we have the Malliavin-Stein bounds
\begin{align}
d_\W(F,Z)&\leq \sqrt{\frac{2}{\pi}}\E\Bbabs{1-\sum_{k=1}^\infty \hat{D}_k(-L^{-1}F)\hat{D}_kF\,Y_k^2}\notag\\
&\;+\sum_{k=1}^\infty \frac{1-2p_kq_k}{\sqrt{p_kq_k}}\E\Bigl[\babs{\hat{D}_k(-L^{-1}F)}\bigl(\hat{D}_kF\bigr)^2\Bigr],\label{rmswass}\\
d_\K(F,Z)&\leq \E\Bbabs{1-\sum_{k=1}^\infty \hat{D}_k(-L^{-1}F)\hat{D}_kF\,Y_k^2}+
2\sum_{k=1}^\infty|p_k-q_k|\E\Bigl[\babs{ \hat{D}_kL^{-1}F}\babs{  \hat{D}_kF}\Bigr]\notag\\
&\;+4\E\Bbabs{\sum_{k=1}^\infty\sqrt{p_kq_k}\bigl(\abs{\hat{D}_kL^{-1}F}\hat{D}_kF\bigr) Y_k},\label{rmskol}
\end{align}
as well as the following Clark-Ocone bounds:
\begin{align}
 d_\W(F,Z)&\leq\sqrt{\frac{2}{\pi}}\E\Bbabs{1- \sum_{k=1}^\infty \bigl(\E[\hat{D}_kF\,|\,\F_{k-1}]\, \hat{D}_kF\bigr)Y_k^2} \notag\\
&\;  + \sum_{k=1}^\infty\frac{1-2p_kq_k}{\sqrt{p_kq_k}} \E\Bigl[\babs{\E[\hat{D}_kF\,|\,\F_{k-1}]}\bigl(\hat{D}_kF\bigr)^2\Bigr],\label{rcowass}\\
   d_\K(F,Z)&\leq\E\Bbabs{1- \sum_{k=1}^\infty \bigl(\E[\hat{D}_kF\,|\,\F_{k-1}]\, \hat{D}_kF\bigr)Y_k^2}+
2\sum_{k=1}^\infty |p_k-q_k|\E\Bigl[\babs{\E[\hat{D}_kF\,|\,\F_{k-1}]}\babs{  \hat{D}_kF}\Bigr] \notag\\
&\;+4\E\Bbabs{\sum_{k=1}^\infty \sqrt{p_kq_k}\,\Bigl(\babs{\E[\hat{D}_kF\,|\,\F_{k-1}]}\hat{D}_kF\Bigr)Y_k }. \label{rcokol}
\end{align}
\end{cor}

\begin{proof}
The proofs of these bounds follow from Theorems \ref{msbound} and \ref{cobound} by using that, as has been mentioned in Remark \ref{domrem} (e), for $F\in L^1_\X$, we have 
\begin{align}\label{relD}
D_kF= Y_k \hat{D}_kF,\quad k\in\N, 
\end{align}
as well as from the following relations that are either standard facts from the Malliavin calculus for Rademacher sequences \cite{Priv08} or straightforward to verify: For $k\in\N$ and $F\in L^1_\X$ with representative $f\in L^1(\mu)$ one has  
\begin{align*}
\hat{D}_kF&=\sqrt{p_kq_k}\Bigl(f\bigl(X_1,\dotsc,X_{k-1},1,X_{k+1},\dotsc\bigr)-f\bigl(X_1,\dotsc,X_{k-1},-1,X_{k+1},\dotsc\bigr)\Bigr)
\end{align*} 
so that $Y_k$ and $\hat{D}_kF$ are, in particular, independent, and from
\begin{align*}
\hat{D}_k| Y_k|&=q_k-p_k, &\E|Y_k|&=2 \sqrt{p_kq_k},&\E\bigl[Y_k|Y_k|\bigr]&=q_k-p_k, &\E|Y_k|^3&=\frac{1-2p_kq_k}{\sqrt{p_kq_k}}.
\end{align*} 
We leave the details of these computations and derivations to the reader.
\end{proof}

\begin{remark}\label{raderem}
\begin{enumerate}[(a)]
\item As a general remark, the main purpose of Corollary \ref{radecor} is to illustrate the flexibility of the theory in this paper. It is reasonable to expect that methods which are specifically tailored to the Rademacher situation should in general yield sharper estimates. 
\item The bound \eqref{rmswass} is similar to the bound given in \cite[Theorem 3.1]{Zheng} but (as in fact all bounds given in Corollary \ref{rade}) it features an extra factor of $Y_k^2$ in the sum appearing in the first term of the bound. Thanks to the independence of $Y_k$ and $\hat{D}_kF, \hat{D}_kL^{-1}F$ mentioned in the above proof, this does not introduce too much additional complexity in the bound, however. Note that in fact $Y_k^2=1$ in the symmetric case, where $p_k=q_k=1/2$ for all $k\in\N$.
\item Similarly, the Clark-Ocone bound \eqref{rcowass} on the Wasserstein distance compares to the bound for Rademacher functionals given in \cite[Theorem 3.2]{PrTo}. Note, however, that the bound in \cite{PrTo} is expressed in terms of $C^2$ test functions and, moreover, features an additional term $B_3$ and therefore appears to be larger than ours. The Berry-Esseen bound \eqref{rcokol} seems to be the first of the Clark-Ocone type in the Malliavin-Stein world for Rademacher sequences.
\item The respective second terms in our Berry-Esseen bounds \eqref{rmskol} and \eqref{rcokol} are only likely to be small, when $p_k$ and $q_k$ are reasonably close to each other for (almost) all $k\in\N$. In fact, in the symmetric situation these terms of course vanish identically.
\item We finally compare the bound \eqref{rmskol} to existing ones from the literature: In the symmetric case it outperforms the Berry-Esseen bound given in \cite[Theorem 3.1]{KRT1} which, as compared to \eqref{rmskol}, contains two additional terms. Moreover, the last term of the bound from \cite{KRT1} still features the supremum over all $z\in\R$ as well as the indicator $\1_{\{F>z\}}$ and is therefore less explicit than the final term in \eqref{rmskol}. In the general case, the bound \eqref{rmskol} looks simpler and features fewer terms than the bound provided in \cite[Proposition 4.1]{KRT2}. However, the latter bound has proven to yield accurate estimates on the rate of convergence even in situations where $p_k$ and $q_k$ are not close to each other or, in fact, are even close to 1 and 0 (or vice versa), respectively. The recent bounds given in displays (3.4) and (3.6) of \cite[Theorem 3.1]{ERTZ} are improvements of the bound in \cite[Proposition 4.1]{KRT2} and, in non-symmetric situations, further seem to be more flexible than our bound \eqref{rmskol} . However, the technical condition (3.3) in \cite{ERTZ} for the bounds in (3.4) and (3.6) to hold, or the condition that $U\in\dom(\hat{\delta})$, where $U_k:= (p_kq_k)^{-1/2}\bigl(\abs{\hat{D}_kL^{-1}F}\hat{D}_kF\bigr)$, $k\in\N$, which must additionally be assumed for the bound (3.6) to apply, might be difficult to verify in practice.
\end{enumerate}
\end{remark}

\section{Proofs of applications}\label{proofs}
\subsection{Proof of Theorem \ref{infdejong}}\label{djproof}
We make use of the following lemma, the first part of which is certainly known. For the second part, we have not been able to find a suitable reference though. 

\begin{lemma}\label{dislemma}
Let $Y,W,W_n$, $n\in\N$, be real-valued random variables. 
\begin{enumerate}[{\normalfont(a)}]
\item If $W,W_n\in L^1(\P)$, $n\in\N$, for the same probability space $(\Om,\F,\P)$ and $(W_n)_{n\in\N}$ converges to $W$ in $L^1(\P)$, then $d_\W(W,Y)\leq\liminf_{n\to\infty} d_\W(W_n,Y)$.
\item If $(W_n)_{n\in\N}$ converges in distribution to $W$, then $d_\K(W,Y)\leq\liminf_{n\to\infty} d_\K(W_n,Y)$.
\end{enumerate}
\end{lemma}

\begin{remark}
Note that the statement of Lemma \ref{dislemma} is reminiscent of well-known properties of norms under weak and weak-$\ast$ convergence in a functional analysis context. It is unclear to the author if this observation may be turned into a rigorous proof though.
\end{remark}

\begin{proof}[Proof of Lemma \ref{dislemma}]
To prove (a) fix $h\in\Lip(1)$. Then, 
\begin{align*}
&\babs{\E[h(W)]-\E[h(Y)]}\leq \babs{\E[h(W)]-\E[h(W_n)]}+\babs{\E[h(W_n)]-\E[h(Y)]}\\
&\leq  \babs{\E[h(W)]-\E[h(W_n)]}+d_\W(W_n,Y)\leq \E\babs{W-W_n}+d_\W(W_n,W).
\end{align*}
Hence, as the left hand side does not depend on $n$ it follows from the $L^1$-convergence that 
\[\babs{\E[h(W)]-\E[h(Y)]}\leq\liminf_{n\to\infty}\Bigl(\E\babs{W-W_n}+d_\W(W_n,Y)\Bigr)=\liminf_{n\to\infty}d_\W(W_n,Y).\]
Taking the supremum over all $h\in\Lip(1)$ yields (a). 

To  prove (b) fix $z\in\R$. If $z$ is a continuity point of the distribution function $F_W$ of $W$, then, by convergence in distribution, we have
\begin{align*}
\babs{\P(W\leq z)-\P(Y\leq z)}=\lim_{n\to\infty}\babs{\P(W_n\leq z)-\P(Y\leq z)}\leq \liminf_{n\to\infty} d_\K(W_n,Y).
\end{align*}
On the other hand, if $F_W$ is discontinuous at $z$, then, for given $\eps>0$ we may choose a continuity point $z_\eps$ of $F_W$ such that $z<z_\eps<z+\eps$. Then, again by convergence in distribution, 
we have  
\begin{align*}
&\P(W\leq z)-\P(Y\leq z)\leq \P(W\leq z_\eps)-\P(Y\leq z_\eps) +  \P(z< Y\leq z_\eps) \notag\\
& \leq \babs{\P(W\leq z_\eps)-\P(Y\leq z_\eps)} +F_Y(z+\eps)-F_Y(z)\notag\\
&=\lim_{n\to\infty} \babs{\P(W_n\leq z_\eps)-\P(Y\leq z_\eps)}+F_Y(z+\eps)-F_Y(z)\notag\\
&\leq \liminf_{n\to\infty} d_\K(W_n,Y)+F_Y(z+\eps)-F_Y(z)
\end{align*}
and, since $\eps>0$ was arbitrary and the distribution function $F_Y$ of $Y$ is right-continuous at $z$, letting $\eps\downarrow0$, we obtain that 
\begin{align}\label{kol2}
\P(W\leq z)-\P(Y\leq z)\leq\liminf_{n\to\infty} d_\K(W_n,Y). 
\end{align} 
Moreover, again by convergence in distribution and since $(-\infty,z]$ is a closed subset of $\R$, the Portmanteau theorem implies that
\begin{align}\label{kol1}
\P(Y\leq z)-\P(W\leq z)&\leq \P(Y\leq z)-\limsup_{n\to\infty}\P(W_n\leq z)\notag\\
&= \liminf_{n\to\infty}\bigl(\P(Y\leq z)-\P(W_n\leq z)\bigr)\leq \liminf_{n\to\infty}d_\K(W_n,Y).
\end{align}
Since $z\in\R$ was arbitrary, \eqref{kol2} and \eqref{kol1} imply (b).
\end{proof}

\begin{proof}[Proof of Theorem \ref{infdejong}]
Define the $\mathbb{F}$-martingale $(F_n)_{n\in\N}$, where 
\[F_n:=\E\bigl[F\,\big|\,\F_n\bigr]=\sum_{M\subseteq[n]: |M|=p} F_M,\quad n\in\N.\]
In particular, for $n\in\N$, $F_n\in L^4_\X$ is a degenerate $U$-statistic of order $p$ based on $X_1,\dotsc,X_n$ and $(F_n)_{n\in\N}$ converges to $F$ $\P$-a.s. and in $L^4(\P)$ by the $L^4$-convergence theorem for martingales. Further, for $n\in\N$, let $\sigma_n^2:=\Var(F_n)$ and note that $\lim_{n\to\infty}\sigma_n^2=1$ and, hence, we may w.l.o.g. assume that $\sigma_n^2>0$ and define $G_n:=\sigma_n^{-1} F_n$ for any $n\in\N$.  Then, we have that also $(G_n)_{n\in\N}$ converges in $L^4(\P)$ and, hence, in particular in $L^1(\P)$ and in distribution to $F$. Since $G_n$ is normalized, from \cite[Theorem 1.3]{DP17} and \cite[Theorem 2.1]{Doe23b}, for any $n\in\N$, we have the bounds
 \begin{align}
d_\W(G_n,Z)&\leq \Biggl(\sqrt{\frac{2}{\pi}}+\frac{4}{3}\Biggr)\sqrt{\babs{\E[G_n^4]-3}}+\sqrt{\kappa_p}\Biggl(\sqrt{\frac{2}{\pi}}+ \frac{2\sqrt{2}}{\sqrt{3}}\Biggr)\rho(G_n)\quad\text{and}  \label{djw}\\
d_\K(G_n,Z)&\leq  11.9\sqrt{\babs{\E[G_n^4]-3}} +\bigl(3.5+10.8\sqrt{\kappa_p}\bigr)\rho(G_n).\label{djk}
\end{align} 
Now, noting that, for each $k\in\N$, as $n\to\infty$,
\begin{align*}
\Inf_k(F_n)&=\sum_{M\subseteq[n]:k\in M}\E[F_M^2]\Big\uparrow \sum_{M\subseteq\N:k\in M}\E[F_M^2]=\Inf_k(F)
\end{align*}
we obtain that also $\rho(F_n^2)$ is increasing in $n$ and, thus, 
\begin{align}\label{dj4}
\lim_{n\to\infty}\rho^2(F_n)=\sup_{n\in\N}\sup_{k\in\N}\Inf_k(F_n)=\sup_{k\in\N}\sup_{n\in\N}\Inf_k(F_n)=\sup_{k\in\N}\Inf_k(F)=\rho^2(F).
\end{align}
Furthermore, as $\rho(G_n)^2=\sigma_n^{-2}\rho^2(F_n)$, $n\in\N$, from \eqref{dj4} we have that $\lim_{n\to\infty}\rho^2(G_n)=\rho^2(F)$. Thus, from Lemma \ref{dislemma}, \eqref{djw}, \eqref{djk} and the $L^4$-convergence of $(G_n)_{n\in\N}$ to $F$ we conclude that 
\begin{align*}
d_\W(F,Z)&\leq\liminf_{n\to\infty}d_\W(G_n,Z)\leq \Biggl(\sqrt{\frac{2}{\pi}}+\frac{4}{3}\Biggr)\sqrt{\babs{\E[F^4]-3}}+\sqrt{\kappa_p}\Biggl(\sqrt{\frac{2}{\pi}}+ 
\frac{2\sqrt{2}}{\sqrt{3}}\Biggr)\rho(F)
\end{align*} 
and 
\begin{align*}
d_\K(F,Z)&\leq\liminf_{n\to\infty}d_\K(G_n,Z)\leq 11.9\sqrt{\babs{\E[F^4]-3}} +\bigl(3.5+10.8\sqrt{\kappa_p}\bigr)\rho(F),
\end{align*} 
completing the proof of Theorem \ref{infdejong}.
\end{proof}

\begin{remark}
It is in fact possible to prove Theorem \ref{infdejong} by means of Theorem \ref{cdcbound} and other results and tools from the present paper. Since the resulting proof would have been longer, we have decided against it and leave the details of such an intrinsic argument to the interested reader. 
\end{remark}

\subsection{Proof of Theorem \ref{rstheo}}\label{rsproof}
We apply the bound \eqref{cowass} from Theorem \ref{cobound}. To fit the application into the framework of Sections \ref{mallop} and \ref{normapp}, we let $X_0:=N$ so that $\X:=(X_n)_{n\in\N_0}$ is our independent sequence. Note that, here, we deal with sequences indexed by $\N_0$ in place of $\N$ which, of course, causes no trouble. Letting $\sigma^2:=\Var(S)$ it is easy to see that $\sigma^2=\E[N]\E[X_1^2]$ and that $F=\sigma^{-1}(S-\E[S])=\sigma^{-1}S$. Moreover, we have $D_kF=\sigma^{-1}D_kS$ for all $k\in\N_0$. Writing $S_n:=\sum_{j=1}^nX_j$, $n\in\N_0$, for $k=0$ we have 
\begin{align*}
D_0S&=S-\E\bigl[S\,\big|\,\G_0\bigr]=S-\sum_{n=0}^\infty\P(N=n)S_n=\sum_{n=1}^\infty \bigl(\1_{\{N=n\}}-\P(N=n)\bigr)S_n
\end{align*}    
and, hence, by independence
\begin{align}\label{rs1}
\E\bigl[D_0S\,\big|\,\F_0\bigr]&=\E\bigl[D_0S\,\big|\,N\bigr]=\sum_{n=1}^\infty \bigl(\1_{\{N=n\}}-\P(N=n)\bigr)\E[S_n]=0.
\end{align}
Note that interchanging the infinite sum and the expectation here is justified since 
\begin{align*}
&\sum_{n=1}^\infty \E\babs{\bigl(\1_{\{N=n\}}-\P(N=n)\bigr)S_n}=\sum_{n=1}^\infty \E\babs{\1_{\{N=n\}}-\P(N=n)}\E\babs{S_n}\\
&\leq \E|X_1|\Bigl(\sum_{n=1}^\infty\bigl( n \E\bigl[\1_{\{N=n\}}\bigr] +n\P(N=n)\bigr)\Bigr)=2 \E|X_1|\E[N]<\infty.
\end{align*}
Further, for $k\geq1$ we have 
\begin{align}\label{rs2}
D_kS&=S-\E\bigl[S\,\big|\,\G_k\bigr]=\1_{\{N\geq k\}}\bigl(X_k-\E[X_k]\bigr)=\1_{\{N\geq k\}}X_k=\E\bigl[D_kS\,\big|\,\F_k\bigr]
\end{align}
as both $X_k$ and $N$ are $\F_k$-measurable. Therefore, \eqref{rs1}, \eqref{rs2} and Remark \ref{remcobound} (a) imply that 
\begin{align}\label{rs3}
&\E\Bbabs{1- \sum_{k=0}^\infty \bigl(D_k\E[F\,|\,\F_k]\bigr)\, D_kF}=\E\Bbabs{1- \sum_{k=1}^\infty \bigl(D_k\E[F\,|\,\F_k]\bigr)\, D_kF}\notag\\
&\leq\biggl(\Var\Bigl(\sum_{k=1}^\infty \bigl(D_k\E[F\,|\,\F_k]\bigr)\, D_kF\Bigr)\biggr)^{1/2}=\frac{1}{\sigma^2}\biggl(\Var\Bigl(\sum_{k=1}^\infty \1_{\{N\geq k\}}X_k^2\Bigr)\biggr)^{1/2}.
\end{align}
Now, by the variance decomposition formula, with $T:=\sum_{k=1}^\infty \1_{\{N\geq k\}}X_k^2$, we have 
\begin{align*}
\Var(T)&=\Var\bigl(\E[T|N]\bigr)+\E\bigl[\Var(T|N)\bigr]
\end{align*}
and, since
\begin{align*}
\E[T|N]&=\E[X_1^2]\sum_{k=1}^\infty \1_{\{N\geq k\}}=N\,\E[X_1^2],
\end{align*}
and, using conditional independence, 
\begin{align*}
\Var(T|N)&=\Var\Bigl(\sum_{k=1}^\infty \1_{\{N\geq k\}}X_k^2\,\Big|\,N\Bigr)=\sum_{k=1}^\infty \Var\Bigl(\1_{\{N\geq k\}}X_k^2\,\Big|\,N\Bigr)\\
&=\sum_{k=1}^\infty \1_{\{N\geq k\}}\Bigl(\E[X_1^4]-\E[X_1^2]^2\Bigr),
\end{align*}
we conclude that 
\begin{align}\label{rs4}
\Var(T)&=\Var\bigl(N\,\E[X_1^2]\bigr)+\Bigl(\E[X_1^4]-\E[X_1^2]^2\Bigr)\E\biggl[\sum_{k=1}^\infty \1_{\{N\geq k\}}\biggr]\notag\\
&=\E[X_1^2]^2\Var(N)+\Bigl(\E[X_1^4]-\E[X_1^2]^2\Bigr)\sum_{n=1}^\infty\P(N\geq n)\notag\\
&=\E[X_1^2]^2\Var(N)+\Bigl(\E[X_1^4]-\E[X_1^2]^2\Bigr)\E[N].
\end{align}
Hence, \eqref{rs3} and \eqref{rs4} together imply that 
\begin{align}\label{rs5}
&\E\Bbabs{1- \sum_{k=0}^\infty \bigl(D_k\E[F\,|\,\F_k]\bigr)\, D_kF}\notag\\
&\leq \frac{1}{\E[N]\E[X_1^2]}\biggl(\E[X_1^2]^2\Var(N)+\Bigl(\E[X_1^4]-\E[X_1^2]^2\Bigr)\E[N]\biggr)^{1/2}\notag\\
&\leq\frac{\sqrt{\Var(N)}}{\E[N]}+\biggl(\frac{\E[X_1^4]}{\E[X_1^2]^2}-1\biggr)^{1/2}\frac{1}{\sqrt{\E[N]}}.
\end{align} 
We now turn to bounding the second term from the bound \eqref{cowass}. From \eqref{rs1}, \eqref{rs2} and the independence of $N$ and the $X_k$ we have 
\begin{align}\label{rs6}
&\sum_{k=0}^\infty\E\Bigl[\babs{D_k\E[F\,|\,\F_k]}\bigl(D_kF\bigr)^2\Bigr]=\sum_{k=1}^\infty\E\Bigl[\babs{D_kF}^3\Bigr]
=\sigma^{-3}\sum_{k=1}^\infty\E\Bigl[\1_{\{N\geq k\}}\babs{X_k}^3\Bigr]\notag\\
&=\sigma^{-3}\E|X_1|^3\sum_{k=1}^\infty\P(N\geq k)=\frac{\E[N]\E|X_1|^3}{\E[N]^{3/2}\E[X_1^2]^{3/2}}=\frac{\E|X_1|^3}{\E[X_1^2]^{3/2} \sqrt{\E[N]}}.
\end{align} 
The claim now follows from \eqref{cowass}, \eqref{rs5} and \eqref{rs6}.

\subsection{Proof of Theorem \ref{monotheo}}\label{monoproof}
We first introduce some useful notation and do some preparations. As in the statement, let $X_1,\dotsc,X_n$ be independent and on $[c]$ uniformly distributed random variables. Thus, letting $X_j:=1$ for all $j\geq n+1$, we obtain an infinite independent sequence $\X=(X_j)_{j\in\N}$ so that the situation fits into the framework of Sections \ref{malliavin} and \ref{normapp}. Throughout this proof we write $a_{i,j}$ for $a_{i,j}(G)$ and define the kernels $\psi:[c]^2\rightarrow\R$ and $\rho:[c]^3\rightarrow\R$ by
\begin{align*}
\psi(x,y)&:=\1_{\{x=y\}}-\frac{1}{c}\quad\text{and}\\
\rho(x,y,z)&:=\1_{\{x=y=z\}}-\frac{1}{c}\Bigl(\1_{\{x=y\}}+\1_{\{x=z\}}+\1_{\{y=z\}}\Bigr)+\frac{2}{c^2},
\end{align*}
respectively. It is straightforward to check that both $\psi$ and $\rho$ are symmetric in their arguments and, moreover, completely degenerate (or canonical) with respect to the uniform distribution on $[c]$. Recall that the latter property means that
\begin{align}
\E\bigl[\psi(x,X_1)\bigr]&=\frac{1}{c}\sum_{s=1}^c\psi(x,s)=0,\quad x\in[c],\label{degpsi}\\
\E\bigl[\rho(x,y,X_1)\bigr]&=\frac{1}{c}\sum_{s=1}^c\rho(x,y,s)=0,\quad x,y\in[c].\label{degrho}
\end{align} 
It is straightforward to check that 
\begin{align*}
\E\bigl[T_2(G)\bigr]&=\frac{m}{c}\quad\text{and}\quad \sigma^2:=\Var\bigl(T_2(G)\bigr)=\frac{m}{c}\Bigl(1-\frac{1}{c}\Bigr)
\end{align*}
and, hence, we have that
\begin{align}\label{Fmono}
F&=\frac{T_2(G)-\E[T_2(G)]}{\sqrt{\Var(T_2(G))}}=\frac{1}{\sigma}\sum_{1\leq i<j\leq n} a_{i,j}\psi(X_i,X_j)
\end{align}
has the form of a weighted, degenerate $U$-statistic of order two with kernel $\psi$ and weights $a_{i,j}$, $1\leq i<j\leq n$. We borrow the following important convention from the proof of \cite[Theorem 4]{BFY}: Without loss of generality we may and will assume that the vertices in $V=[n]$ are ordered in such a way that their degrees are non-increasing, i.e. we assume that 
$\deg(1)\geq \deg(2)\geq\ldots\geq \deg(n)$. Otherwise we could simply relabel the vertices of $G$. We will frequently make use of the following lemma.

\begin{lemma}\label{le1}
With the above convention we have 
\begin{enumerate}[{\normalfont (i)}]
\item $\displaystyle \sum_{1\leq i<j<k\leq n}a_{i,k}a_{j,k}\leq \sqrt{2}m^{3/2}$,
\item $\displaystyle \sum_{1\leq i<j<k\leq n}a_{i,j}a_{j,k}\leq \sqrt{2}m^{3/2}$.
\end{enumerate}
\end{lemma}

\begin{proof}
Part (i) has actually been proved in the proof of \cite[Lemma 1]{BFY}. The proof of part (ii) is very similar. Since it is important for what follows, we give the argument: Using that $\deg(j)\leq\deg(i)$ if $i<j$, we have
\begin{align*}
&\sum_{1\leq i<j<k\leq n}a_{i,j}a_{j,k}=\sum_{1\leq i<j\leq n-1}a_{i,j}\sum_{k=j+1}^n a_{j,k}\leq \sum_{1\leq i<j\leq n-1}a_{i,j}\deg(j)\\
&=\sum_{1\leq i<j\leq n-1}a_{i,j}\bigl(\deg(j)\wedge\deg(i)\bigr)\leq \sum_{\{i,j\}\in E}\bigl(\deg(j)\wedge\deg(i)\bigr) \leq \sqrt{2} m^{3/2},
\end{align*}
where the final inequality is proved on page 37 of \cite{BHJ}. 
\end{proof}

\begin{lemma}\label{le2}
We have the following variance formulae:
\begin{enumerate}[{\normalfont (i)}]
\item $\displaystyle \Var\bigl(\psi(X_1,X_2)\bigr)=\E\bigl[\psi(X_1,X_2)^2\bigr]=\frac{1}{c}\Bigl(1-\frac{1}{c}\Bigr)$.
\item $\displaystyle \Var\bigl(\rho(X_1,X_2,X_3)\bigr)=\E\bigl[\rho(X_1,X_2,X_3)^2\bigr]=\frac{2}{c^4}-\frac{3}{c^3}+\frac{1}{c^2}$.
\end{enumerate}
\end{lemma}
\begin{proof}
The respective first identities hold thanks to the degeneracy of $\psi$ and $\rho$. The actual formulae follow from straightforward computations that we omit here.
\end{proof}

\begin{lemma}\label{le3}
\begin{enumerate}[{\normalfont (i)}]
\item The Hoeffding decomposition of $W:=\psi(X_1,X_2)\psi(X_1,X_3)$ is given by 
\[W=\frac{1}{c}\psi(X_2,X_3)+\rho(X_1,X_2,X_3).\]
\item The Hoeffding decomposition of $V:=\psi(X_1,X_2)^2$ is given by 
\[V=\frac{1}{c}\Bigl(1-\frac{1}{c}\Bigr)+\Bigl(1-\frac{2}{c}\Bigr)\psi(X_1,X_2).\]
\end{enumerate}
\end{lemma}

\begin{proof}
Both claims follow from the degeneracy of $\psi$ and $\rho$, respectively, and from the fact that the identities in (i) and (ii) actually hold true, which is elementary to verify.
\end{proof}

\begin{proof}[Proof of Theorem \ref{monotheo}]
We apply the bound \eqref{cowass} in Theorem \ref{cobound} as well as the bound \eqref{cobrem}.
We will first deal with the first term in the bound \eqref{cowass}. Note that, for $k\in[n]$, from \eqref{Fmono} and \eqref{degpsi} we have
\begin{align}\label{mon1}
D_kF&=F-\E[F|\G_k]=\frac{1}{\sigma}\sum_{1\leq i<j\leq n} a_{i,j}\psi(X_i,X_j)-\frac{1}{\sigma}\sum_{1\leq i<j\leq n:\,k\in\{i,j\}} a_{i,j}\psi(X_i,X_j)\notag\\
&=\frac{1}{\sigma}\sum_{i\in[n]:\,i\not=k}a_{i,k}\psi(X_i,X_k)
\end{align}
and, thus, again by \eqref{degpsi}
\begin{align}\label{mon2}
\E\bigl[D_kF\,\big|\F_k\bigr]&=\frac{1}{\sigma}\sum_{i=1}^{k-1}a_{i,k}\psi(X_i,X_k).
\end{align}
Hence, from \eqref{mon1},\eqref{mon2} and Lemma \ref{le3} we have that 
\begin{align}\label{hdm1}
&(D_kF)\E\bigl[D_kF\,\big|\F_k\bigr]=\frac{1}{\sigma^2}\sum_{i=1}^{k-1}a_{i,k}\psi(X_i,X_k)^2+\frac{2}{\sigma^2}\sum_{1\leq i<j<k-1}a_{i,k} a_{j,k}\psi(X_i,X_k)\psi(X_j,X_k)\notag\\
&\;+\frac{1}{\sigma^2}\sum_{i=1}^{k-1}\sum_{j=k+1}^n a_{i,k} a_{j,k}\psi(X_i,X_k)\psi(X_j,X_k)\notag\\
&=\frac{1}{\sigma^2}\sum_{i=1}^{k-1}a_{i,k}\biggl(\frac{1}{c}\Bigl(1-\frac{1}{c}\Bigr)+\Bigl(1-\frac{2}{c}\Bigr)\psi(X_i,X_k)\biggr)\notag\\
&\;+\frac{2}{\sigma^2}\sum_{1\leq i<j<k-1}a_{i,k} a_{j,k}\biggl(\frac{1}{c}\psi(X_i,X_j)+\rho(X_i,X_j,X_k)\biggr)\notag\\
&\;+\frac{1}{\sigma^2}\sum_{i=1}^{k-1}\sum_{j=k+1}^n a_{i,k} a_{j,k}\biggl(\frac{1}{c}\psi(X_i,X_j)+\rho(X_i,X_j,X_k)\biggr)\notag\\
&=\E\Bigl[(D_kF)\E\bigl[D_kF\,\big|\F_k\bigr]\Bigr]+\frac{c-2}{c\sigma^2}\sum_{i=1}^{k-1}a_{i,k}\psi(X_i,X_k)
+\frac{2}{c\sigma^2}\sum_{1\leq i<j<k-1}a_{i,k} a_{j,k}\psi(X_i,X_j)\notag\\
&\;+\frac{1}{c\sigma^2}\sum_{i=1}^{k-1}\sum_{j=k+1}^n a_{i,k} a_{j,k}\psi(X_i,X_j)+\frac{2}{\sigma^2}\sum_{1\leq i<j<k-1}a_{i,k} a_{j,k}\rho(X_i,X_j,X_k)\notag\\
&\;+\frac{1}{\sigma^2}\sum_{i=1}^{k-1}\sum_{j=k+1}^n a_{i,k} a_{j,k}\rho(X_i,X_j,X_k)
\end{align}
is the Hoeffding decomposition of $(D_kF)\E\bigl[D_kF\,\big|\F_k\bigr]$. Summing over $k=1,\dotsc,n$ and grouping terms we hence obtain the Hoeffding decomposition of 
\begin{align}\label{hdm2}
&S:=\sum_{k=1}^n(D_kF)\E\bigl[D_kF\,\big|\F_k\bigr]\notag\\
&=\E[S]+\frac{1}{c\sigma^2}\sum_{1\leq i<j\leq n}\biggl( (c-2) a_{i,j} +2\sum_{k=j+1}^n a_{i,k}a_{j,k}+\sum_{k=i+1}^{j-1}a_{i,k}a_{j,k} \biggr)\psi(X_i,X_j)\notag\\
&\;+\frac{1}{\sigma^2}\sum_{1\leq i<k<j\leq n}\Bigl(a_{i,k}a_{j,k}+2a_{i,j}a_{j,k}\1_{\{k< j-1\}}\Bigr)\rho(X_i,X_j,X_k).
\end{align}
Using the orthogonality of the Hoeffding components, Lemma \ref{le2}, $a_{i,j}^2=a_{i,j}$ as well as the inequalities $(a+b+c)^2\leq 3a^2+3b^2+3c^2$ and $(a+b)^2\leq 2a^2+2b^2$ we thus have that 
\begin{align}\label{mon3}
&\Var(S)=\frac{1}{c^2\sigma^4}\sum_{1\leq i<j\leq n}\biggl((c-2) a_{i,j} +2\sum_{k=j+1}^n a_{i,k}a_{j,k}+\sum_{k=i+1}^{j-1}a_{i,k}a_{j,k} \biggr)^2 \E\bigl[\psi(X_i,X_j)^2\bigr]\notag\\
&\;+\frac{1}{\sigma^4}\sum_{1\leq i<k<j\leq n}\Bigl(a_{i,k}a_{j,k}+2a_{i,j}a_{j,k}\1_{\{k< j-1\}}\Bigr)^2\E\bigl[\rho(X_i,X_j,X_k)^2\bigr]\notag\\
&\leq\frac{1}{m^2 (c-1) }\sum_{1\leq i<j\leq n}3\biggl( (c-2)^2a_{i,j} +4\sum_{k,l=j+1}^n a_{i,k}a_{j,k}a_{i,l}a_{j,l}+\sum_{k,l=i+1}^{j-1}a_{i,k}a_{j,k} a_{i,l}a_{j,l}\biggr)\notag\\
&\;+\frac{c^2-3c+2}{(c-1)^2m^2}\sum_{1\leq i<k<j\leq n}2\Bigl(a_{i,k}a_{j,k}+4a_{i,j}a_{j,k}\Bigr)\notag\\
&\leq \frac{3(c-2)}{m^2 }\sum_{1\leq i<j\leq n}a_{i,j}+\frac{12}{m^2 (c-1) }\sum_{1\leq i<j<k\leq n} a_{i,k}a_{j,k}+\frac{3}{m^2 (c-1) }\sum_{1\leq i<k<j\leq n}a_{i,k}a_{j,k}\notag\\
&\;+\frac{24}{m^2 (c-1) }\sum_{1\leq i<j<k<l\leq n}a_{i,k}a_{j,k}a_{i,l}a_{j,l} +\frac{6}{m^2 (c-1) }\sum_{1\leq i<k<l<j\leq n}a_{i,k}a_{j,k} a_{i,l}a_{j,l}\notag\\
&\;+\frac{2}{m^2}\sum_{1\leq i<k<j\leq n}\Bigl(a_{i,k}a_{j,k}+4a_{i,j}a_{j,k}\Bigr) .
\end{align}
Now, noting that 
\begin{align*}
N(C_4,G)&=\frac18\sum_{(i,j,k,l)\in[n]^4_{\not=}}a_{i,k}a_{k,j} a_{j,l}a_{l,i}\\
&\geq \sum_{1\leq i_1<i_2<i_3<i_4\leq n}a_{i_{\pi(1)},i_{\pi(2)}}a_{i_{\pi(2)},i_{\pi(3)}} a_{i_{\pi(3)},i_{\pi(4)}}a_{i_{\pi(4)},i_{\pi(1)}}
\end{align*}
for each permutation $\pi$ of $\{1,2,3,4\}$, that $\sum_{1\leq i<j\leq n}a_{i,j}=m$ and by Lemma \ref{le1}, we obtain from \eqref{mon3} that 
\begin{align}\label{mon4}
&\Var(S)\leq\frac{3(c-2)}{m }+\frac{12\sqrt{2}}{m^{1/2} (c-1) }+\frac{24N(C_4,G)}{m^2 (c-1) }+\frac{3\sqrt{2}}{m^{1/2} (c-1) }+\frac{6N(C_4,G)}{m^2 (c-1) } +\frac{10\sqrt{2}}{m^{1/2}  }\notag\\
&=\frac{3(c-2)}{m }+\frac{10\sqrt{2}}{m^{1/2}  }+\frac{15\sqrt{2}}{m^{1/2} (c-1) }+\frac{30N(C_4,G)}{m^2 (c-1) }.
\end{align}
Thus, for the first term in the bound \eqref{cowass} we have 
\begin{align}\label{ftm}
&\sqrt{\frac{2}{\pi}}\E\Bbabs{1- \sum_{k=1}^n \bigl(D_k\E[F\,|\,\F_k]\bigr)\, D_kF}\leq \sqrt{\frac{2}{\pi}}\sqrt{\Var(S)}\notag\\
&\leq\sqrt{\frac{2}{\pi}}\biggl(\frac{3(c-2)}{m }+\frac{10\sqrt{2}}{m^{1/2}  }+\frac{15\sqrt{2}}{m^{1/2} (c-1) }+\frac{30N(C_4,G)}{m^2 (c-1) }\biggr)^{1/2}.
\end{align}
To estimate the second term in \eqref{cowass}, we apply the bound \eqref{cobrem}. First note that, from \eqref{poincare2} we have that 
\begin{align}\label{mon5}
\sum_{k=1}^\infty\E\bigl[(D_kF)^2\bigr]\leq 2\Var(F)=2
\end{align}
since $F\in\mathcal{H}_2$ as $\psi$ is a canonical kernel of order 2. Next, from \eqref{hdm1}, using orthogonality, Lemma \ref{le2} and $a_{i,j}^2=a_{i,j}$, similarly as for \eqref{mon3}, for $k\in[n]$, we have 
\begin{align*}
&\Var\Bigl((D_kF)\E\bigl[D_kF\,\big|\F_k\bigr]\Bigr)\leq\frac{c-2}{m^2}\sum_{i=1}^{k-1}a_{i,k}+\frac{4}{m^2(c-1)}\sum_{1\leq i<j<k-1}a_{i,k} a_{j,k}\\
&\;+\frac{1}{m^2(c-1)}\sum_{i=1}^{k-1}\sum_{j=k+1}^n a_{i,k} a_{j,k}+\frac{4}{m^2}\sum_{1\leq i<j<k-1}a_{i,k} a_{j,k}+\frac{1}{m^2}\sum_{i=1}^{k-1}\sum_{j=k+1}^n a_{i,k} a_{j,k}
\end{align*}
so that, using Lemma \ref{le1} again, by summing over $k=1,\dotsc,n$, we obtain 
\begin{align}\label{mon6}
 &\sum_{k=1}^n\Var\Bigl((D_kF)\E\bigl[D_kF\,\big|\F_k\bigr]\Bigr)\notag\\
&\leq\frac{c-2}{m}+\frac{4\sqrt{2}}{m^{1/2}(c-1)}+\frac{\sqrt{2}}{m^{1/2}(c-1)}+\frac{4\sqrt{2}}{m^{1/2}}+\frac{\sqrt{2}}{m^{1/2}}\notag\\
&=\frac{c-2}{m}+\frac{5\sqrt{2}}{m^{1/2}(c-1)}+\frac{5\sqrt{2}}{m^{1/2}}.
\end{align}
Finally, from \eqref{mon2}, again using Lemma \ref{le1} we have that 
\begin{align}\label{mon7}
&\sum_{k=1}^n\Var\Bigl(\E\bigl[D_kF\,\big|\F_k\bigr]\Bigr)^2=\frac{(1-1/c)^2}{c^2\sigma^4}\sum_{k=1}^n\biggl(\sum_{i=1}^{k-1}a_{i,k}\biggr)^2\notag\\
&=\frac{c-1}{cm^2}\biggl(\sum_{1\leq i<k\leq n}a_{i,k}+2\sum_{1\leq i<j<k\leq n}a_{i,k}a_{j,k}\biggr)
\leq \frac{1}{m}+\frac{2\sqrt{2}}{m^{1/2}}.
\end{align}
Hence, from \eqref{cobrem}, \eqref{mon5}, \eqref{mon6} and \eqref{mon7} we obtain for the second term in the bound \eqref{cowass} that
\begin{align}\label{stm}
&\sum_{k=1}^n\E\Bigl[\babs{D_k\E[F\,|\,\F_k]}\bigl(D_kF\bigr)^2\Bigr]\leq\sqrt{2}\biggl(\frac{c-1}{m}+\frac{5\sqrt{2}}{m^{1/2}(c-1)}+\frac{7\sqrt{2}}{m^{1/2}}\biggr)^{1/2}
\end{align}
so that the claim follows from \eqref{cowass}, \eqref{ftm} and \eqref{stm}.
\end{proof}

\section{Appendix} \label{appendix}
In this section we give a self-contained proof of Proposition \ref{infhoeff} on infinite Hoeffding decompositions. We first review the concepts of unconditional convergence and summability in Banach spaces.

\subsection{Unconditional convergence and summability in Banach spaces } 
Let $(B,\norm{\cdot})$ be a normed space, let $\emptyset\not=I$ be an index set and let $(x_i)_{i\in I}$ be a family of vectors in $B$. 

The (symbolic) series $\sum_{i\in I} x_i$ is said to \textit{converge unconditionally} to the 
vector $x\in B$, if the following two conditions hold:
\begin{itemize}
 \item The set $I^*:=\{i\in I\,:\, x_i\not=0\}$ is at most countably infinite.
 \item If $\abs{I^*}=\infty$ and $I^*=\{i_n\, :\, n\in\N\}$ is an enumeration of $I^*$, then the series $\sum_{n=1}^\infty x_{i_n}$ converges to $x$, i.e.
 \[\lim_{m\to\infty}\Bigl\|\sum_{n=1}^m x_{i_n} -x\Bigr\|=0\,.\]
\end{itemize}
On the other hand, the family $(x_i)_{i\in I}$ is called \textit{summable}, if there is an $s\in B$ with the following property: For each $\epsilon>0$ there is a finite set $I_\eps\subseteq I$ such that 
for all finite $J$ with $I_\eps\subseteq J\subseteq I$ we have 
\[\Bigl\|\sum_{i\in J} x_i -s\Bigr\|<\eps\,.\]
It is not difficult to see that such an $s$ is necessarily unique. Moreover, it is a standard result in functional analysis that summability and unconditional convergence of series in Banach spaces are equivalent and that, with the above notation, one has $x=s$.

In finite-dimensional spaces, these conditions are further equivalent to absolute convergence of the series, i.e. to the condition that $\sum_{n\in\N}\norm{x_{i_n}}<\infty$. 

For infinite-dimensional Banach spaces $B$, however, according to the famous \textit{Dvoretzky-Rogers theorem} \cite{DvoRo}, there always exists a series $\sum_{i\in I} x_i$ that is unconditionally but not absolutely convergent.

The following result about the unconditional convergence of families of centered and uncorrelated random variables in $L^2(\P)$ is certainly known. As we have not found a suitable reference, though, we include its statement and also provide a complete proof.

\begin{prop}\label{serieslemma}
Let $\emptyset\not=I$ be 
an index set and suppose that $(X_i)_{i\in I}$ is a family of square-integrable, centered and pairwise uncorrelated random variables in $L^2(\P)$ for some probability space $(\Omega,\A,\P)$. Assume that the family $(\Var(X_i))_{i\in I}$ of nonnegative real numbers satisfies
\[\sup\Bigl\{\sum_{i\in J}\Var(X_i)\,:\,J\subseteq I\text{ finite }\Bigr\}<\infty. \]
Then, there is a random variable $S\in L^2(\P)$ such that the series $\sum_{i\in I} X_i$ converges unconditionally to $S$ in the $L^2(\P)$-sense.
\end{prop}

\begin{proof}
It is not difficult to see that the assumption implies that the set $I^*:=\{i\in I\,:\, \Var(X_i)\not=0\}$ is countable and we may assume without loss of generality that it is countably infinite, since in the finite case the result is immediate.
Let us first fix an enumeration $(i_n)_{n\in\N}$ of $I^*$. Then, we obviously have 
\[\sum_{n=1}^\infty \Var(X_{i_n})=\sup\Bigl\{\sum_{i\in J}\Var(X_i)\,:\,J\subseteq I\text{ finite }\Bigr\}<\infty.\]
Therefore, using orthogonality, for integers $1\leq m<n$ it follows that 
\begin{align*}
 \Bigl\|\sum_{k=m+1}^n X_{i_k}\Bigr\|_2^2&=\sum_{k,l=m+1}^n\E\bigl[X_{i_k}X_{i_l}\bigr]=\sum_{k=m+1}^n\Var\bigl(X_{i_k}\bigr)
 \longrightarrow0\,,\quad\text{as }n,m\to\infty.
\end{align*}
Hence, the sequence $(S_n)_{n\in\N}$ with $S_n=\sum_{k=1}^n X_{i_k}$ is Cauchy in $L^2(\P)$ and converges by completeness to some $S\in L^2(\P)$. 
It remains to make sure that this random variable $S$ does not depend on the particular enumeration chosen. Thus, let $\pi:\N\rightarrow\N$ be any bijection. Then, as above, one obtains that also the series 
$\sum_{k=1}^\infty X_{i_{\pi(k)}}$ converges to some $S_\pi\in L^2(\P)$. We show that $S=S_\pi$ $\P$-a.s. Let $\eps>0$ be given and choose $n_0=n_0(\eps)\in\N$ such that 
\[\sum_{k=n_0+1}^\infty \Var\bigl(X_{i_k}\bigr)<\eps^2\,.\]
Furthermore, choose $m_0=m_0(\eps)\in\N$ such that $[n_0]\subseteq\pi([m_0])$. Then, by the triangle inequality, we obtain 
\begin{align*}
 \Enorm{S-S_\pi}&\leq\BEnorm{S-\sum_{k=1}^{n_0}X_{i_k}}+\BEnorm{\sum_{k=1}^{n_0}X_{i_k}-\sum_{k=1}^{m_0}X_{i_{\pi(k)}}}+\BEnorm{\sum_{k=1}^{m_0}X_{i_{\pi(k)}}-S_\pi}\\
 &\leq \Bigl(\sum_{k=n_0+1}^\infty\Var\bigl(X_{i_k}\bigr)\Bigr)^{1/2}+\Bigl(\sum_{l\in\pi([m_0])\setminus[n_0]}\Var\bigl(X_{i_{l}}\bigr)\Bigr)^{1/2}\\
&\; +\Bigl(\sum_{l\in\N\setminus\pi([m_0])}\Var\bigl(X_{i_l}\bigr)\Bigr)^{1/2}\leq 3 \Bigl(\sum_{k=n_0+1}^\infty\Var\bigl(X_{i_k}\bigr)\Bigr)^{1/2}<3\epsilon\,.
\end{align*}
Since $\eps>0$ was arbitrary, this implies that $S=S_\pi$ $\P$-a.s.and $\sum_{i\in I} X_i$ converges unconditionally to $S$.\\
\end{proof}

\subsection{Proof of Proposition \ref{infhoeff} }
Recall that $F_M=F_{\max(M),M}=F_{n,M}$ for all finite $M\subseteq\N$ and all $n\geq\max(M)$. 
 We first prove \eqref{genhd2}. Due to \eqref{cons3}, for $n\in\N$ we have 
 \begin{align}\label{hd1}
  F_n&=\sum_{\substack{M\subseteq\N:\\ \max(M)\leq n}}F_{M}.
\end{align}
Hence, using orthogonality and \eqref{l2bound2} we obtain 
\begin{align}\label{hd2}
 \sum_{\substack{M\subseteq\N:\\ \max(M)\leq n}}\Var\bigl(F_{M}\bigr)&
 =\Var(F_n)\leq \Var(F)\,.
\end{align}
As the right hand side of the inequality \eqref{hd2} does not depend on $n$ we conclude that 
\begin{align}\label{hd3}
 \sum_{\substack{M\subseteq\N:\\ \abs{M}<\infty}}\Var\bigl(F_{M}\bigr)&=\lim_{n\to\infty}\sum_{\substack{M\subseteq\N:\\ \max(M)\leq n}}\Var\bigl(F_{M}\bigr)=\lim_{n\to\infty}\Var(F_n)=\Var(F)<+\infty\,,
\end{align}
where we have used the $L^2(\P)$ convergence of $(F_n)_{n\in\N}$ to $F$ to obtain the last identity. As the summands are pairwise orthogonal and $L^2(\P)$ is a Banach space, by Proposition \ref{serieslemma} this already implies that the first infinite series $\sum_{\substack{M\subseteq\N: \abs{M}<\infty}}F_{M} $ on the right hand side of \eqref{genhd2} converges unconditionally in the $L^2(\P)$-sense to some limit $\tilde{F}\in L^2_\X$. We need to make sure that $\tilde{F}=F$ holds $\P$-a.s. As $(F_n)_{n\in\N}$ converges to $F$ in $L^2(\P)$ this will follow if we can show that also $\tilde{F}$ is the $L^2(\P)$-limit of $(F_n)_{n\in\N}$. Let $(M_k)_{k\in\N}$ be an enumeration of the finite subsets of $\N$. Then, by unconditional convergence, we have 
\[\tilde{F}=\sum_{k=1}^\infty F_{M_k}\]
in the $L^2(\P)$-sense. Thus, given $\eps>0$ we may choose $k_0\in\N$ such that 
\[\Biggl(\sum_{k=k_0+1}^\infty\E\bigl[F_{M_k}^2\bigr]\Biggr)^{1/2}=\BEnorm{\sum_{k=k_0+1}^\infty F_{M_k}}=\BEnorm{\tilde{F}-\sum_{k=1}^{k_0} F_{M_k}} <\eps.\]
Then, for $n\geq n_0$, where $n_0\in\N$ is chosen in such a way that $M_1,\dotsc,M_{k_0}\subseteq[n_0]$ it holds that
\begin{align*}
 \Enorm{\tilde{F}-F_n}&=\Biggl(\sum_{\substack{k\in\N:\\ M_k\not\subseteq [n]}}\E\bigl[F_{M_k}^2\bigr]\Biggr)^{1/2}\leq \Biggl(\sum_{k=k_0+1}^\infty\E\bigl[F_{M_k}^2\bigr]\Biggr)^{1/2}
 <\eps,
\end{align*}
as desired. Thus, $F=\tilde{F}$ $\P$-a.s.
This proves the first identity in \eqref{genhd2}. By the same argument with obvious adaptations one shows that \eqref{genhd1} is true for each $p\in\N_0$. 
It remains to prove the second identity in \eqref{genhd2}. Since 
\begin{align*}
 &\sup_{m\in\N}\sum_{p=0}^m\Var\bigl(F^{(p)}\bigr)= \sup_{m\in\N}\lim_{n\to\infty}\sum_{p=0}^m\Var\bigl(F^{(p)}_n\bigr)=\sup_{m\in\N}\sup_{n\in\N}\sum_{p=0}^m\Var\bigl(F^{(p)}_n\bigr)\\
 &=\sup_{n\in\N}\sup_{m\in\N}\sum_{p=0}^m\Var\bigl(F^{(p)}_n\bigr)=\sup_{n\in\N}\sum_{p=0}^n\Var\bigl(F^{(p)}_n\bigr)=\sup_{n\in\N}\Var(F_n)=\Var(F)<\infty,
\end{align*}
it follows from \eqref{orthrel} and Proposition \ref{serieslemma} again that the series $\sum_{p=0}^\infty F^{(p)}$ converges unconditionally to some $T\in L^2_\X$ in the $L^2(\P)$-sense. Given $\eps>0$ we can hence find an $m_0\in\N$ such that for all $m\geq m_0$ 
\[\Biggl(\sum_{p=m+1}^\infty\E\Bigl[\bigl(F^{(p)} \bigl)^2\Bigr]\Biggr)^{1/2}=\BEnorm{\sum_{p=m+1}^\infty F^{(p)}}=\BEnorm{T-\sum_{p=0}^{m} F^{(p)}}<\eps.\]
Then, for given $\eps>0$ with $(M_k)_{k\in\N}$ and $k_0$ as above we choose $m_1\geq m_0$ large enough to ensure that $|M_{k}|\leq m_1$ for $k=1,\dotsc,k_0$. Then, we have 
\begin{align*}
 \Enorm{T-F}&\leq\BEnorm{T-\sum_{p=0}^{m_1} F^{(p)}}+\BEnorm{\sum_{p=0}^{m_1} F^{(p)}-\sum_{k=1}^{k_0}F_{M_k}}+\BEnorm{\sum_{k=k_0+1}^{\infty}F_{M_k}}\\
 &\leq \BEnorm{\sum_{p=m_1+1}^\infty F^{(p)}}+2\BEnorm{\sum_{k=k_0+1}^{\infty}F_{M_k}}<3\eps,
\end{align*}
where we have used the fact that the convergence in 
\[\sum_{p=0}^{m_1} F^{(p)}=\sum_{\substack{M\subseteq\N:\\ |M|\leq m_1}} F_M\]
is again unconditional, which follows from the unconditional convergence in \eqref{genhd1} and the fact that summability is obviously preserved under addition of finitely many series.
Thus, $T=F$ $\P$-a.s. which finishes the proof of Proposition \ref{infhoeff}.

\normalem
\bibliography{malliavin_product}{}

\begin{thebibliography}{LRPY22}

\bibitem[Alo81]{Alon}
N.~Alon.
\newblock On the number of subgraphs of prescribed type of graphs with a given
  number of edges.
\newblock {\em Israel J. Math.}, 38(1-2):116--130, 1981.

\bibitem[BDM17]{BDM}
B.~B. Bhattacharya, P.~Diaconis, and S.~Mukherjee.
\newblock Universal limit theorems in graph coloring problems with connections
  to extremal combinatorics.
\newblock {\em Ann. Appl. Probab.}, 27(1):337--394, 2017.

\bibitem[BFY22]{BFY}
B.~B. Bhattacharya, X.~Fang, and H.~Yan.
\newblock Normal approximation and fourth moment theorems for monochromatic
  triangles.
\newblock {\em Random Structures Algorithms}, 60(1):25--53, 2022.

\bibitem[BGL14]{BGL14}
D.~Bakry, I.~Gentil, and M.~Ledoux.
\newblock {\em Analysis and geometry of {M}arkov diffusion operators}, volume
  348 of {\em Grundlehren der Mathematischen Wissenschaften [Fundamental
  Principles of Mathematical Sciences]}.
\newblock Springer, Cham, 2014.

\bibitem[BH91]{BH}
N.~Bouleau and F.~Hirsch.
\newblock {\em Dirichlet forms and analysis on {W}iener space}, volume~14 of
  {\em De Gruyter Studies in Mathematics}.
\newblock Walter de Gruyter \& Co., Berlin, 1991.

\bibitem[BHJ92]{BHJ}
A.~D. Barbour, L.~Holst, and S.~Janson.
\newblock {\em Poisson approximation}, volume~2 of {\em Oxford Studies in
  Probability}.
\newblock The Clarendon Press, Oxford University Press, New York, 1992.
\newblock Oxford Science Publications.

\bibitem[BLM13]{BLM}
S.~Boucheron, G.~Lugosi, and P.~Massart.
\newblock {\em Concentration inequalities}.
\newblock Oxford University Press, Oxford, 2013.
\newblock A nonasymptotic theory of independence, With a foreword by Michel
  Ledoux.

\bibitem[BM17]{BM}
A.~Basak and S.~Mukherjee.
\newblock Universality of the mean-field for the {P}otts model.
\newblock {\em Probab. Theory Related Fields}, 168(3-4):557--600, 2017.

\bibitem[CGS11]{CGS}
L.~H.~Y. Chen, L.~Goldstein, and Q.-M. Shao.
\newblock {\em Normal approximation by {S}tein's method}.
\newblock Probability and its Applications (New York). Springer, Heidelberg,
  2011.

\bibitem[Cha08]{Cha08}
S.~Chatterjee.
\newblock A new method of normal approximation.
\newblock {\em Ann. Probab.}, 36(4):1584--1610, 2008.

\bibitem[CS07]{ChSh07}
L.~H.~Y. Chen and Q.-M. Shao.
\newblock Normal approximation for nonlinear statistics using a concentration
  inequality approach.
\newblock {\em Bernoulli}, 13(2):581--599, 2007.

\bibitem[DH19]{DH}
L.~Decreusefond and H.~Halconruy.
\newblock Malliavin and {D}irichlet structures for independent random
  variables.
\newblock {\em Stochastic Process. Appl.}, 129(8):2611--2653, 2019.

\bibitem[DHM23]{DHM}
S.~Das, Z.~Himwich, and N.~Mani.
\newblock A fourth-moment phenomenon for asymptotic normality of monochromatic
  subgraphs.
\newblock {\em Random Structures Algorithms}, 63(4):968--996, 2023.

\bibitem[dJ90]{deJo90}
P.~de~Jong.
\newblock A central limit theorem for generalized multilinear forms.
\newblock {\em J. Multivariate Anal.}, 34(2):275--289, 1990.

\bibitem[DK19]{DK19}
C.~D\"{o}bler and K.~Krokowski.
\newblock On the fourth moment condition for {R}ademacher chaos.
\newblock {\em Ann. Inst. Henri Poincar\'{e} Probab. Stat.}, 55(1):61--97,
  2019.

\bibitem[DKP22]{D19}
C.~D\"{o}bler, M.~Kasprzak, and G.~Peccati.
\newblock The multivariate functional de {J}ong {CLT}.
\newblock {\em Probab. Theory Related Fields}, 184(1-2):367--399, 2022.

\bibitem[D{\"o}b12]{Doe12rs}
C.~D{\"o}bler.
\newblock {On rates of convergence and Berry-Esseen bounds for random sums of
  centered random variables with finite third moments}.
\newblock arXiv:1212.5401, 2012.

\bibitem[D{\"o}b15]{D6}
C.~D{\"o}bler.
\newblock New {B}erry-{E}sseen and {W}asserstein bounds in the {CLT} for
  non-randomly centered random sums by probabilistic methods.
\newblock {\em ALEA Lat. Am. J. Probab. Math. Stat.}, 12(2):863--902, 2015.

\bibitem[D{\"o}b23a]{D23}
C.~D{\"o}bler.
\newblock Normal approximation via non-linear exchangeable pairs.
\newblock {\em ALEA Lat. Am. J. Probab. Math. Stat.}, 20(1):167--224, 2023.

\bibitem[D{\"o}b23b]{Doe23b}
C.~D{\"o}bler.
\newblock {The Berry-Esseen bound in de Jong's CLT}.
\newblock Preprint, arXiv:2307.03189, 2023.

\bibitem[DP17]{DP17}
C.~D\"obler and G.~Peccati.
\newblock {Quantitative de Jong theorems in any dimension}.
\newblock {\em Electron. J. Probab.}, 22:no. 2, 1--35, 2017.

\bibitem[DP18a]{DP18a}
C.~D\"{o}bler and G.~Peccati.
\newblock The fourth moment theorem on the {P}oisson space.
\newblock {\em Ann. Probab.}, 46(4):1878--1916, 2018.

\bibitem[DP18b]{DP18c}
C.~D\"{o}bler and G.~Peccati.
\newblock Fourth moment theorems on the {P}oisson space: analytic statements
  via product formulae.
\newblock {\em Electron. Commun. Probab.}, 23:Paper No. 91, 12, 2018.

\bibitem[DP19]{DP19}
C.~D\"obler and G.~Peccati.
\newblock {Quantitative CLTs for symmetric $U$-statistics using contractions}.
\newblock {\em Electron. J. Probab.}, 24:1--43, 2019.

\bibitem[DR50]{DvoRo}
A.~Dvoretzky and C.~A. Rogers.
\newblock Absolute and unconditional convergence in normed linear spaces.
\newblock {\em Proc. Nat. Acad. Sci. U. S. A.}, 36:192--197, 1950.

\bibitem[Due21]{Duer}
M.~Duerinckx.
\newblock On the size of chaos via {G}lauber calculus in the classical
  mean-field dynamics.
\newblock {\em Comm. Math. Phys.}, 382(1):613--653, 2021.

\bibitem[DVZ18]{DVZ18}
C.~D\"{o}bler, A.~Vidotto, and G.~Zheng.
\newblock Fourth moment theorems on the {P}oisson space in any dimension.
\newblock {\em Electron. J. Probab.}, 23:Paper No. 36, 27, 2018.

\bibitem[Eng83]{Eng83}
G.~Englund.
\newblock A remainder term estimate in a random-sum central limit theorem.
\newblock {\em Teor. Veroyatnost. i Primenen.}, 28(1):143--149, 1983.

\bibitem[ERTZ23]{ERTZ}
P.~Eichelsbacher, B.~Redno\ss, C.~Th\"{a}le, and G.~Zheng.
\newblock A simplified second-order {G}aussian {P}oincar\'{e} inequality in
  discrete setting with applications.
\newblock {\em Ann. Inst. Henri Poincar\'{e} Probab. Stat.}, 59(1):271--302,
  2023.

\bibitem[ES81]{EfSt}
B.~Efron and C.~Stein.
\newblock The jackknife estimate of variance.
\newblock {\em Ann. Statist.}, 9(3):586--596, 1981.

\bibitem[ET14]{ET14}
P.~Eichelsbacher and C.~Th{\"a}le.
\newblock New {B}erry-{E}sseen bounds for non-linear functionals of {P}oisson
  random measures.
\newblock {\em Electron. J. Probab.}, 19:no. 102, 25, 2014.

\bibitem[Fan15]{Fang}
X.~Fang.
\newblock A universal error bound in the {CLT} for counting monochromatic edges
  in uniformly colored graphs.
\newblock {\em Electron. Commun. Probab.}, 20:no. 21, 6, 2015.

\bibitem[FR79]{FrRa}
J.~H. Friedman and L.~C. Rafsky.
\newblock Multivariate generalizations of the {W}ald-{W}olfowitz and {S}mirnov
  two-sample tests.
\newblock {\em Ann. Statist.}, 7(4):697--717, 1979.

\bibitem[Fri89]{Fr}
K~.O. Friedrich.
\newblock A {B}erry-{E}sseen bound for functions of independent random
  variables.
\newblock {\em Ann. Statist.}, 17(1):170--183, 1989.

\bibitem[GK96]{GneKo}
B.~V. Gnedenko and V.~Yu. Korolev.
\newblock {\em Random summation}.
\newblock CRC Press, Boca Raton, FL, 1996.
\newblock Limit theorems and applications.

\bibitem[Hae88]{Haeu88}
E.~Haeusler.
\newblock On the rate of convergence in the central limit theorem for
  martingales with discrete and continuous time.
\newblock {\em Ann. Probab.}, 16(1):275--299, 1988.

\bibitem[HB70]{HB70}
C.~C. Heyde and B.~M. Brown.
\newblock On the departure from normality of a certain class of martingales.
\newblock {\em Ann. Math. Statist.}, 41:2161--2165, 1970.

\bibitem[KB94]{KB-book}
V.~S. Koroljuk and Yu.~V. Borovskich.
\newblock {\em Theory of {$U$}-statistics}, volume 273 of {\em Mathematics and
  its Applications}.
\newblock Kluwer Academic Publishers Group, Dordrecht, 1994.
\newblock Translated from the 1989 Russian original by P. V. Malyshev and D. V.
  Malyshev and revised by the authors.

\bibitem[Kor88]{Kor88}
V.Y. Korolev.
\newblock Accuracy of the normal approximation for distributions of sums of a
  random number of independent random variables.
\newblock {\em Teor. Veroyatnost. i Primenen.}, 33(3):577--581, 1988.
\newblock \href{http://www.ams.org/mathscinet-getitem?mr=MR968403}{MR968403}.

\bibitem[KR82]{KR}
S.~Karlin and Y.~Rinott.
\newblock Applications of {ANOVA} type decompositions for comparisons of
  conditional variance statistics including jackknife estimates.
\newblock {\em Ann. Statist.}, 10(2):485--501, 1982.

\bibitem[KRT16]{KRT1}
K.~Krokowski, A.~Reichenbachs, and C.~Th\"ale.
\newblock Berry-{E}sseen bounds and multivariate limit theorems for functionals
  of {R}ademacher sequences.
\newblock {\em Ann. Inst. Henri Poincar\'e Probab. Stat.}, 52(2):763--803,
  2016.

\bibitem[KRT17]{KRT2}
K.~Krokowski, A.~Reichenbachs, and C.~Th\"ale.
\newblock Discrete {M}alliavin--{S}tein method: {B}erry--{E}sseen bounds for
  random graphs and percolation.
\newblock {\em Ann. Probab.}, 45(2):1071--1109, 2017.

\bibitem[Kwa10]{kwapien}
S.~Kwapie\'{n}.
\newblock On {H}oeffding decomposition in {$L^p$}.
\newblock {\em Illinois J. Math.}, 54(3):1205--1211, 2010.

\bibitem[Las16]{Lastsa}
G.~Last.
\newblock Stochastic analysis for {P}oisson processes.
\newblock In {\em Stochastic analysis for {P}oisson point processes}, volume~7
  of {\em Bocconi Springer Ser.}, pages 1--36. Bocconi Univ. Press, [place of
  publication not identified], 2016.

\bibitem[Lee90]{Lee}
A.~J. Lee.
\newblock {\em {$U$}-statistics}, volume 110 of {\em Statistics: Textbooks and
  Monographs}.
\newblock Marcel Dekker, Inc., New York, 1990.
\newblock Theory and practice.

\bibitem[LRP17]{LRP3}
R.~Lachi\`eze-Rey and G.~Peccati.
\newblock New {B}erry-{E}sseen bounds for functionals of binomial point
  processes.
\newblock {\em Ann. Appl. Probab.}, 27(4):1992--2031, 2017.

\bibitem[LRPY22]{LRPY}
R.~Lachi\`eze-Rey, G.~Peccati, and X.~Yang.
\newblock Quantitative two-scale stabilization on the {P}oisson space.
\newblock {\em Ann. Appl. Probab.}, 32(4):3085--3145, 2022.

\bibitem[Maj13]{major}
P.~Major.
\newblock {\em On the estimation of multiple random integrals and
  {$U$}-statistics}, volume 2079 of {\em Lecture Notes in Mathematics}.
\newblock Springer, Heidelberg, 2013.

\bibitem[MOO10]{MOO10}
E.~Mossel, R.~O'Donnell, and K.~Oleszkiewicz.
\newblock Noise stability of functions with low influences: invariance and
  optimality.
\newblock {\em Ann. of Math. (2)}, 171(1):295--341, 2010.

\bibitem[Nou]{Nweb}
I.~Nourdin.
\newblock {Malliavin-Stein approach}.
\newblock \url{https://sites.google.com/site/malliavinstein/home}.

\bibitem[NP05]{NuaPec}
D.~Nualart and G.~Peccati.
\newblock Central limit theorems for sequences of multiple stochastic
  integrals.
\newblock {\em Ann. Probab.}, 33(1):177--193, 2005.

\bibitem[NP09]{NouPec09}
I.~Nourdin and G.~Peccati.
\newblock Stein's method on {W}iener chaos.
\newblock {\em Probab. Theory Related Fields}, 145(1-2):75--118, 2009.

\bibitem[NP12]{NouPecbook}
I.~Nourdin and G.~Peccati.
\newblock {\em Normal approximations with {M}alliavin calculus}, volume 192 of
  {\em Cambridge Tracts in Mathematics}.
\newblock Cambridge University Press, Cambridge, 2012.
\newblock From Stein's method to universality.

\bibitem[NPR10a]{NPR2}
I.~Nourdin, G.~Peccati, and G.~Reinert.
\newblock Invariance principles for homogeneous sums: universality of
  {G}aussian {W}iener chaos.
\newblock {\em Ann. Probab.}, 38(5):1947--1985, 2010.

\bibitem[NPR10b]{NPR}
I.~Nourdin, G.~Peccati, and G.~Reinert.
\newblock Stein's method and stochastic analysis of {R}ademacher functionals.
\newblock {\em Electron. J. Probab.}, 15:paper no. 55, 1703--1742, 2010.

\bibitem[Nua06]{Nualart}
D.~Nualart.
\newblock {\em The {M}alliavin calculus and related topics}.
\newblock Probability and its Applications (New York). Springer-Verlag, Berlin,
  second edition, 2006.

\bibitem[PR16]{RePe-book}
G.~Peccati and M.~Reitzner, editors.
\newblock {\em Stochastic Analysis for Poisson Point Processes}.
\newblock Springer-Verlag, 2016.

\bibitem[Pri08]{Priv08}
N.~Privault.
\newblock Stochastic analysis of {B}ernoulli processes.
\newblock {\em Probab. Surv.}, 5:435--483, 2008.

\bibitem[PS18]{PS1}
N.~Privault and G.~Serafin.
\newblock Stein approximation for functionals of independent random sequences.
\newblock {\em Electron. J. Probab.}, 23:Paper No. 4, 34, 2018.

\bibitem[PS22]{PS2}
N.~Privault and G.~Serafin.
\newblock Berry-{E}sseen bounds for functionals of independent random
  variables.
\newblock {\em Electron. J. Probab.}, 27:Paper No. 71, 37, 2022.

\bibitem[PSTU10]{PSTU}
G.~Peccati, J.~L. Sol{{\'e}}, M.~S. Taqqu, and F.~Utzet.
\newblock Stein's method and normal approximation of {P}oisson functionals.
\newblock {\em Ann. Probab.}, 38(2):443--478, 2010.

\bibitem[PT13]{PrTo2}
N.~Privault and G.~L. Torrisi.
\newblock Probability approximation by {C}lark-{O}cone covariance
  representation.
\newblock {\em Electron. J. Probab.}, 18:no. 91, 25, 2013.

\bibitem[PT15]{PrTo}
N.~Privault and G.~L. Torrisi.
\newblock The {S}tein and {C}hen-{S}tein methods for functionals of
  non-symmetric {B}ernoulli processes.
\newblock {\em ALEA Lat. Am. J. Probab. Math. Stat.}, 12(1):309--356, 2015.

\bibitem[Rob48]{Rob48}
H.~Robbins.
\newblock The asymptotic distribution of the sum of a random number of random
  variables.
\newblock {\em Bull. Amer. Math. Soc.}, 54:1151--1161, 1948.
\newblock \href{http://www.ams.org/mathscinet-getitem?mr=MR0027974}{MR0027974}.

\bibitem[Sch16a]{Schu16}
M.~Schulte.
\newblock Normal approximation of {P}oisson functionals in {K}olmogorov
  distance.
\newblock {\em J. Theoret. Probab.}, 29(1):96--117, 2016.

\bibitem[Sch16b]{Sch16}
M.~Schulte.
\newblock Normal approximation of {P}oisson functionals in {K}olmogorov
  distance.
\newblock {\em J. Theoret. Probab.}, 29(1):96--117, 2016.

\bibitem[Ser80]{serfling}
R.~J. Serfling.
\newblock {\em Approximation theorems of mathematical statistics}.
\newblock John Wiley \& Sons, Inc., New York, 1980.
\newblock Wiley Series in Probability and Mathematical Statistics.

\bibitem[Ste86a]{Steele}
J.~M. Steele.
\newblock An {E}fron-{S}tein inequality for nonsymmetric statistics.
\newblock {\em Ann. Statist.}, 14(2):753--758, 1986.

\bibitem[Ste86b]{St86}
C.~Stein.
\newblock {\em Approximate computation of expectations}.
\newblock Institute of Mathematical Statistics Lecture Notes---Monograph
  Series, 7. Institute of Mathematical Statistics, Hayward, CA, 1986.

\bibitem[SZ19]{ShaZh}
Q.-M. Shao and Z.-S. Zhang.
\newblock Berry-{E}sseen bounds of normal and nonnormal approximation for
  unbounded exchangeable pairs.
\newblock {\em Ann. Probab.}, 47(1):61--108, 2019.

\bibitem[Vit92]{vitale}
R.~A. Vitale.
\newblock Covariances of symmetric statistics.
\newblock {\em J. Multivariate Anal.}, 41(1):14--26, 1992.

\bibitem[vZ84]{vanZwet}
W.~R. van Zwet.
\newblock A {B}erry-{E}sseen bound for symmetric statistics.
\newblock {\em Z. Wahrsch. Verw. Gebiete}, 66(3):425--440, 1984.

\bibitem[Zhe17]{Zheng}
G.~Zheng.
\newblock Normal approximation and almost sure central limit theorem for
  non-symmetric {R}ademacher functionals.
\newblock {\em Stochastic Process. Appl.}, 127(5):1622--1636, 2017.

\end{thebibliography}
\bibliographystyle{alpha}
\end{document}